\documentclass[final]{amsart}
\usepackage[everylevel = section]{richard}                
\usepackage[parfill]{parskip}    
\usepackage{enumitem}
\usepackage{dsfont}

\setlist[enumerate]{before=\setlength{\baselineskip}{20pt}, itemsep=0pt}
\setlist[itemize]{before=\setlength{\baselineskip}{20pt}, itemsep=0pt}

\newcommand{\F}{\mathcal F}
\newcommand{\J}{\mathcal J}

\newcommand{\K}{\mathbb K}
\newcommand{\hk}{\mathbb{\hat K}}
\renewcommand{\k}{k}
\newcommand{\hv}{\mathbb\C_v}
\newcommand{\hp}{\mathbb\C_p}
\newcommand{\bH}{\mathbb{H}}
\newcommand{\CD}{\overline D}

\newcommand{\an}{\text{an}}

\newcommand{\q}{q}
\newcommand{\bvec}[1]{{\bf#1}}
\newcommand{\emp}{\emptyset}

\newcommand{\nin}{\notin}

\newcommand{\equivar}{equivariant{}}
\renewcommand{\and}{\text{ and }}
\newcommand{\dashto}{\dashrightarrow}

\newcommand{\Dir}[1]{T_{#1}\P^1_\an}

\mynewtheorem{propdefn}{Proposition/Definition}{Prop/Defn}

\newcommand{\weaklyindifferent}{numerically indifferent}
\newcommand{\weaklyrepelling}{numerically repelling}
\newcommand{\weaklyattracting}{numerically attracting}
\newcommand{\weakly}{numerically}

\newcommand{\stronglyindifferent}{indifferent}
\newcommand{\stronglyrepelling}{repelling}
\newcommand{\stronglyattracting}{attracting}

\DeclareMathOperator{\PGL}{PGL}
\DeclareMathOperator{\interior}{int}
\DeclareMathOperator{\rdeg}{rdeg}
\DeclareMathOperator{\wdeg}{wdeg}
\DeclareMathOperator{\Ram}{Ram}
\DeclareMathOperator{\Crit}{Crit}
\DeclareMathOperator{\Hull}{Hull}
\DeclareMathOperator{\Inj}{Inj}

\DeclareMathOperator{\diam}{diam}

\DeclareMathOperator{\Frac}{Frac}
\DeclareMathOperator{\dom}{dom}

\DeclareMathOperator{\GO}{GO}
\DeclareMathOperator{\chara}{char}

\makeatletter
\newcommand{\abs}[2][\@nil]{%
  \def\tmp{#1}%
   \ifx\tmp\@nnil
       \left|#2\right|
    \else
         \left|#2\right|_{#1}
    \fi}
\makeatother

\newtheorem*{thm*}{Theorem}

\begin{document}
  \title{Skew Products on the Berkovich Projective Line}
 \author{\mylongname}
 
\begin{abstract}
 In this article, we develop a dynamical theory for what shall be called a \emph{skew product} on the Berkovich projective line, $\phi_*: \P^1_\an(K) \to \P^1_\an(K)$ over a non-Archimedean field $K$. These functions are defined algebraically yet strictly generalise the notion of a \emph{rational map} on $\P^1_\an$. 
We describe the analytical, algebraic, and dynamical properties of skew products, including a study of periodic points, and a Fatou/Julia dichotomy. The article culminates with the classification of the connected components of the Fatou set.

%

\end{abstract}
 
 \maketitle

 \tableofcontents
 
 \clearpage


 \section{Introduction}

\makeatletter
\let\the@thm\thethm
\let\the@prop\theprop
\let\the@cor\thecor
\let\the@lem\thelem
\makeatother

\renewcommand*{\thethm}{\Alph{thm}}
\renewcommand*{\theprop}{\Alph{prop}}
\renewcommand*{\thecor}{\Alph{cor}}
\renewcommand*{\thelem}{\Alph{lem}}

A rational map $\phi : \P^1_\an(K) \to \P^1_\an(K)$ on the Berkovich projective line can be thought of as a $K$-algebra endomorphism $\phi^* : K(z) \to K(z)$. This article is dedicated to developing an analytical and dynamical theory for a map called a \emph{skew product}. In this more general context, we relax the condition that $\phi^*|_K$ is the identity, but ask that that it respects the non-Archimedean metric.

After unraveling the algebraic structure of a skew product, we will see that the map is still piecewise linear on $\P^1_\an$, however the slopes are not necessarily integers, nor at least $1$. The possibility of contraction leads to more interesting dynamics, especially for fixed points and the Fatou-Julia dichotomy. In the `simple' case (where the slopes are integers) we generalise the Rivera-Letelier style classification of Fatou components, and certain facts about wandering domains. These latter results will be fundamental for applications to dynamics in two dimensions.

The results in the present article and the forthcoming applications feature in the author’s PhD thesis.

\subsection{Motivation}
The dynamical theory of a skew product will be essential for applications to the dynamics of rational maps on a complex surface in a later article. 
For the purposes of understanding the dynamical degrees (algebraic entropy) and algebraic stability of a rational map $f : X \dashto X$ on a surface, one often needs to understand the potential dynamical behaviour of $f$ for all possible birational changes of coordinates. 
Favre and Jonsson \cite{FJ11} considered the case of a polynomial map $f : \P^2_\C \dashto \P^2_\C$ which is invariant over the line at infinity. They translate the dynamics of $f$ to that of a universal `dual graph' of all possible exceptional curves over this line. The `valuative tree' \cite{FJ04} they obtain is equivalent to a one dimensional Berkovich space over the Puiseux series, and the corresponding function turns out to essentially be a contraction mapping. 
Typically, such an induced function will not be a contraction mapping. The later application of the present article will be to the case of a skew product $\phi : X \dashto X$ on a complex ruled surface, and hence lends its name to the non-Archimedean version. Classically, a skew product is a map of the form $\phi(x, y) = (\phi_1(x), \phi_2(x, y))$ on a product space. When $X = \P^1 \times \P^1$, the rational map $\phi$ corresponds to a $\C$-algebra endomorphism $\phi^* : \C(x)(y) \to \C(x)(y)$. Furthermore if $\phi_1 = \id$ then this is a $\C(x)$-algebra endomorphism, and so it extends to a rational map $\phi_2 : \P^1_\an(\K) \to \P^1_\an(\K)$ over the field $\K$ of complex Puiseux series in $x$. DeMarco and Faber \cite{DeF1, DeF2} used this perspective when $\phi_1 = \id$ to find a (possibly singular) surface $\hat X$ where $\phi$ is algebraically stable. The core of their argument was to use the Rivera-Letelier classification of Fatou components \cite{RL2, RL1} for the corresponding Berkovich rational map. Inspired by this result, we will deal with the case of complex skew products where $\phi_1$ is not necessarily trivial. In this more general setting, $\phi^* : \C(x)(y) \to \C(x)(y)$ will not be a $\C(x)$-algebra endomorphism; whence the induced function $\phi_* : \P^1_\an(\K) \to \P^1_\an(\K)$ cannot correspond to a rational map, but to a different kind of map - the non-Archimedean \emph{skew product}. 

 \subsection{Skew Products on the Berkovich Projective Line}

Aside from examples, this theory will be discussed for a general non-Archimedean field $K$; further study of the specialisation to the Puiseux series $K = \K$ and applications to complex/algebraic dynamics will be deferred to a sequel. Skew products on the Berkovich projective line are of independent interest as a generalisation of Berkovich \emph{rational maps}, for instance because they also encompass a class of field automorphisms.

\improvement{Mention Vladimir Berkovich somewhere.}
Emerging out of the study of the dynamics of rational maps over the p-adic numbers there has been substantial interest in the notions of non-Archimedean Fatou and Julia sets \cite{Drem, Hsia1, Hsia2, Bez1, Bez2}. 
 The landmark theses of Robert L. Benedetto \cite{BeneThesis, Bene0, Bene2, Bene3} and Juan Rivera-Letelier \cite{RL2} together established the essential theory for the dynamics of non-Archimedean Fatou components. Rivera-Letelier, in particular, brought the insight of extending a rational map from $\P^1(K)$ to a \emph{rational map} on a non-Archimedean analytic space $\P^1_\an(K)$ which compactifies the former. He gave the classification of Fatou components over the $p$-adics. Following this there were numerous papers on the structure of the Fatou and Julia sets \cite{Bene4, Bene5, Bene1, Bene6, BBP, RL1, RL3, RL4}\improvement{Add Faber?}. The field exploded and many fruitful connections appeared between complex and non-Archimedean dynamics. For example by Kiwi \cite{Kiwi1, Kiwi2, Kiwi3}, by Baker and DeMarco \cite{BDe11, BDe13}, by Favre and Gauthier \cite{FG}, by Dujardin and Favre \cite{DuF}, by DeMarco and Faber \cite{DeF1, DeF2}, and by Favre \cite{Fav20} -- to name just a selection. Following this progress Fatou and Julia theory, three independent groups of mathematicians developed potential theory and proved equidistribution theorems in the non-Archimedean setting, mimicking the complex dynamical case; namely by Chambert-Loir and Thuillier \cite{CL, Thuillier}, by Baker and Rumely \cite{BR1}, and by Favre and Rivera-Letelier \cite{FRL1, FRL2, FRL3}.
 For a fuller account of this history we refer the reader to the excellent survey by Benedetto \cite{BeneSurvey}. For theoretical background the author recommends the books by Benedetto, and also Baker and Rumely \cite{Bene, BR2}.
 
 In particular, however, the reader is advised to compare with \cite{Bene} as they read \autoref{sec:nonarch}. In this work we will use the notation built up by Benedetto, and follow his development when possible for building up the theory of non-Archimedean skew products. The author hopes this will allow for an easy adjustment to the reader already familiar with Berkovich rational maps. \unsure{Move to section 2?}
 
 
 For this introduction, we briefly recall the structure of the Berkovich projective line $\P^1_\an(K)$. Unless otherwise stated, we consider $(K, | \cdot |)$ to be an arbitrary non-Archimedean algebraically closed field, and $\hv$ its completion. Algebraically, $\P^1_\an$ is the set of seminorms on the ring $K[y]$ which extend the norm $\abs{\, \cdot\, }$ on $K$, and also a point at infinity. Topologically, it is uniquely path connected, meaning it is a tree, however not in the finite combinatorial sense because it has a dense set of vertices on any interval. These (interior) vertices, called \emph{Type II} points, have one branch for every element of $\P^1(\k)$, where $\k$ is the residue field of $K$. All other points on an open interval are \emph{Type III} points (edge points). Every Type II and III point $\zeta = \zeta(a, r)$ corresponds to a closed disk $\CD(a, r) \subset \P^1(K)$, with $r$ in the value group $\abs{K^\times}$ or not respectively. The points in the classical projective line $\P^1(K) \subset \P^1_\an(K)$ form a set of endpoints called \emph{Type I} points, which alone would be a totally disconnected and non-compact set. There are other endpoints called \emph{Type IV} points, corresponding empty intersections of nested disks. All of these naturally correspond to seminorms through the geometric data and Berkovich showed \cite{Berk} that the four types listed constitute the whole space. The \emph{hyperbolic plane} $\bH = \P^1_\an(K) \sm \P^1(K)$ is defined as the set of Type II, III, and IV points; it is endowed with a \emph{hyperbolic metric} $d_\bH$. See \autoref{sec:berk} for more.

Much like a rational map, the starting point for a skew product is a homomorphism $\phi^* : K(y) \to K(y)$. 
To define a Berkovich \emph{rational map} we would equivalently require that $\phi^*$ is a $K$-algebra homomorphism, i.e.\ the identity on $K$. Unfortunately this is too restrictive for the application to complex skew products $\phi(x, y) = (\phi_1(x), \phi_2(x, y))$ because $\left.\phi^*\right|_\K : x \mapsto \phi_1(x)$ is not the identity map on $K = \K$. 
On the other hand, if we allow the homomorphism $\phi^*$ to be completely arbitrary, the induced map on the Berkovich projective line becomes unwieldy.
 Instead we impose on $\phi^*$ the condition that that $\phi^*$ uniformly scales the norm on $K$. This is easily satisfied in the geometric/complex case, for example \[\abs{\phi^*(x)} = \abs{\phi_1(x)} = \abs{c_nx^n + c_{n+1}x^{n+1} + \cdots} = \abs x^n.\] In general, if $\abs{\phi^*(a)} = \abs{a}^{1/\q}$ for every $a \in K$ then we call $\phi^*$ an \emph{\equivar{} skew endomorphism}, and define the \emph{skew product} $\phi_*$ by
\begin{align*}
 \phi_* : \P^1_\an(K) &\longrightarrow \P^1_\an(K) \\
 \zeta &\longmapsto \phi_*(\zeta)\\
 \text{where }\norm[\phi_*(\zeta)]{\psi} &= \norm[\zeta]{\phi^*(\psi)}^\q
\end{align*}
We call $\q$ the \emph{scale factor}, and for general $K$ this can be any positive real number; in the geometric/complex case, $\q = 1/n$.


  \subsection{Properties of Skew Products}

Fundamental to understanding the structure of skew products is the decomposition result of \autoref{thm:skew:comp}. We define $\phi_1^* = \left.\phi^*\right|_K$ but extended trivially to $K(y)$, and secondly we define $\phi_2^*$ to capture only the action of $\phi^*$ on $y$. However, perhaps unintuitively, $\phi_{1*}$ acts as $(\phi_1^*)^{-1}$ on classical points. Every skew endomorphism has this decomposition $\phi^* = \phi_2^* \circ \phi_1^*$, and it descends to the skew product $\phi_* = \phi_{1*} \circ \phi_{2*}$. Most important are the facts that $\phi_{2*}$ is a non-Archimedean rational map on $\P^1_\an$ and $\phi_{1*}$ uniformly scales the metric in the hyperbolic plane $\bH \subset \P^1_\an$ by a factor of $\q$, see \autoref{thm:skew:scaledhomeo}. The term \emph{scale factor} for $\q$ was chosen with this in mind.

In \autoref{thm:skew:opencts} and preceding results we demonstrate that any non-constant skew product $\phi_*$ is a continuous, open mapping that is precisely the unique continuous extension of $(\phi_1^*)^{-1} \circ \phi_2$ on $\P^1(K) \subset \P^1_\an(K)$. As one would hope, it preserves the types of each Berkovich point. Furthermore, \autoref{thm:skew:affinoidmapping} states that a skew product maps connected affinoids to connected affinoids, and by \autoref{prop:skew:dto1} it will map each component of $\phi_*^{-1}(U)$ to $U$ in a $d$-to-$1$ manner.

We then proceed to consider the map on tangent directions $\phi_\# : \Dir\zeta \to \Dir{\phi_*(\zeta)}$ and degrees (local and in tangent directions) for a skew product; these degrees are the \emph{relative degree} $\rdeg(\phi_*)$, \emph{local degree} $\deg_\zeta(\phi)$ and \emph{local degree in a direction} $\deg_{\phi, \bvec v}(\phi)$. \autoref{prop:skew:chaindir}, \autoref{thm:skew:chainrule} and \autoref{thm:sumdirdegs} show that the previously established relations between these quantities for rational maps still hold in the new setting. After reintroducing reduction for skew products we obtain a generalisation \autoref{thm:skew:reduction} of further consequential results \cite[Theorem 7.34, Lemma 7.35]{Bene}. These results state that the tangent map $\phi_\#$ can be understood through the reduction $\overline \phi$ and that the tangent map disagrees with $\phi_*$ itself when and only when a \emph{bad direction} $\bvec v$ contains a preimage of the entire projective line. \autoref{thm:intervalstretch} extends the idea that on a small interval with local degree $m$, $\phi_*$ is a linear map with gradient $m\q$.

 We reuse the definitions of ramification locus $\Ram(\phi) = \set[\zeta \in \P^1_\an]{\deg_\zeta(\phi) \ge 2}$, and injectivity locus $\Inj(\phi) = \P^1_\an \sm \Ram(\phi)$. We also say a skew product $\phi_* = \phi_{1*} \circ \phi_{2*}$ is \emph{tame} iff the underlying rational map $\phi_2$ is tame. Note that if $\chara K = p$ then we can always at the least find a decomposition where $\phi_2$ is separable, by transferring (inverse) Frobenius automorphisms onto $\phi_1^*$. We discuss the the separability and uniqueness of the decomposition of skew products. 
 Again, thanks to decomposition, we obtain \autoref{thm:skew:ramconvcrit} as a direct corollary of the theorem by Xander Faber for rational maps \cite{Faber13.1}. This says that $\Ram(\phi)$ is a subset of $\Hull(\Crit(\phi_2))$ whose endpoints lie in $\Crit(\phi_2) = \Ram(\phi) \cap \P^1$.
 \unsure{Should I write cor or thm?}
As a further corollary, we find that if $\phi_*$ is tame then $\deg_\zeta(\phi) = 1$ for every Type IV point of $\P^1_\an$.

\hyperref[sec:skew:top]{Subsection~\ref*{sec:skew:top}} contains some possibly new proofs and restatements of established topological ideas, and an upgrade for \autoref{thm:intervalstretch} as \autoref{thm:skew:bigintervalstretch}. We highlight the Extreme Value Theorem \autoref{thm:skew:extvalue} and the Hyperbolicity Theorem (\autoref{thm:skew:hyperbolicity}). \hyperref[sec:skew]{Subsection~\ref*{sec:skew}} contains more results that we leave out of this introduction.


\subsection{Dynamics of a Skew Product}

Having established the basic analytical properties of a skew product, which largely appear to be the same, it is natural to ask: which dynamical properties of rational maps still hold for skew products? Despite the decomposition theorem, a skew product is \emph{nowhere} analytic, and under iteration this means various algebraic techniques used to prove dynamical results for rational maps will fail in the new setting -- we are required to think more topologically.

At first, the situation may seem discouraging, because whilst a rational map always has finitely many fixed points in $K$, a skew product may have uncountably many or zero classical fixed points, see \autoref{ex:dyn:infinitelymanyfixedone}, \autoref{ex:dyn:zerofixed}.


Skew products do end up behaving like rational maps in many important ways, but there are marked and fascinating differences. For instance, consider periodic points on the hyperbolic space, $\bH$. We see that when $\q < 1$, a skew product can have interior \emph{\stronglyattracting{}} points $\zeta \in \bH$ with all directions attracting, except perhaps finitely many indifferent ones. Moreover a Type II point could have both attracting and repelling directions - we call these \emph{saddle} points. 
%
%
For skew products, a Type III or IV periodic point may be repelling or attracting; see \autoref{ex:dyn:typeIIIrep} for an example of a repelling Type III point. Under a reasonable hypothesis about the value group, we prove that Type III points are indifferent \autoref{thm:dyn:typeIIIindiff}. We also resolve the issue with Type IV points for a skew product $\phi_*$ with scale factor $\q \ge 1$, \autoref{thm:dyn:typeIVindiff}. More is discussed in \autoref{sec:per}.

\subsection{Fatou and Julia}

\begin{defn}
Let $\phi_*$ be a skew product. We say an open set $U \subseteq \P^1_\an$ is \emph{dynamically stable} under $\phi_*$ iff $\displaystyle \bigcup_{n \ge 0} \phi_*^n(U)$ omits infinitely many points of $\P^1_\an$. 
The \emph{(Berkovich) Fatou set} of $\phi_*$, denoted $\F_{\phi, \an}$, is the subset of $\P^1_\an$ consisting of all points $\zeta \in \P^1_\an$ having a dynamically stable neighbourhood. 
The \emph{(Berkovich) Julia set} of $\phi_*$ is the complement $\J_{\phi, \an} = \P^1_\an \sm \F_{\phi, \an}$ of the Berkovich Fatou set.
\end{defn}

We first seek to classify the periodic points of a skew product as Fatou or Julia, generalising 
 \cite[Theorem 8.7]{Bene2} for rational maps which asserts that all Type III and IV points are Fatou, and that Type II points are Julia if they repel or have non-periodic bad directions. Unfortunately this is rather troublesome; to begin with the above discussion, we know that Type III or IV points could be repelling, and there might exist attracting directions and saddle points. 
 Fortunately, the new theorem below takes on very clean form, perhaps even more so than the original version for rational maps, but albeit of different flavour. It turns out that the main deciding factor is the multiplier, and that saddle points are always Julia because they are \weaklyrepelling{}. We say a direction $\bvec v \in \Dir \zeta$ is \emph{exceptional} iff it has a finite backward orbit when restricted to the orbit of the periodic point $\zeta$.  
 
\begin{thm}\label{thm:intro:perptdynone}
 Let $\phi_*$ be a skew product of scale factor $\q$, and let $\zeta \in \P^1_\an$ be a periodic point of Type II, III, or IV of period $n \ge 1$.
 Then $\zeta \in \J_{\phi, \an}$ is Julia if and only if either
\begin{enumerate}
 \item $\zeta$ is \weaklyrepelling{} i.e. $\deg_\zeta(\phi)\q > 1$; or
 \item $\zeta$ is Type II and either none of the directions $\bvec v \in \Dir\zeta$ are exceptional or any of the bad directions $\bvec v \in \Dir\zeta$ is not exceptional.
\end{enumerate}
 Moreover if $\zeta \in \F_{\phi, \an}$ is Fatou then every direction intersecting $\J_{\phi, \an}$ is exceptional and is $\phi_\#^j(\bvec v)$ for some bad direction $\bvec v$.
\end{thm}
 
 Unfortunately, $\q < 1$ means that fixed classical points will often be `superrepelling', and worse also Fatou according to the definition of $\F_{\phi, \an}$; see \autoref{ex:dyn:superrepellingfatou}. However a (non-super-)repelling point is always Julia, as stated in \autoref{prop:dyn:classicalrepelling}.

One can also extend the results of Benedetto about wandering domains of rational maps \cite{Bene1} to skew products. In particular, \autoref{thm:wand:wanddomsaredisks} says that if $\phi_*$ is a skew product of relative degree $d \ge 2$ and scale factor $\q \ge 1$ then any wandering domain $U \subseteq \F_{\phi, \an}$ eventually becomes a disk under iteration. If additionally our base field $K = \K$ the Puiseux series and $\q = 1$, then the boundary points of $U$ are Type II Julia preperiodic points.
%

 \subsection{Classification of Fatou Components}
 
 We also take the same definitions of attracting, indifferent and wandering Fatou component as with rational maps, see \cite{Bene}. We occasionally specialise to simple ($\q = 1$) skew products which are defined over a discretely valued subfield. For instance, these results will apply to simple \emph{$\k$-rational} skew products, meaning they are defined over the field $\k((x))$ of Laurent series with coefficients in some algebraically closed characteristic $0$ field $\k$ (the residue field).
 
 The following theorems describing the attracting and indifferent components are a generalisation to skew products of those due to Rivera-Letelier \cite{RL1, RL2, RL4}.

\begin{thm}\label{thm:intro:attractingperpt}
 Let $\phi_*$ be a skew product of relative degree $d \ge 2$, let $a \in \P^1(K)$ be an attracting periodic point of minimal period $m \ge 1$, and let $U \subseteq \P^1_\an$ be the immediate attracting basin of $\phi_*$. Then the following hold.
\begin{enumerate}
 \item $U$ is a periodic Fatou component of $\F_{\phi, \an}$, of period $m$.
 \item $U$ is either an open Berkovich disk or a domain of Cantor type.
 \item If $U$ is a disk, then its unique boundary point is a Type II repelling periodic point, of minimal period dividing $m$.
 \item If $U$ is of Cantor type, then its boundary is homeomorphic to the
Cantor set and is contained in the Berkovich Julia set.
\item The mapping $\phi_*^m : U \to U$ is $l$-to-$1$, for some integer $l \ge 2$.
\end{enumerate}
\end{thm}

\begin{thm}\label{thm:intro:rivindiff}
Let $L$ be a discretely valued subfield of $K$, and let $\phi_*$ be a simple skew product defined over $L$ of relative degree $d \ge 2$. Let $U \subseteq \P^1_\an$ be an indifferent component for $\phi_*$. Then the following hold.
\begin{enumerate}[label=(\alph*), ref=\theenumi]
\item $U$ is a periodic connected component of $\F_{\phi, \an}$, of some minimal period $m \ge 1$.
\item $U$ is a rational open connected Berkovich affinoid.
\item Each of the finitely many points of the boundary $\partial U$ is a Type II point lying in the Berkovich Julia set.
\item $\phi_*^m$ permutes the boundary points of $U$; in particular, each is periodic.
\item The mapping $\phi_*^m : U \to U$ is bijective.
\end{enumerate}
\end{thm}

 We end this section of the introduction with the most important theorem in \autoref{sec:class} -- we recover the classification of Fatou components, which was originally proved for rational maps by Rivera-Letelier \cite{RL2, RL1}.
 
 \begin{thm}[Classification of Fatou Components]\label{thm:intro:class}
 Let $L$ be a discretely valued subfield of $K$. Let $\phi_* : \P^1_\an(K) \to \P^1_\an(K)$ be a simple skew product defined over $L$ of relative degree $d \ge 2$ with Berkovich Fatou set $\F_{\phi, \an}$, and let $U \subset \F_{\phi, \an}$ be a periodic Fatou component. Then $U$ is either an indifferent component or an attracting component, but not both.
\end{thm}

In closing the introduction, we mention that H.\ Nie and S.\ Zhao have developed a noteworthy alternative approach to these problems and have an independent proof of the classification of Fatou components. They also plan to release a proof of equidistribution for skew products.


\subsection{Organisation}

 We start with \autoref{sec:nonarch}, an overview of non-Archimedean analysis and algebra that will be useful in practice. Next, \autoref{sec:skew} provides a lengthy development of the non-Archimedean skew product, from motivation to local degrees, and other geometric results. In the final sections we will explore the dynamics of a skew product, again comparing to its rational cousin. In \autoref{sec:per} we give an elementary study of periodic points. \hyperref[sec:fj]{Section~\ref*{sec:fj}} 
 is dedicated to defining and understanding the basic properties of the Fatou and Julia set of a skew product; we examine how periodic points are related to the Fatou-Julia dichotomy, proving \autoref{thm:intro:perptdynone}. The ultimate goal is to prove the generalisation, \autoref{thm:intro:class}, of the Rivera-Letelier classification of Fatou components, which will be the focus of \autoref{sec:class}.

 \subsection{Acknowledgements}
First and foremost I would like to thank my doctoral advisor, Jeffrey Diller, for suggesting this avenue of research, for his time, encouragement, and for sharing his remarkable editorial expertise.

Along with my advisor, I thank my dissertation committee, Nicole Looper, Eric Riedl, and Roland Roeder, for their patience, insight, and helpful comments which lead to a more effective exposition of new concepts.


I am grateful to Xander Faber for various helpful discussions about non-Archimedean dynamics.

I would like to show my appreciation to Robert Benedetto for his exceptional book which imparted a great deal of intuition and insight. His mathematical framework for rational maps positively guided the development of skew products. 

Lastly, I thank Hongming Nie and Shengyuan Zhao for intriguing discussions about their `twisted rational maps', which are equivalent to the skew products introduced here.

In addition, I am grateful to the NSF for its support of my work through grant DMS-1954335.





\makeatletter
\let\thethm\the@thm
\let\theprop\the@prop
\let\thecor\the@cor
\let\thelem\the@lem
\makeatother





\section{Background}\label{sec:nonarch}
\subsection{Non-Archimedean Metrics}

\begin{defn}
 A metric space $(X, d)$ is \emph{non-Archimedean} iff \[d(x, z) \le \max \set{d(x, y), d(y, z)} \quad \forall x, y, z \in X.\]
\end{defn}

Such a $d$ is often called an \emph{ultrametric}, due to this \emph{strong triangle inequality}. In an algebraic context, we would like this to derive from a norm which respects addition and multiplication. Furthermore, we will want to consider something called a \emph{seminorm}, informally this is a norm where we relax the condition $\norm a = 0 \implies a=0$. The following notions will be fundamentally important in this work. 

\begin{defn}\label{defn:nonarch:seminorm}
 Let $G$ be a group. A \emph{seminorm} is a function $\norm\cdot : G \to \R_{\ge 0}$ such that $\norm 0 = 0$, $\norm{a} = \norm{-a} \quad\forall a \in G$, and $\norm{a + b} \le \norm{a} + \norm{b}\quad\forall a, b \in G$.
\begin{itemize}
 \item This is a \emph{norm} iff additionally \[\norm{a} = 0 \implies a=0.\]
 \item A seminorm $\norm\cdot$ on $G$ is said to be \emph{non-Archimedean} iff \[\norm{a+b} \le \max \set{\norm a, \norm b} \quad\forall a, b \in G.\]
 \item A seminorm $\norm\cdot$ on a ring $R$ is \emph{multiplicative} iff \[\norm{a\cdot b} = \norm{a}\cdot\norm{b} \quad\forall a, b \in R.\]
 \item We say a field $(K, \abs\cdot)$ is \emph{non-Archimedean} iff $\abs\cdot$ is a multiplicative non-Archimedean norm on $K$. In this case we refer to $\abs\cdot$ as an \emph{absolute value}.
\end{itemize}
\end{defn}

It is clear that any non-Archimedean norm induces a non-Archimedean metric.

\begin{rmk}
 A \emph{(semi)valuation} can always be obtained by taking $\log$ of a (semi)norm, $\log_{\eps}\norm\cdot$, for any $\eps \in (0, 1)$. We apply the other adjectives of the previous definition respectively.
 \end{rmk}
  
\begin{rmk}
 Note that $\norm{a} = \norm{-a} \quad\forall a \in R$ on a ring is implied by the multiplicative condition.
\end{rmk}

\begin{ex}[Trivial Norm]
 Let $G$ be any group. Then there is a non-Archimedean norm $\norm[\text{triv}]\cdot$ such that $\norm[\text{triv}] 0 = 0$ and $\norm[\text{triv}]{a} = 1$ for any $a \in G \sm\set0$. If $G$ is an integral domain (e.g. a field), this is also multiplicative.
\end{ex}

\begin{defn}\label{defn:nonarch:ring}
 Let $(K, \abs\cdot)$ be a non-Archimedean field.
\begin{enumerate}
 \item The \emph{ring of integers} of $K$ is given by \[\bO_K = \set[a \in K]{\abs a \le 1}.\]
 \item This $\bO_K$ has a unique maximal ideal, \[\mathcal M_K = \set[a \in K]{\abs a < 1}.\]
 \item The \emph{residue field} of $K$ is the quotient field $\k = \bO_K / \mathcal M_K$.
 \item The \emph{value group}, $\abs{K^\times} \le (0, \infty)$ is the range of $\abs\cdot$ on $K^\times = K \sm \set 0$.
\end{enumerate}
\end{defn}

It is possible that $K$ has \emph{mixed characteristic} meaning $\chara K = 0$ and $\chara k = p$. Otherwise they have \emph{equal characteristic} $\chara K = \chara k$. The value group is a (multiplicative) subgroup of reals; if the value group is dense in $(0, \infty)$ then we shall say it is \emph{densely valued}; if this is non-trivial but not dense in $(0, \infty)$ then it must be cyclic, i.e. $|K^\times| \cong \Z$, in which case we say $\abs\cdot$ is \emph{discretely valued}.

\begin{prop}\label{prop:nonarch:series}
 Let $(G, \norm\cdot)$ be a non-Archimedean group. An infinite series $\sum_{n=1}^N a_n$ is Cauchy if and only if $a_n \to 0$; in particular when $(G, \norm\cdot)$ is complete an infinite series converges if and only if its terms tend to $0$. 
\end{prop}

\begin{prop}[Completion]
 Given a non-Archimedean group $(G, \norm\cdot)$, we can form it's completion $\hat G$. This can be thought of as all convergent series over $G$. The induced norm is naturally given by \[\norm{\lim_{n \to \infty}a_n} = \lim_{n \to \infty} \norm{a_n}.\]
\end{prop}

\begin{ex}[$p$-adic Numbers]
 Consider $\Q = (\Q, | \cdot |_p)$, the rational numbers with $| \cdot |_p$ the $p$-adic norm. This is defined such that $\left| \frac abp^n\right| = p^{-n}$ if $a, b \in \Z$ are coprime to $p$. One can easily check that this norm is non-Archimedean and multiplicative. This norm also gives rational numbers a natural $p$-adic expansion. For example when $p = 5$, we can write
 \[\frac{42}{5} = 2\cdot 5^{-1} + 3 \cdot 1 + 1 \cdot 5^1.\]
 Observe that $\left|\frac{42}{5}\right|_5 = 5^1$ is given by \[\left|\frac{42}{5}\right|_5 = \max \set{|2\cdot 5^{-1}|_5, |3 \cdot 1|_5, |1 \cdot 5^1|_5}\]
 \[= \max \set{|2|_5 |5^{-1}|_5, |3|_5 |1|_5, |1|_5 |5^1|_5}\]
  \[= \max \set{|5^{-1}|_5, |1|_5, |5^1|_5}\]
  \[= \max \set{5^1, 1,  5^{-1}}.\]
  This is a very typical way we use the non-Archimedean property of a (semi)norm.
  
  The completion of $(\Q, | \cdot |_p)$ is $(\Q_p, | \cdot |_p)$, the $p$\emph{-adic numbers}. One should think of $\Q_p$ as all the half-infinite $p$-adic expansions, i.e.\ `Laurent series' in the `variable' $p$. The ring of integers of $\Q_p$ is the $p$\emph{-adic integers}, $\Z_p$ whose elements are the `Maclaurin series' in $p$. One can check that $\mathcal M_{\Q_p} = p\Z_p$. The residue field for both $\Q_p$ and $\Q$ is $\Z/p$, meaning they have mixed characteristic. Finally, we define $\hp$ as the completion of $\Q_p$.
\end{ex}

This shows that a power series structure provides a simple way to define a non-Archimedean seminorm, using the ``lowest order term'', or the lowest index. Here is another.

\begin{ex}[Formal Laurent series]
 Consider $\k((x)) = \Frac \k[[x]]$, the field of formal Laurent series over some base field $k$. Note that these have finite principle part. We can define the order of vanishing norm $| \cdot |$ as follows, beginning with a fixed $\eps \in (0, 1)$. For any $m_0 \in \Z$ and coefficients $(c_m) \subset k$, let \[f(x) = \sum_{m = m_0}^\infty c_m x^k\]
 where $c_{m_0} \ne 0$. Then $\abs f = \eps^{m_0}$. We see that the constants $k$ are all given the trivial norm and a series $f$ has its norm determined by its lowest order term, $\abs f = \abs{c_m x^m} = \abs{x^m}$. Here the residue field is $k$ itself, which naturally lives inside $\k((x))$, hence these are of equal characteristic. This field is discretely valued, with value group $\set[\eps^n]{n\in \Z}$.
\end{ex}

The following lemmata spell out the idea we have already used in the above examples, that the norm of an element is always the norm of its dominant summand.

\begin{prop}[Strong triangle equality]
 Let $\norm{\cdot}$ be a non-Archimedean seminorm on $G$. Then \[\norm a \ne \norm b\implies \norm{a+b} = \max \set{\norm a, \norm b}.\]
 \end{prop}

\begin{proof}
 $\norm{a+b} \le \max \set{\norm a, \norm b}$ by the non-Archimedean definition, hence $\norm{a+b} \le \norm a$. Also we have $\norm a = \norm{a+b + -b} \le \max \set{\norm{a+b}, \norm{-b}} = \max \set{\norm{a+b}, \norm{b}}$. If WLOG $\norm a > \norm b$, it must be that $\max \set{\norm{a+b}, \norm{b}} = \norm{a+b}$. Therefore $\norm{a+b} = \norm{a}$. 
\end{proof}

\begin{prop}[Extended strong triangle (in)equality]\label{lem:nonarch:extendedtriangle}
Suppose that $\sum_{j=1}^\infty a_j$ converges in $(G, \norm\cdot)$. Then
\[\norm{\sum_{j=1}^\infty a_j} \le \max_j \norm{a_j},\]
and moreover if $\norm{a_N} > \norm{a_j} \forall j\ne N$, then \[\norm{\sum_{j=1}^\infty a_j} = \norm{a_N}.\]
\end{prop}

The proof follows from the previous two results. 
The following definition provides not only a key example, but also one of the main fields of interest in sequels and the author's thesis.

\begin{ex}[Puiseux series]\label{defn:nonarch:puiseux}
 Let $\k$ be a field. We shall define $\K(\k) = \K$, the field of \emph{Puiseux\footnote{or Newton-Puiseux} series} over $\k$ with variable $x$. For $a = a(x) \in \K$ there is an $n \in \N$, $m_0 \in \Z$ and coefficients $(c_m) \subset \k$ such that we write \[a = \sum_{m = m_0}^{\infty} c_{m}x^{\frac{m}{n}}.\] Addition and multiplication works as with any power series. Moreover, these series converge from their partial sums under the non-Archimedean metric defined next. We fix a value of $0 < |x| = \eps < 1$ and give $\k$ the trivial norm by setting $|c| = 1\quad~\forall c\in \k^\times$. This information uniquely determines a norm on $\K$ in the sense of \autoref{lem:nonarch:extendedtriangle}: assuming $c_{m_0} \ne 0$, \[|a| = |x|^{\frac{m_0}{n}}.\]
 This dependence of the power series on the denominator $n$ is the only thing preventing $\K$ from being complete with respect to $\abs{\cdot}$. The completion of $\K$, the \emph{Levi-Civita field} $\hk$, is the field with elements of the form \[\gamma = \sum_{j = 0}^{\infty} c_{j}x^{r_j}, \text{ where }(r_j) \subset \Q,\ r_j \to \infty.\]
\end{ex}

\begin{thm}[Puiseux's Theorem]
 The Puiseux series, $\K(\C)$, is the algebraic closure of the formal Laurent series $\C((x))$.
\end{thm}

This is a useful field to use when working with germs of algebraic curves in $\C^2$. \emph{Puiseux's Theorem} says that any irreducible curve $P(x, y) = 0$ crossing $\set{x = 0}$ (except the line itself) can be given locally by branches of the form $y = \gamma(x)$ where $\gamma$ is a Puiseux series. \info{I think the $n$ is the degree of ramification at $x=0$.} 
 Note that $\K$ is the direct limit $\varinjlim \C((x^\frac 1n))$. The Levi-Civita field $\hk$ is both algebraically closed and complete. When $\chara k = 0$, the Puiseux series over $\overline k$, $\K(k)$ is the algebraic closure of $k((x))$. However $\overline{k((x))}$ is larger in positive characteristic.
 
 Throughout this article, when we use Puiseux series, we will not declare an $\eps \in (0, 1)$ for which $\abs x = \eps$ as in the definition above. Instead, we will simply refer to the quantity $\abs x$ which intrinsically provides the same information. This also serves as a visual reminder to the order of vanishing of a series, for instance if $\abs{a(x)} = \abs x^{3/2}$ then the first non-zero term of the Puiseux series $a(x)$ must be $cx^{3/2}$.

 
 \subsection{Disks}
 
\begin{defn}\label{defn:nonarch:disks}
 Let $(K, \abs\cdot)$ be a non-Archimedean field\footnote{although similar can be said for any non-Archimedean metric space}. We define the open and closed disks of radius $r$ centred at $a \in K$, respectively below.
 \[D(a,r):=\set[b\in K]{|b-a|<r}\]
 \[\overline{D}(a,r):=\set[b\in K]{|b-a|\le r}\]
 If the radius $r$ of a disk is in the value group $|K^\times|$, we say this disk (and its radius $r$) is \emph{rational}, otherwise, we say it is \emph{irrational}.
\end{defn}
 
 By convention we will allow notationally that $\CD(a, 0) = \set a$, but not formally refer to this as a `disk'. 
 The terminology of rationality follows from the notion in fields like the Puiseux series, where the value group is $\abs{\K^\times} = \abs x^\Q \cong \Q$, however we will use these adjectives even if the value group is not isomorphic to $\Q$. It follows immediately from the definition that for an irrational radius $r$, open and closed disks coincide, $D(a, r) = \CD(a, r)$. The non-Archimedean metric results in some weird quirks for disks. For example, consider $a, b \in K$ such that $\abs{a-b} = r > 0$. Then in an Archimedean space, the overlap of the disks $\overline{D}(a,r)$, $\overline{D}(b,r)$ would be a non-trivial proper subset, much similar to the overlap of $D(a,r)$ and $D(b,r)$. In this non-Archimedean setting we have that $\overline{D}(a,r) = \overline{D}(b,r)$ but $D(a,r) \cap D(b,r) = \emp$. \improvement{Draw nice picture} Moreover any two disks are either disjoint or nested (or equal). Perhaps confusingly, the rational closed disk $\CD(a, r)$ is never the closure of the rational open disk $D(a, r)$. The following proposition, lifted from \cite[Proposition 2.4]{Bene}, details these differences.
 
\begin{prop}\label{prop:nonarch:disks}
 Let $(K, \abs\cdot)$ be a non-Archimedean field.
\begin{enumerate}
 \item Let $a, b \in K$, and $R \ge r > 0$ such that $a \in D(b, R)$. Then
 \[D(a, r) \subseteq D(b, R) \and D(a, R) = D(b, R).\]
 \item Let $a, b \in K$, and $R \ge r > 0$ such that $a \in \CD(b, R)$. Then
 \[\CD(a, r) \subseteq \CD(b, R) \and \CD(a, R) = \CD(b, R).\]
 \item Let $D_1, D_2$ be disks such that $D_1 \cap D_2 \ne \emp$. Then either
 \[D_1 \subseteq D_2 \text{ or } D_1 \supseteq D_2\]
 \item All disks in $K$ are clopen.
 \item $K$ is totally disconnected.
\end{enumerate}
\end{prop}
 
\begin{rmk}
 If $D$ is an open (resp.\ closed) disk then the smallest possible radius $r$ for which $D = D(a, r)$ (resp.\ $\CD(a, r)$) is also its \emph{diameter}, $\diam(D) = \sup_{a, b \in D} \abs{a-b}$. If $K$ is densely valued, then there is a unique radius. See \cite[Proposition 2.5]{Bene}. In particular, whenever $K$ is algebraically closed, then $\abs{K^\times}$ is a $\Q$-vector space dense in $(0, \infty)$.
\end{rmk}

It is common for non-Archimedean fields to not be \emph{spherically complete} - that is, there may be a sequence of nested closed disks \[\CD(a_1, r_1) \supset \CD(a_2, r_2) \supset \CD(a_3, r_3) \supset\CD(a_4, r_4) \supset \cdots\]
with empty intersection \[\bigcap_{n = 1}^\infty \CD(a_n, r_n) = \emp.\]
Of course, in a complete field, if $r_n \to 0$, then this intersection is always a singleton. As an example, consider the following sequence in the complete non-Archimedean field $\hk$.
\[\CD\left(1, \abs{x^{1-1/2}}\right) \supset \CD\left(1+ x^{1-1/2}, \abs{x^{1-1/3}}\right) \supset \CD\left(1+ x^{1-1/2} + x^{1-1/3}, \abs{x^{1-1/4}}\right) \supset \cdots \]
\[\cdots \supset \CD\left(\sum_{n=1}^{N-1} x^{1-1/n}, \abs{x^{1-1/N}}\right) \supset \cdots\]
One can check it has a non-empty intersection if and only if the infinite series $\sum x^{1-1/n}$ converges in $\hk$. The limit does not exist! We will return to this idea when we discuss Type IV norms in the Berkovich projective line. In any case, if the intersection of nested closed disks is non-empty, containing say $a \in K$, then \[\bigcap_{n = 1}^\infty \CD(a_n, r_n) = \CD(a, r),\] where $r = \lim r_n$ is possibly zero. 

\subsection{Projective Line and Affinoids}

So far we have discussed a non-Archimedean field $K = \A^1(K)$, which is the affine line, but in general we shall work on the projective line $\P^1(K) = \A^1(K) \cup \set \infty$ with its usual definition. It is natural to extend the definition of disks and their types to one that is invariant of the fractional linear transformations $\PGL(2, K)$. We shall also recall (from \cite[\S3.5]{Bene}) important topological objects called affinoids, which are merely disks subtracted from disks.

\begin{defn}
 Let $(K, \abs\cdot)$ be a non-Archimedean field. A \emph{disk} is any one of the following.
\begin{itemize}
 \item A \emph{rational closed disk} $D \subset \P^1(K)$ is either a rational closed disk $D \subset K$ or $D = \P^1(K) \sm E$ where $E \subset K$ is a rational open disk.
  \item A \emph{rational open disk} $D \subset \P^1(K)$ is either a rational open disk $D \subset K$ or $D = \P^1(K) \sm E$ where $E \subset K$ is a rational closed disk.
   \item An \emph{irrational disk} $D \subset \P^1(K)$ is either an irrational disk $D \subset K$ or $D = \P^1(K) \sm E$ where $E \subset K$ is an irrational disk.
\end{itemize}
\end{defn}

The generalisation of \autoref{prop:nonarch:disks} (iii) would include the possibility that two disks cover the whole space.

\begin{defn}
A \emph{connected affinoid} is a nonempty intersection of finitely many disks $D_1, \dots, D_n$. If all of the disks $D_1, \dots, D_n$ are closed, open, rational, or irrational, then the connected affinoid
$D_1 \cap \cdots \cap D_n$ is respectively said to be closed, open, rational, or irrational.
The connected open affinoid of the form $\set{r < \abs{z-a} < R} = D(a, R) \sm \CD(a, r)$ is called an \emph{open annulus}.
\end{defn}

If two connected affinoids $U$ and $V$ and non-empty intersection, then both $U \cap V$ and $U \cup V$ are connected affinoids.

\subsection{Power Series and Constructing Seminorms}\label{sec:seminorms}

To study rational functions on a non-Archimedean field $K$, we will want to understand them as analytic functions in neighbourhoods. Here, `analytic' means given by Taylor or Laurent series, rather than some notion of holomorphy, however it remains true that any rational function is locally analytic. In this section we will define and recall some notions of convergent power series for non-Archimedean fields. See \cite[\S 3]{Bene}.

\subsection{Taylor Series}

Following \autoref{prop:nonarch:series} a Taylor series
\[f(y) = \sum_{n=0}^\infty c_n (y - a)^n \in K[[y-a]]\] will converge at $y=b$ if and only if $\abs{c_n(b-a)^n} \to 0$. Let $\abs{b-a} = r$, then the series converges at $b$ if and only if it converges at every $b' \in \CD(a, r)$. This behaviour is a little nicer than the Archimedean situation in complex analysis.

\begin{defn}
 The \emph{radius of convergence}, $R \in [0, \infty]$, for a Taylor series $f(y)$ as above is \[R = \sup \set[r \in \R]{\abs{c_n}r^n \to 0}.\]
 The \emph{domain of convergence} of a Taylor series $f(y)$, is defined to be \[\dom(f) = \set[a\in K]{f(a) \text{ converges}}.\]
\end{defn}

\begin{prop}
 Let $(K, \abs\cdot)$ is non-trivially and densely valued non-Archimedean field. Let $a \in K$ and $f(y) \in K[[y-a]]$ be a Taylor series with radius of convergence $R$. Then
 \[\dom(f) = 
\begin{cases}
 D(a, R) = \CD(a, R), & \text{if } R \nin \abs{K^\times}, \text{ or otherwise} \\
 D(a, R), & \text{if } \abs{c_n}R^n \nrightarrow 0, \text{ and}\\
 \CD(a, R), & \text{if } \abs{c_n}R^n \to 0.
\end{cases}
\]
\end{prop}

\begin{defn}\label{defn:nonarch:powerseriesring}
 Let $(K, \abs\cdot)$ be a non-Archimedean field. We define the power series rings
 \[\mathcal A(a, r) = \set[f \in {K[[y-a]]}]{D(a, r) \subseteq \dom(f)}\]
 \[\overline{\mathcal A}(a, r) = \set[f \in {K[[y-a]]}]{\CD(a, r) \subseteq \dom(f)}\]
\end{defn}

Note that the polynomials live in every power series ring, moreover $K \subset K[y] \subset \overline{\mathcal A}(a, r) \subset \mathcal A(a, r)$.


\begin{defn}[Weierstrass Degree]\label{defn:nonarch:wdegdisk}
 Let $a \in K$, \[f(y) = \sum_{n=0}^\infty c_n (y-a)^n \] be a non-zero power series with radius of convergence $r > 0$, and let $D = \dom(f)$. The \emph{Weierstrass degree} $\wdeg_D(f)$ of $f$ is defined according to two cases as follows.
\begin{enumerate}
 \item If $D = \CD(a, r)$ is a rational closed disk, then \[\wdeg_D(f) = \max \set [d \in \N]{\abs{c_d}r^d = \max_n \abs{c_n}r^n}.\]
 \item If $D = D(a, r)$ is an irrational or a rational open disk, then \[\wdeg_D(f) = \min \set [d \in \N]{\abs{c_d}r^d = \sup_n \abs{c_n}r^n} \cup \set \infty.\]
\end{enumerate}
\end{defn}

The Weierstrass degree plays a crucial role in understanding the zeros images of analytic functions on a non-Archimedean field. One can easily check that $\wdeg_D(g) + \wdeg_D(h)$ = $\wdeg_D(gh)$ and hence any power series with a multiplicative inverse (a unit) has Weierstrass degree $0$. The Weierstrass preparation theorem shows that this quantity really is a `degree' on $D$.

\begin{thm}[Weierstrass Preparation Theorem]
 Let $(K, \abs\cdot)$ be a complete non-Archimedean field, $a \in K$, $r \in \abs{K^\times}$, and $f \in \overline{\mathcal A}(a, r)$ be a non-zero power series. Then there exists a monic polynomial $g \in K[y]$ of degree $d = \wdeg_{\CD(a, r)}(f)$ and a unit power series $h \in \overline{\mathcal A}^\times(a, r)$ such that $f = gh$ and all the zeroes of $g$ lie in $\CD(a, r)$.
\end{thm}

This has several consequences; immediately, we see that in an algebraically closed complete field, such a power series $f$ has $d$ zeroes (counting multiplicity) in the disk; secondly, moreover $f$ is a $d$-to-$1$ mapping on $\CD(0, r)$.

\begin{thm}\label{thm:seminorms:mapdisk}
Let $(K, \abs\cdot)$ be a complete and algebraically closed non-Archimedean field, $a \in K$, and \[f(y) = \sum_{n=0}^\infty c_n (y-a)^n \] be a non-zero power series converging on a disk $D$ of radius $r > 0$. Suppose that $d = \wdeg_D(f - c_0) < \infty$, then
\begin{enumerate}
 \item $f(D)$ is a disk of the same type as $D$ (rational closed, rational open, or irrational), centred at $f(a) = c_0$, of radius $\abs{c_d}r^d$; and
 \item $f : D \to f(D)$ is a $d$-to-$1$ mapping, counting multiplicity.
\end{enumerate}
\end{thm}

\subsection{Laurent Series}

The \emph{formal} Laurent series $K((y-a)) = \Frac K[[y-a]]$ represent the set of power series in $(y-a)$ with infinitely many positive powers but only \emph{finitely many negative} ones. It is not hard to see that most of the results in the previous subsection apply to formal Laurent series on punctured domains. To be precise, if \[f(y) = \sum_{n= -n_0}^\infty c_n (y-a)^n \] is a formal Laurent series ($n_0 > 0$) and \[f_+(y) = \sum_{n=0}^\infty c_n (y-a)^n \] converges on a disk $\CD(a, r)$, then $f$ converges on $\CD^*(a, r) = \CD(a, r)\sm\set 0$. However, it will be useful to consider bi-infinite Laurent series when we inspect rational maps.

For example, suppose $0 < \abs a < \abs b$ and we want to consider the rational map \[f(y) = \frac 1{y-a} - \frac 1{y-b}\] over the annulus $U = \set{\abs a < \abs{z} < \abs b}$. We may expand both as series in powers of $y$ using the usual binomial trick \[(1-t)^{-1} = 1+ t + t^2 + \cdots,\] but this converges when and only when $\abs t < 1$. Therefore on the annulus $U$ we must expand $f$ as
\[f(y) = \frac 1y\frac 1{1-\frac ay} + \frac 1b\frac 1{1-\frac yb} = \frac 1y\sum_{n=0}^\infty \left(\frac ay\right)^n + \frac 1b\sum_{n=0}^\infty \left(\frac yb\right)^n.\]

In general, we can study rational maps through Laurent series converging on annuli. 

\begin{defn}
 Let $K$ be a densely valued non-Archimedean field and $(c_n)_{n=-\infty}^\infty \subset K$. A \emph{Laurent series} $f(y)$ about $a \in K$ is a series of the form \[f(y) = \sum_{n \in \Z} c_n (y - a)^n \in K[[y-a, (y-a)^{-1}]].\]
  The \emph{inner} and \emph{outer radii of convergence} for $f(y)$, $r, R \in [0, \infty]$, are defined respectively (if they exist) as \[r = \inf \set[s \in \R]{\abs{c_n}s^n \to 0}\] \[R = \sup \set[s \in \R]{\abs{c_n}s^n \to 0}.\]
 The \emph{domain of convergence} of a Laurent series $f(y)$, is defined to be \[\dom(f) = \set[a\in K]{f(a) \text{ converges}}.\] 
\end{defn}

\begin{prop}
 Let $(K, \abs\cdot)$ be a densely valued non-Archimedean field. Let $a \in K$ consider a Laurent series \[f(y) = \sum_{n\in\Z} c_n (y - a)^n \in K[[y-a, (y-a)^{-1}]].\] Then $f(y)$ will converge at $y=b$ if and only if $\abs{c_n(b-a)^n} \to 0$ both as $n \to \infty$ and as $n \to -\infty$. Hence $f(y)$ converges for some $y=b$ if and only if the inner and outer radii of convergence $r$ and $R$ both exist (with $r \le \abs{b-a} \le R$). In this case moreover the domain of convergence, $\dom(f)$ is one of the following annuli \[\set{r < \abs{z-a} < R} = D(a, R) \sm \CD(a, r),\]
 \[\set{r \le \abs{z-a} < R} = D(a, R) \sm D(a, r),\]
 \[\set{r < \abs{z-a} \le R} = \CD(a, R) \sm \CD(a, r),\]
 \[\set{r \le \abs{z-a} \le R} = \CD(a, R) \sm D(a, r),\]
 depending only on the boundary cases, whether $\abs{c_n}r^n \to 0$ and/or $\abs{c_n}R^n \to 0$.
\end{prop}

\begin{defn}
 Let \[f(y) = \sum_{n\in\Z} c_n(y - a)^n\] be a Laurent series about $a \in K$. On any open annulus $U = \set{r < \abs{z-a} < R} \subset \dom(f)$ we define
\begin{enumerate}
\item the \emph{inner Weierstrass degree} $\overline\wdeg_{a, r}(f)$ of $f$ to be the largest integer $M \in \Z$ such that $\abs{c_M}r^M = \sup_{n \in \Z}\abs {c_n}r^n$, or $-\infty$ if there is no such integer; and
\item the \emph{outer Weierstrass degree} $\wdeg_{a, R}(f)$ of $f$ to be the smallest integer $N \in \Z$ such that $\abs{c_N}R^N = \sup_{n \in \Z}\abs{c_n}R^n$, or $\infty$ if there is no such integer.
\end{enumerate}
\end{defn}

Note that for Taylor series, $\overline\wdeg_{a, r}(f) = \wdeg_{\CD(a, r)}$ and $\wdeg_{a, r}(f) = \wdeg_{D(a, r)}$.

Despite the hypothesis of the definition, one can think of the inner and outer Weierstrass degrees as a function of the radii and coefficients
\[\wdeg_{a, R}(f) = \min \set [d \in \N]{\abs{c_d}R^d = \sup_n \abs{c_n}R^n} \cup \set \infty,\] ignorant of domains or annuli. As the radius $R$ changes for the annulus $U = \set{r < \abs{z-a} < R}$, $U$ may absorb zeroes of $f(y)$; for each new zero (counted with multiplicity) the outer Weierstrass degree will \emph{increase} by one, and on annuli without zeroes the Weierstrass degree remains constant. This is made explicit in the proposition below, see \cite[Proposition 3.32]{Bene}. \improvement{point to the result about local degrees}
As suggested by the example above, all rational functions have Laurent series expansions on annuli away from their poles. If one considers the Weierstrass degree only as a function of the rational map and the radius, independent of the Laurent series representation. Then one can interpret this number a count of poles and zeroes. Indeed, we can see that the Weierstrass degree \emph{decreases} by one every time the radius crosses a pole of $f(y)$.

\begin{prop}\label{prop:seminorms:rationaldegrees}
Let $(K, \abs\cdot)$ be an algebraically closed, complete non-Archimedean field and let $f(y) \in K(y)$ be a rational function.
\begin{enumerate}
\item If $f$ has no poles in $U = \set{r < \abs{z-a} < R}$, an open annulus, then $f$ has a unique Laurent series expansion converging on $U$.
\item Hence we may write $\overline\wdeg_{a, r}(f)$ and $\wdeg_{a, R}(f)$ for the inner and outer Weierstrass degrees of this unique series at radius $R$ about $a$. Hence these quantities are well defined and finite for any $r, R > 0$.
\item If $f$ has no poles or zeros in $U= \set{r < \abs{z-a} < R}$, then the inner and outer Weierstrass degrees of $f$ on $U$ coincide. i.e. $\overline\wdeg_r(f) = \wdeg_R(f)$.
\item If $f$ has $N_\infty$ poles and $N_0$ zeros in the open disk $D(a, R)$, then \[\wdeg_{a, R}(f) = N_0-N_\infty.\]
\item If $f$ has $N_\infty$ poles and $N_0$ zeros in the closed disk $\CD(a, R)$, then \[\overline\wdeg_{a, R}(f) = N_0-N_\infty.\]
\item If $f$ has $N_\infty$ poles and $N_0$ zeros in $U = \set{r < \abs{z-a} < R}$, then \[\wdeg_{a, R}(f) - \overline\wdeg_{a, r}(f) = N_0-N_\infty.\]
\item If $f$ has $N_\infty$ poles and $N_0$ zeros in the circle $\CD(a, R)\sm D(a, R)$, then \[\overline\wdeg_{a, R}(f) - \wdeg_{a, R}(f) = N_0-N_\infty.\]
\end{enumerate}
\end{prop}

One can extend \autoref{thm:seminorms:mapdisk} to the case of Laurent series and annuli as follows (see \cite[Theorem 3.33]{Bene}).

\begin{thm}\label{thm:seminorms:mapannulus}
Let $(K, \abs\cdot)$ be an algebraically closed, complete non-Archimedean field. Let $0 < r < R$, let $U = \set{r < \abs{z-a} < R}$ be an open annulus, and let $f(y)$ be a non-constant convergent Laurent series on $U \subset \dom(f)$. Write \[f(y) = \sum_{n\in\Z} c_n(y - a)^n\] and suppose that $f - c_0$ has finite inner and outer Weierstrass degrees $M \le N \in \Z$, respectively. Let \[s = \abs{c_M}r^M \and t = \abs{c_N}R^N.\]
\begin{enumerate}
 \item If $M < N$, then $f(U) = D\left(c_0, \max\set{s, t}\right)$.
 \item If $M = N \ge 1$, then $f(U) = \set{s < \abs{z-c_0} < t}$, and the mapping $f : U \to f(U)$ is $M$-to-$1$.
 \item If $M = N \le -1$, then $f(U) = \set{t < \abs{z-c_0} < s}$, and the mapping $f : U \to f(U)$ is $(-M)$-to-$1$.
\end{enumerate}
It follows that in the last two cases, $\abs{f(z) - c_0} = \abs{c_M(z-a)^M}$, for any $z \in U$.
\end{thm}

Finally, we recall a description of how a rational map acts on affinoids, this is lifted from \cite[Proposition 3.29]{Bene}.

\begin{thm}\label{thm:semninorms:affinoidmapping}
Let $(K, \abs\cdot)$ be an algebraically closed, complete non-Archimedean field, and $U \subseteq \P^1(K)$ be a connected affinoid. Let $f(y) \in K(y)$ be a rational function of degree $d \ge 1$. Then
\begin{enumerate}
 \item $f(U)$ is either $\P^1(K)$ or a connected affinoid of the same type, if any, as $U$, and
 \item $f^{-1}(U)$ is a disjoint union $V_1 \cup \cdots \cup V_m$ of connected affinoids, each of the same type, if any, as $U$, and with $1 \le m \le d$.
Moreover, for each $i = 1, \dots, m$, there is an integer $1 \le d_i \le d$ such that every point in $U$ has exactly $d_i$ preimages in $V_i$, counting multiplicity, and such that $d_1 + \cdots + d_m = d$.
\end{enumerate}
\end{thm}

\subsection{Seminorms of Power Series}

Power series rings $K[[y-a]]$ can be equipped with many different non-Archimedean seminorms, but we shall focus on those which agree with the absolute value on $K$. These shall constitute the seminorms of the Berkovich affine line. Whenever we define a seminorm $\norm[\zeta]\cdot$, notationally we will use $\norm[\zeta]\cdot$ and $\zeta$ interchangeably. This will make more sense after defining the Berkovich affine line which contains $K$ in the form of Type I points.

\begin{defn}[Type I Seminorm]
Let $(K, \abs\cdot)$ be a non-Archimedean field and $a \in K$. We define a function called a \emph{Type I seminorm} $\norm[a]\cdot : K[[y-a]] \to [0, \infty)$ by
\[\norm[a]f = \abs{f(a)}.\]
\end{defn}

\begin{prop}
 Let $a \in K$, then $\norm[a]\cdot$ is a well defined non-Archimedean, multiplicative seminorm on $K[[y-a]] \supset K[y]$, which extends the norm $\abs \cdot$ on $K$, i.e. $\norm[a]c = \abs c\quad \forall c\in K$. However it is never a norm since $\norm[a]{y-a} = 0$.
\end{prop}

\begin{defn}[Type II/III Norm]
Let $(K, \abs\cdot)$ be a non-Archimedean field, $a \in K$, $0 < R$. 
We define the ring \[\mathcal L(a, R) = \set[f \in {K[[y-a, (y-a)^{-1}]]}]{\wdeg_{a, R}(f) < \infty}.\]
We also define a function $\norm[\zeta(a, R)]\cdot : \mathcal L(a, r) \to [0, \infty)$ by
\[\norm[\zeta(a, R)]f = \abs{c_d}R^d\]
where \[f(y) = \sum_{n=0}^\infty c_n (y-a)^n \in \mathcal L(a, R), \quad \wdeg_{a, R}(f) = d.\]
We say $\zeta(a, R)$ is a Type II or Type III norm if $R$ is rational or not, respectively.
\end{defn}

\begin{prop}\label{prop:seminorms:order}
 Let $a \in K$ and $r > 0$. Then $\norm[\zeta(a, r)]\cdot$ is a well defined non-Archimedean, multiplicative norm on $\mathcal L(a, r) \supset \overline{\mathcal A}(a, r) \supset K[y]$ which extends the norm $\abs \cdot$ on $K$. Hence, for any disk $D(b, R) \supsetneq \CD(a, r)$, $\norm[\zeta(a, r)]\cdot$ is a norm on $\overline{\mathcal A}(b, R)$ and $\mathcal A(b, R)$, moreover
 \[\norm[\zeta(a, r)]f \le \norm[\zeta(b, R)]f \quad \forall f\in \overline{\mathcal A}(b, R).\]
\end{prop}

If $f \in \overline{\mathcal A}(a, R)$ or more generally $f$ is a Laurent series converging on a closed annulus $E = \set{R- \eps \le \abs{z-a} \le R}$, i.e. $\abs{c_n}R^n \to 0$ as $n \to \pm \infty$, then the supremum $\sup_n \abs{c_n}R^n$ is attained at some $d \in \Z$ and thus the outer Weierstrass degree of $f$ is finite. Therefore $\mathcal L(a, r) \supset \overline{\mathcal A}(a, r)$. However if $f$ converges on $D(a, r)$ but not $\CD(a, r)$ then the sequence $\abs{c_n}r^n$ may not attain its supremum or be bounded.

 By \autoref{prop:seminorms:rationaldegrees}, any rational function $f(y)$ for any radius $R$ has a Laurent expansion on $U = \set{R-\eps < \abs{z-a} < R}$, hence the above definition works for all rational functions $f \in K(y) \subset \mathcal L(a, r)$. One could also consider the opposite annulus $\set{R < \abs{z-a} < R+\eps}$ and try make a similar definition using the inner Weierstrass degree; or one could pick a different centre $b\in \CD(a, R)$. These all turn out to be equal. Further, the Type II/III norm $\norm[\zeta(a, r)]{\cdot}$ is actually a `sup-norm' on $\overline{\mathcal A}(a, r)$. This is showcased in \cite{Bene} but we shall state and prove a little more. \unsure{Haven't seen reference, but doubt this is really new.}\unfinished{PROVE}

\begin{prop}\label{prop:seminorms:typeII/IIImain}
 Let $(K, \abs\cdot)$ be a non-Archimedean field, $a \in K$, $r > 0$, and $f(y)$ be a non-constant rational function. 
 Pick $\eps > 0$ such that $f$ has no poles on $V = \set{R < \abs{z-a} < R + \eps}$ and write \[f(y) = \sum_{n \in \Z} c_n (y-a)^n, \quad \overline \wdeg_{a, R}(f) = d.\] Then
\[\norm[\zeta(a, R)]f = \abs{c_d}R^d = \lim_{\abs{b} \to R}\abs{f(b)}.\]
Where the limit excludes $\abs b = R$. 
 Moreover, $\norm[\zeta(a, R)]f = \abs{f(b)}$ for every $b$ in all but finitely many residue classes of $\CD(a, R)$, i.e. avoiding any open disks $D(b, R) \subset \CD(a, R)$ containing zeroes or poles of $f$.
  Furthermore, $\zeta(a, R)$ depends only on the choice of closed disk $\CD(a, R)$. That is for any $b \in \CD(a, R)$ we have that \[\zeta(a, R) = \zeta(b, R).\]
\end{prop}

This is mostly remarkable because the values $\overline \wdeg_{a, R}(f)$ and $\wdeg_{a, R}(f)$ could be different, deriving from distinct Laurent series. The result derives from the following:
\begin{prop}
 Let $(K, \abs\cdot)$ be a densely valued non-Archimedean field, $a \in K$, $r > 0$, and consider the non-zero Taylor series \[f(y) = \sum_{n=0}^\infty c_n (y-a)^n.\]
\begin{enumerate}
 \item If $f$ converges on $D = D(a, r)$ with Weierstrass degree $\wdeg_{a, r}(f) = d$, then \[\norm[\zeta(a, r)]f = \abs{c_d}r^d = \sup_{b \in D}\abs{f(b)} = \sup_{b \in D}\norm[b]f = \lim_{\abs{b} \nearrow r}\norm[b]f.\]
 \item If $f$ converges on $E = \CD(a, r)$ 
 with Weierstrass degree $\overline \wdeg_{a, r}(f) = \overline d$, then \[\norm[\zeta(a, r)]f =  \abs{c_{\overline d}}r^{\overline d} = \sup_{b \in E}\abs{f(b)} = \sup_{b \in E}\norm[b]f.\]
 \item If $f$ converges on $D(a, R) \supsetneq \CD(a, r)$, then \[\norm[\zeta(a, r)]f = \abs{c_{\overline d}}r^{\overline d} = \lim_{\abs{b} \searrow r}\norm[b]f.\]
 \item Moreover, $\norm[\zeta(a, r)]f = \abs{f(b)}$ for every $b$ in all but finitely many residue classes of $\CD(a, r)$, to be precise we could pick any $b \in \CD(a, r) \sm (D(a_1, r) \cup \cdots \cup D(a_{\overline d}, r))$, where $(a_j)$ are the finitely many solutions to $f(y) = 0$.
\end{enumerate}
Furthermore, $\zeta(a, R)$ depends only on the choice of closed disk $\CD(a, R)$. That is for any $b \in \CD(a, R)$ we have that \[\zeta(a, R) = \zeta(b, R).\]
\end{prop}
\unfinished{Add proof or move to a later part. Be careful about $z^p - z$. Degree less than res char should be enough. Add an equivalence using radius.}
\unsure{Include e.g. prop 3.20?}

\begin{proof}
 Since $f \in \mathcal L(a, r)$, we have by definition $\norm[\zeta(a, r)]f = \abs{c_d}r^d$. Since $f$ has Weierstrass degree $d = \wdeg_{a, r}(f)$, for $n > d$ we have that $\abs{c_n}r^n \le \abs{c_d}r^d$ and for $n < d$ we have $\abs{c_n}r^n < \abs{c_d}r^d$. By rearranging the former we get that $\abs{c_n}r^{n-d} \le \abs{c_d}$ and hence for any $s < r$ we have $\abs{c_n}s^{n-d} < \abs{c_d}$ which implies $\abs{c_n}s^n < \abs{c_d}s^d$ for every $n > d$. On the other hand, for the finitely many $0 \le n < d$ we can use continuity to find an $s_0 < r$ large enough such that for every $s \in (s_0, r)$ the latter inequality remains true $\abs{c_n}s^n < \abs{c_d}s^d$. 
 Suppose that $b \in D(a, r)$, specifically with $\abs{b-a} = s < r$. Then by \autoref{lem:nonarch:extendedtriangle} \[\abs{f(b)} = \abs{\sum_{n=0}^\infty c_n (b-a)^n} \le \max_n \abs{c_n} \abs{b-a}^n = \max_n \abs{c_n} s^n \le \max_n \abs{c_n} r^n = \abs{c_d}r^d,\] and moreover if $s_0 < s < r$, then $\abs{f(b)} = \abs{c_d}\abs{b-a}^d = \abs{c_d}s^d$, so as $s \nearrow r$, we have $\abs{f(b)} \to  \abs{c_d}r^d$. 
 Items (ii), (iii) can be proven similarly.
 
 It is enough to show (iv) in the completion of the algebraic closure of $K$, since disregarding finitely many disks in a field extension will do the same in $K$. If $\overline{d} = 0$, then by \autoref{lem:nonarch:extendedtriangle} $\abs{f(b)} = \abs{c_0}$ for every $b \in \CD(a, r)$. Otherwise, $\overline{d} = \overline\wdeg_{a, r}(f-c_0)$ and hence $\abs{c_0} \le \norm[\zeta(a, r)]f$. By \autoref{thm:seminorms:mapdisk} we have that $f(\CD(a, r)) = \CD\left(c_0, \norm[\zeta(a, r)]f \right) = \CD\left(0, \norm[\zeta(a, r)]f \right)$, $f(D(b, r)) = D\left(f(b), \norm[\zeta(a, r)]f \right)$, and there are at most $\overline{d}$ such open disks $D(a_1, r), \dots, D(a_{\overline d}, r)$ whose image contains $0$, i.e.\ $f(D(a_j, r)) = D\left(0, \norm[\zeta(a, r)]f \right)$. Hence for any $b$ \emph{not} in such a disk, we find $f(b) \in \CD\left(0, \norm[\zeta(a, r)]f \right) \sm D\left(0, \norm[\zeta(a, r)]f \right)$ and so $\abs{f(b)} = \norm[\zeta(a, r)]f$.
\end{proof}

Observe that if one considers the definition of Type II/III norm $\zeta(a, R$) with $R = 0$ we recover the Type I \emph{semi}norm. We make one last definition of a seminorm; later we will see this is necessary to complete the Berkovich line.

\begin{defn}[Type IV Norm]
Let $(K, \abs\cdot)$ be a non-Archimedean field, and suppose the following nested sequence of disks has empty intersection. \[\CD(a_1, r_1) \supset \CD(a_2, r_2) \supset \CD(a_3, r_3) \supset\CD(a_4, r_4) \supset \cdots\]
Also let  \[\mathcal A(\zeta) = \bigcup_n \overline{\mathcal A}(a_n, r_n).\]
 We define a function $\norm[\zeta]\cdot : \mathcal A(\zeta) \to [0, \infty)$ by
\[\norm[\zeta]f = \inf_{n \ge N} \norm[\zeta(a_n, r_n)]{f}\] where $f \in \overline{\mathcal A}(a_N, r_N) \subset \mathcal A(\zeta)$. We call $\zeta$ the \emph{Type IV norm} associated to the sequence $(\CD(a_n, r_n))$.
\end{defn}
Note that \autoref{prop:seminorms:order} says that the sequence $\norm[\zeta(a_n, r_n)]{f}$ above is decreasing, so this infimum is a limit. Moreover for a fixed $f$, one can show that for large enough $n$, $f \in \overline{\mathcal A}^\times(a_n, r_n)$ is a unit, and so the sequence $\norm[\zeta(a_n, r_n)]{f}$ is eventually constant. \improvement{link to prop}Consider a rational function $f(y)$. We know that $f$ has finitely many poles, so for some large $N$, these poles lie outside of $\CD(a_n, r_n)$, hence $f \in \overline{\mathcal A}(a_n, r_n)$. Therefore $K(y) \subset \mathcal A(\zeta)$ for every Type IV $\zeta$.


\begin{defn}\label{defn:seminorms:typeIVequiv}
 We say the sequences \[\CD(a_1, r_1) \supset \CD(a_2, r_2) \supset \CD(a_3, r_3) \supset\CD(a_4, r_4) \supset \cdots\] and 
 \[\CD(b_1, s_1) \supset \CD(b_2, s_2) \supset \CD(b_3, s_3) \supset\CD(b_4, s_4) \supset \cdots\] 
 are \emph{equivalent} if for any $n \in \N$ we can find an $N \in \N$ such that $\CD(a_n, r_n) \supset \CD(b_N, s_N)$ and vice versa. 
\end{defn}

\begin{prop}\label{prop:seminorms:typeIVmain}
 Type IV norms are non-Archimedean multiplicative norms on $\mathcal A(\zeta) \supset K(y)$. If $\zeta$ and $\zeta'$ are Type IV norms associated to equivalent nested sequences of disks, then $\zeta = \zeta'$.
\end{prop}

\unsure{Include 3.22 about inverses? Maybe need to think about including trivial residue field.}


\subsection{The Berkovich Projective Line}\label{sec:berk}

It is an easy exercise to show that given any ring $A$, a non-Archimedean multiplicative norm $\norm\cdot$ on $A$ extends to a non-Archimedean norm on $\Frac A$ by $\norm{a/b} = \norm a / \norm b$. This provides a simpler proof without annuli of domains of convergence, as in the last section, that the Type II, III, and IV norms automatically extend to $K(y)$. Hence the set of non-Archimedean multiplicative norms on $K[y]$ is the same as those on $K(y)$, and it would be natural to define the Berkovich line simply as the set of such norms on $K(y)$. It would be easy to define the action of a rational map $\phi$ on this space of norms by $\norm[\phi_*(\zeta)]{f(y)} = \norm[\zeta]{f \circ \phi(y)}$, because $f \in K(y) \implies f \circ \phi \in K(y)$. Unfortunately, this construction fails in the case of Type I \emph{seminorms}, since $\norm[a]{f/g}$ is infinite when $g(a) = 0$. One of the important reasons to define the Berkovich projective line $\P^1_\an(K)$ is that we would like a complete and compact space (indeed to compactify $\P^1(K)$), and the Type I points, seen as $\P^1(K) \subset \P^1_\an(K)$ will play a key role as (limiting) endpoints in the resulting tree. Since the Type I points are ill-defined on $K(y)$, we shall define the Berkovich affine line as seminorms over $K[y]$. However the reader is encouraged to bear in mind that the Type II, III, and IV points in the space are norms on $K(y)$, and Type I norms can be treated as a special case.

Throughout the remainder of this section let $K$ denote an algebraically closed field and $\hv$ a complete and algebraically closed field, which in most cases will be the completion of the former $\hv = \widehat K$.

\subsection{The Berkovich Affine Line}\label{sec:affine}

\begin{defn}
 The Berkovich affine line $\A^1_\an = \A^1_\an(K)$ is the set of non-Archimedean multiplicative seminorms on $K[y]$ extending $(K, \abs\cdot)$. (Meaning $\norm{a} = \abs a \quad \forall a \in K$.) A topology is given to $\A^1_\an(K)$ by taking the coarsest topology for which $\norm[\zeta]\cdot \mapsto \norm[\zeta]f$ is continuous for every $f \in K[y]$.
\end{defn}

 We often refer to the elements $\zeta \in \A^1_\an(\hv)$ (or $\P^1_\an(\hv)$) as \emph{points}, but we shall write $\norm[\zeta]\cdot$ when we think of $\zeta$ as a seminorm. It turns out that $\A^1_\an$ is a tree; one way to see this is through its poset structure which we define now and expand upon later.
 
\begin{defn}
 We define a partial order $\preceq$ on $\A^1_\an$ by \[\zeta \preceq \xi \iff \norm[\zeta]f \le \norm[\xi]f \quad \forall f \in \hv[y].\]
\end{defn}

\begin{defn}
Let $\zeta \in \A^1_\an$. Its \emph{absolute value} is \[\abs\zeta = \norm[\zeta]y.\]
 Its \emph{diameter} is \[\diam(\zeta) = \inf_{a \in \hv}\norm[\zeta]{y-a}.\]
\end{defn}

It is clear that diameter exists and $\diam(\zeta) \le \abs\zeta$. Some key examples here are that $\diam(\zeta(a, r)) = r$, $\abs{\zeta(a, r)} = \min \set{\abs a, r}$, $\diam(a) = 0$, and $\abs a = \abs a$.
 
 \begin{defn}
 We define the \emph{open} and \emph{closed Berkovich disks} of radius $r > 0$ centred at $a \in K$, respectively below.
 \[D_\an(a,r):=\set[\zeta \in \A_\an(K)]{\norm[\zeta]{y-a}<r}\]
 \[\CD_\an(a,r):=\set[\zeta \in \A_\an(K)]{\norm[\zeta]{y-a}\le r}\]
 If $r \in |K^\times|$, we say this disk is \emph{rational}, otherwise, we say it is \emph{irrational}.
\end{defn}

There is a natural inclusion of $K = \A^1(K) \subset \A^1_\an(K)$; recall that for every $a\in K$ we have the Type I seminorm $\norm[a]\cdot$. We call these the \emph{classical points} of $\A^1_\an$, again referring to $a$ as a `point' of $\A^1_\an$. Furthermore we can consider the classical disk $D(a, r)$ as a subset of $\A^1_\an$. In the following proposition and throughout, we allow for radii to be zero where $D_\an(a, r) = \emp$ and $\CD_\an(a, r) = \set a$, although these are not `disks'.

\begin{prop}\label{prop:berk:disks}
The following hold.
\begin{itemize}
 \item $D_\an(a, r)$ is open and $\CD_\an(a, r)$ is closed.
 \item $\displaystyle D(a, r) = \A^1(K) \cap D_\an(a, r)$
 \item $\displaystyle \CD(a, r) = \A^1(K) \cap \CD_\an(a, r)$
 \item $\displaystyle D_\an(a, r) \subseteq D_\an(b, s) \iff D(a, r) \subseteq D(b, s)$
 \item $\displaystyle \CD_\an(a, r) \subseteq \CD_\an(b, s) \iff \CD(a, r) \subseteq \CD(b, s)$
 \item $\displaystyle \zeta(a, r) \in D_\an(b, s) \iff \CD(a, r) \subseteq D(b, s)$
 \item $\displaystyle \zeta(a, r) \in \CD_\an(b, s) \iff \CD(a, r) \subseteq \CD(b, s)$
 \item $\displaystyle \zeta(a, r) \succeq \xi\quad \iff \xi \in \CD_\an(a, r)$
 \item $\displaystyle \CD_\an(a, r) = \set{\zeta(a, r)}\ \cup \bigcup_{b\in \CD(a, r)}D_\an(b, r)$.\\ Hence $\zeta(a, r)$ is the unique boundary point of $\CD_\an(a, r)$.
 \item $\displaystyle \diam(\xi) \le r \quad \forall \xi \in \CD_\an(a, r)$ with equality if and only if $\xi = \zeta(a, r)$.
 \item $\displaystyle \zeta \in \CD_\an(a, r) \implies \norm[\zeta]{y-b} = \abs{a-b}$ for every $b \nin \CD(a, r)$.\\ Hence if $\zeta, \xi \in \CD_\an(a, r)$ then $\zeta = \xi$ if and only if $\norm[\zeta]{y-b} = \norm[\xi]{y-b}$ for every $b \in \CD(a, r)$.
\end{itemize}
\end{prop}

We see some clear differences compared to classical disks. For irrational $r \nin \abs{K^\times}$, whereas we had $\CD(a, r) = D(a, r)$, now there is a distinction $\CD_\an(a, r) = \set{\zeta(a, r)} \cup D_\an(a, r)$. Actually it is for these reasons that $\A^1_\an$ is connected, but $\A^1(K)$ was not. Indeed for any $a, r$, $D_\an(a, r)$ is never closed and $\CD_\an(a, r)$ is never open.

In the previous section we laid out four types of seminorm, with some equivalent definitions. The following theorem of Berkovich \cite{Berk} says these are the only four.
 
\begin{thm}[Berkovich's Classification]
 Let $\zeta \in \A^1(\hv)$. Then $\zeta$ is of exactly one of the following four types.\unsure{formatting}
\begin{enumerate}[label=\Roman*:, ref=\theenumi] 
 \item $\norm[\zeta]\cdot = \norm[a]\cdot$ for some unique $a \in \hv$.
\item $\norm[\zeta]\cdot = \norm[\zeta(a, r)]\cdot$ corresponding to a unique rational closed disk $\CD(a, r) \subset \hv$.
\item$\norm[\zeta]\cdot = \norm[\zeta(a, r)]\cdot$ corresponding to a unique irrational disk $\CD(a, r) \subset \hv$.
 \item$\norm[\zeta]\cdot = \lim_{n \to \infty}\norm[\zeta_n]\cdot$ where $\zeta_n = \zeta(a_n, r_n)$ corresponds to a decreasing nested sequence of closed disks $\CD(a_n, r_n)$ with empty intersection in $\hv$. The sequence is unique up to equivalence as in \autoref{defn:seminorms:typeIVequiv}.
\end{enumerate}
\end{thm}

\begin{proof}[Sketch Proof]
 Following \autoref{prop:seminorms:typeII/IIImain} and \autoref{prop:seminorms:typeIVmain} there is little more to say about uniqueness. The trick to classification is the following:
 
\begin{itemize}
 \item Any $\zeta$ is contained in a nested sequence of closed Berkovich disks $\CD_\an(a_n, r_n)$ with $r_n \to \diam(\zeta)$; starting with $\zeta \in \CD_\an(0, r_0)$ for $r_0 = \abs\zeta = \norm[\zeta]{y}$.
 \item $\norm[\zeta]\cdot \le \lim_{n \to \infty}\norm[\zeta_n]\cdot$ where $\zeta_n = \zeta(a_n, r_n)$. We will prove equality.
 \item If the limit of disks is non-empty then $\bigcap_n \CD(a_n, r_n) = \CD(a, r)$ (with $r$ possibly $0$) and therefore $\zeta = \zeta(a, r)$ because $\zeta \in \CD_\an(a, r)$ and has diameter $r$ (by \autoref{prop:berk:disks}).
 \item If $r_n \to 0$ then the limit $\bigcap_n \CD(a_n, r_n)$ is the single point $\set a$ and $\zeta = a$ is Type I.
 \item Otherwise $\zeta = \zeta(a, r)$ is Type II iff $r \in \abs{\hv^\times}$ and Type III iff $r \nin \abs{\hv^\times}$.
 \item If $\bigcap_n \CD(a_n, r_n) = \emp$, then $\zeta$ is the Type IV point $\lim_{n \to \infty} \zeta_n$ associated with this sequence. Indeed, both are in $\CD_\an(a_n, r_n)$ for every $n$ so they agree on $y-b$ for every $b$ outside of $\bigcap_n \CD(a_n, r_n)$, by the last part of \autoref{prop:berk:disks}.
\end{itemize}
\end{proof}
 
 \subsection{The Berkovich Projective Line}
 
 The projective line of a field $\P^1$ is often defined as $\A^1 \cup \set \infty$, or more rigorously as two $\A^1$ affine lines glued over $\A^1 \sm \set 0$ using the transition map $z \mapsto 1/z$. Here we can do either for the Berkovich projective line. In order for this transition map to work need to extend seminorms on $K[y]$ to $K\left[y, y^{-1}\right]$. Fortunately, this is not a big problem. For any $f \in K[y, y^{-1}]$ we can write $f(y) = g(y)/y^d$ with $g \in K[y]$ and some $d \in \N$. Every seminorm $\zeta \in \A^1_\an \sm \set 0$ can be defined on $f$ by \[\norm[\zeta]{f(y)} = \norm[\zeta]{g(y)/y^d} = \norm[\zeta]{g(y)}/\norm[\zeta]{y^d}\] since $\norm[\zeta]{y} \ne 0$. Conversely, any seminorm on $K[y, y^{-1}]$ is in $\A^1_\an \sm \set 0$.
 
\begin{defn}
  We define the \emph{Berkovich projective line} $\P^1_\an = \P^1_\an(K)$ with two charts given by $\A^1_\an(K)$ using the homeomorphism \[\A^1_\an(K)\sm\set{0} \to \A^1_\an(K)\sm\set{0}\] given by \[\norm[\frac 1\zeta]{f(y)} = \norm[\zeta]{f\left(\frac{1}{y}\right)}.\] 
 \end{defn}
 
\begin{rmk}\label{rmk:berk:classicalpoint}
  We will see later that this transition map naturally extends to an involution $\zeta \mapsto 1/\zeta$ of $\P^1_\an$ with the point $\infty$ defined as $1/0$. We also refer to $\infty$ as a (Type I) classical point and part of the natural inclusion $\P^1(K) \subset \P^1_\an(K)$. With this hindsight, we can expand the definition of the classical points $a \in \P^1$ to evaluate elements of $K(y)$ as follows: for any $f \in K(y)$ we define $\norm[a]{f} = \abs{f(a)}$, allowing for $\norm[a]{f} = \infty$ if and only if $f(a) = \infty$. Finally, we can make sense of $\infty = 1/0$: \[\abs{f(\infty)} = \norm[\infty]{f(y)} = \norm[\frac10]{f(y)} = \norm[0]{f\left(\frac{1}{y}\right)}.\]
 For more, see Benedetto \cite[\S6]{Bene}.
\end{rmk}
 
 The most `central' point in $\P^1_\an$ is the Type II point $\zeta(0, 1)$, we call this the \emph{Gauss point}. This is mostly because the ring of integers $\bO_K$ corresponds to $\CD(0, 1)$ and the residue classes $\overline b \in k = \bO_K/\mathcal M_K$ correspond to the open disks $D(b, 1) \subset \CD(0, 1)$. The corresponding Berkovich open disks will play a significant role in analysing maps etc. Often when considering a Type II point $\zeta(a, r)$, it will be useful to find a $\PGL(2, \hv)$ transformation, changing coordinates in $\P^1_\an$, such that $\zeta(a, r)$ is moved to $\zeta(0, 1)$.
 
 \subsection{Berkovich Disks, Affinoids, and Directions}
 
 We further extend our definitions of Berkovich disks and affinoids.

\begin{defn}
 A \emph{Berkovich disk} in $\P^1_\an$ is any one of the following.
\begin{itemize}
 \item A \emph{closed Berkovich disk} $D \subset \P^1_\an$ is either $\CD_\an(a, r) \subset \A^1_\an$ or $\P^1_\an \sm D_\an(a, r)$.
 \item A \emph{open Berkovich disk} $D \subset \P^1_\an$ is either $D_\an(a, r) \subset \A^1_\an$ or $D = \P^1_\an \sm \CD_\an(a, r)$.
 \item A disk is \emph{rational} if $r \in \abs{\hv}$ and \emph{irrational} otherwise.
\end{itemize}
\end{defn}

\begin{defn}
A \emph{connected Berkovich affinoid} is a nonempty intersection of finitely many Berkovich disks $D_1, \dots, D_n$. If all of the disks $D_1, \dots, D_n$ are closed, open, rational, or irrational, then the connected affinoid
$D_1 \cap \cdots \cap D_n$ is respectively said to be closed, open, rational, or irrational.

The connected open affinoid of the form \[\set[\zeta \in \A^1_\an]{r < \norm[\zeta]{y-a} < R} = D_\an(a, R) \sm \CD_\an(a, r)\] is called an \emph{open annulus}. We will often abuse notation and write this as 
\[\set{r < \abs{\zeta-a} < R},\]
distinguished from the classical annulus by the use of the Greek $\zeta$ instead of the Roman $z$.

A \emph{Berkovich affinoid} is a finite union of connected Berkovich affinoids. We may apply the usual adjectives as appropriate.
\end{defn}

\begin{thm}\label{thm:berk:topbasis}\ 
\begin{itemize}
 \item The set of open connected Berkovich affinoids in $\P^1_\an$ forms a basis for the weak topology.
 \item In particular, for any Type II point $\zeta(a, r)$ and open set $U$ containing $\zeta(a, r)$, $U$ 
 contains $D_\an(a, R) \sm (D_1 \cup \cdots \cup D_n)$ where each $D_j$ is a closed Berkovich disk of the form $\CD_\an(b, s) \subsetneq \CD_\an(a, r)$ and $R > r$.
 \item The Berkovich projective line $\P^1_\an(\hv)$ is a connected, complete, compact, Hausdorff space. Moreover, every connected affinoid is connected.
\end{itemize}
\end{thm}

From \autoref{prop:berk:disks} it is clear that for any Type II or III point, $\P^1_\an \sm \zeta(a, r)$ is a union of open Berkovich disks. To be precise \[\P^1_\an \sm \zeta(a, r) = \left(\P^1_\an\sm\CD_\an(a, r)\right) \cup \bigcup_{b\in \CD(a, r)} D_\an(b, r).\]
However we should really write this as a disjoint union. In the special case of the Gauss point, it is somewhat easier to see that we should have one distinct open disk in $\P^1_\an \sm \zeta(a, r)$ for every element of $\P^1(k)$, where $k$ is the residue field.

\begin{defn}\label{defn:berk:dirn}
 Let $\zeta \in \P^1_\an$. The connected components of $\P^1_\an \sm \set\zeta$ are called the \emph{residue classes}, or \emph{directions}, or \emph{tangent vectors} at $\zeta$. The set of directions at $\zeta$ is denoted $\Dir\zeta$. For any $\xi \in \P^1_\an \sm \set\zeta$ we define $\vec v(\xi)$ to be the (unique) direction at $\zeta$ containing $\xi$.
\end{defn}

\begin{prop}\label{prop:berk:dirn}
 Let $\zeta \in \P^1_\an$. Then $\Dir\zeta$ can be described by
\begin{enumerate}
 \item If $\zeta$ is Type I or IV, then there is only one direction at $\zeta$, meaning $\zeta$ is an endpoint of the tree.
 \item If $\zeta = \zeta(a, r)$ is Type II, then $\Dir\zeta \cong \P^1(k)$, meaning $\zeta$ has one direction for each distinct open disk $D_\an(b, r)$ for $b \in \CD(a, r)$, and also $\P^1_\an \sm \CD(a, r)$, the residue class associated with $\infty$.
 \item If $\zeta = \zeta(a, r)$ is Type III, then the two directions are $D_\an(a, r)$ and $\P^1_\an \sm \CD(a, r)$.
\end{enumerate}
\end{prop}

\subsection{Paths and Hyperbolic Metric}

Recall from \autoref{sec:affine} the definition of the partial order $\preceq$ on $\A^1_\an$; we naturally extend this to $\P^1_\an$ by asserting that $\zeta \preceq \infty$ for every $\zeta \in \A^1_\an$.

\begin{prop}\ 
\begin{enumerate}
 \item The relation $\preceq$ defines a partial order on $\P^1_\an$.
\item All Type I and IV points are minimal with respect to $\preceq$.
\item $\infty$ is a maximum point.
\item For any $\zeta, \xi \in \A^1_\an$ with $\xi \preceq \zeta$, we have $\diam(\xi) \le \diam(\zeta)$, with equality if and only if $\xi = \zeta$.
\item For any two $\zeta_0, \zeta_1 \in \P^1_\an$ there is a unique least upper bound, $\zeta_0 \vee \zeta_1$, defined below.
\end{enumerate}
\end{prop}

\begin{defn}
 Let $\zeta_0, \zeta_1 \in \P^1_\an$. The \emph{least upper bound} or \emph{join} of $\zeta_0$ and $\zeta_1$, denoted $\zeta_0 \vee \zeta_1$, is the unique element of $\P^1_\an$ such that:
\begin{enumerate}
 \item $\zeta_0, \zeta_1 \preceq \zeta_0 \vee \zeta_1$; and
 \item if $\xi \in \A^1_\an$ and $\zeta_0, \zeta_1 \preceq \xi$, then $\zeta_0 \vee \zeta_1 \preceq \xi$.
\end{enumerate}
\end{defn}

\unfinished{cite}

\begin{defn}
Let $X$ be a topological space. We say $X$ is uniquely path-connected iff for any two distinct points $x_0, x_1 \in X$, there is a unique subset $I \subset X$ containing $x_0$ and $x_1$ such that $I$ is homeomorphic to the real closed interval $[0, 1]$, with the homeomorphism mapping $x_0$ to $0$ and $x_1$ to $1$. We call $[x_0, x_1] = I$ the \emph{closed interval} between $x_0$ and $x_1$. We call, $(x_0, x_1) = [x_0, x_1] \sm \set{x_0, x_1}$ an \emph{open interval}. We can similarly define the half-open intervals $[x_0, x_1)$, and $(x_0, x_1]$.
\end{defn}

\begin{thm}\label{thm:berk:interval}
  Let $U \subseteq \P^1_\an$ be a connected Berkovich affinoid, then $U$ is uniquely path-connected. Hence $\P^1_\an$ is locally connected. Moreover, for any $\zeta_0, \zeta_1 \in U$, all points on $(\zeta_0, \zeta_1)$, are of Type II or III. A path $[\zeta_0, \zeta_1]$ is always of the form \[[\zeta_0, \zeta_0 \vee \zeta_1] \cup [\zeta_0 \vee \zeta_1, \zeta_1] = \set{\zeta_0 \preceq \xi \preceq \zeta_0 \vee \zeta_1} \cup \set{\zeta_0 \vee \zeta_1 \preceq \xi \preceq \zeta_1}.\]
\end{thm}

\begin{defn}
 Let $S \subseteq \P^1_\an$ be any subset of the Berkovich projective line. The \emph{convex hull} of $S$ is the set 
 \[\Hull(S) = \set[\xi \in \P^1_\an]{\exists \zeta_0, \zeta_1 \in S \st \xi \in [\zeta_0, \zeta_1]}.\]
\end{defn}

\begin{defn}
The set $\bH = \P^1_\an(K) \sm \P^1(K)$ is the \emph{hyperbolic space} over $K$. We define a \emph{hyperbolic metric} $d_\bH : \bH \times \bH \to [0, \infty)$ given by \[d_\bH(\zeta, \xi) = 2 \log \left(\diam(\zeta \vee \xi)\right) - \log \left(\diam(\zeta)\right) - \log \left(\diam(\xi)\right).\]
\end{defn}

\begin{rmk}
 The hyperbolic metric measures distances by the logarithm of diameter along the lines of the poset structure. Observe that when $\zeta \preceq \xi$ then $\zeta \vee \xi = \xi$, so \[d_\bH(\zeta, \xi) =  \log \left(\diam(\zeta)\right) - \log \left(\diam(\xi)\right).\] Then in general, since $[\zeta, \xi] = 
 [\zeta, \zeta \vee \xi] \cup [\zeta \vee \xi, \xi]$ by \autoref{thm:berk:interval}, it is natural to see that 
 \[d_\bH(\zeta, \xi) = d_\bH(\zeta, \zeta \vee \xi) + d_\bH(\zeta \vee \xi, \xi).\]
 
 The topology given by $d_\bH$ is much stronger than the weak topology, even though they agree on intervals. For instance, a hyperbolic ball around a Type II point does not contain a single direction, but every (weak) open neighbourhood contains all but finitely many.
\end{rmk}

\subsection{Rational Maps}\label{sec:ratl}

Recall that the aim of this part of this article is to generalise the notion of \emph{rational maps} on the Berkovich projective line and also its dynamical theory. When we define \emph{skew products} on $\P^1_\an$ it will not only be useful to refer and compare them to rational maps, but we will want to understand skew products through their associated rational maps.

The purpose of this section is to recall some of the definitions and basic theory of \emph{rational maps} on the Berkovich projective line. However much of this theory will be deferred to the following sections (and referenced) as we generalise these results to skew products, since the reader can always recover the original theorem as a special case of the newly printed one.

\emph{Notation.} Typically authors write $\phi$ for both a rational function in $\hv(y)$ and also the induced function $\phi : \P^1_\an(\hv) \to \P^1_\an(\hv)$. In this case we will instead distinguish $\phi_*$ as the induced function on the Berkovich line and $\phi^*$ as the homomorphism induced on the function field $\hv(y)$. By now the reader will also have noticed the propensity to write rational functions as $f(y)$ rather than $f(x)$, $f(t)$, or $f(z)$; whilst the latter would have been fine, we preserve the variable $x$ for the variable in the field of Puiseux series $\K$, giving a natural extension from the field $\C(x, y)$ to $\K(y)$.

For the following definition, recall \autoref{rmk:berk:classicalpoint} about considering a Type I point as an honorary $[0, \infty]$-valued seminorm on $K(y)$. Indeed, for any $a \in \P^1(K) \subset \P^1_\an(K)$ one can check that $\phi_*(a) = \phi(a)$.
\unfinished{cite Benedetto}

\begin{defn}\label{defn:ratl}
 Let $\phi(y) \in \hv(y)$ be a rational function. Then we define the associated \emph{rational map} on the Berkovich projective line by
\begin{align*}
 \phi_* : \P^1_\an(\hv) &\longrightarrow \P^1_\an(\hv)\\
 \zeta &\longmapsto \phi_*(\zeta)\\
 \text{where } \norm[\phi_*(\zeta)]{f(y)} &= \norm[\zeta]{f \circ \phi(y)}.
\end{align*}
\end{defn}

Importantly, this function is well defined because either $\norm[\phi_*(\zeta)]\cdot$ is a seminorm on $K[y]$ which extends the one on $K$. 
To see the latter, consider $\norm[\phi_*(\zeta)]a$ for $a \in K$. Applying $\phi$ to a constant does nothing i.e. $a \circ \phi(y) = a$ and so $\norm[\phi_*(\zeta)]{a} = \norm[\zeta]{a \circ \phi} = \norm[\zeta]{a} = \abs a$. This will be the main challenge in making a more general definition later.

\begin{thm}
 Let $\phi \in \hv(y)$ be a rational function. Then the function $\phi_* : \P^1_\an \to \P^1_\an$ of \autoref{defn:ratl} is the unique continuous extension of the rational function $\phi : \P^1 \to \P^1$ to $\P^1_\an$.
\end{thm}

\begin{prop}
 Let $\phi, \psi \in \hv(y)$ be rational functions, then $(\psi \circ \phi)_* = \psi_* \circ \phi_*$.
\end{prop}

\begin{thm}
 Let $\phi \in \hv(y)$ be a non-constant rational function.
\begin{enumerate}
 \item Suppose $D(a, r)$ contains no poles of $\phi$ and $\phi(D(a, r)) = D(b, s)$, then $\phi_*(\zeta(a, r)) = \zeta(b, s)$.
 \item $\phi$ preserves the types of points.
 \item $\phi$ is an open mapping.
\end{enumerate}
\end{thm}

\subsection{Reduction}\label{sec:berk:red}

As suggested by \autoref{defn:berk:dirn} and \autoref{prop:berk:dirn}, the most natural way to think about the directions at a Type II point $\zeta$ is by identifying each one with a residue in $\P^1(\k)$. In this subsection we shall discuss reduction of elements and maps from $K$ or $K(y)$ to $\k$ or $\k(y)$, and what we learn about local degrees. This will generalise the content of \cite[\S7.5]{Bene}. \unfinished{HOW?}

Recall that the residue field of $K$ is the quotient field $k = \bO_K / \mathcal M_K$. The quotient map $\bO_K \to k$ is called the \emph{reduction map}. We denote the reduction of $a \in \bO_K$ by $\overline a$. This has a more useful extension to the projective line $\P^1(K)$ since every element can be written as $[a_0 : a_1]$ with $a_0, a_1 \in \bO_K$.
\begin{align*}
\P^1(K) &\longrightarrow \P^1(k)\\
[a_0 : a_1] &\longmapsto [\overline a_0 : \overline a_1]\\
\end{align*}
Furthermore, this induces a reduction map on $\bO_K[y] \to k[y]$ by reducing each coefficient in the polynomial. Reduction of rational functions of $K$ is a little troublesome. One can write any $f \in \P^1(K(y))$ as a fraction $f=g/h$ or ratio $[g : h]$ of polynomials in $\bO_K[y]$, with at least one coefficient in $g$ or $h$ having absolute value $1$. We then define the reduction $\overline f \in \P^1(\k(y))$ as follows:
\begin{align*}
\P^1(K(y)) &\longrightarrow \P^1(\k(y))\\
f = [g : h] &\longmapsto [\overline g : \overline h] = \overline f\\
\end{align*}
\emph{Warning}: this definition is sensitive to the choice of $f$ and $g$ - if one allows both $g, h \in \mathcal M_K$ then the reduction will be ill defined as $[0 : 0]$.

Unfortunately, reduction can change the basic properties of a polynomial.
\begin{ex}
Three examples of reduction over the complex numbers.
\begin{enumerate}
 \item Let $g(x, y) \in \C(x)[y]$ be defined by $g = xy^2 + y - 1$. Then $\overline g = y-1$ so $\deg(g) \ne \deg(\overline g)$.
 \item Let $g(x, y) = (y-x)y$, then $\overline g = y^2$. We see that $g$ had two distinct roots, but its reduction has one despite having the same degree.
\item Let $f = (y-x)/y$ then $\overline f = 1$.
\end{enumerate}
\end{ex}

A rational function $\phi(y) \in K(y)$ induces a rational map.
\begin{align*}
\overline \phi : \P^1(\k) &\longrightarrow \P^1(\k)\\
\overline a &\longmapsto \overline \phi(\overline a)\\
\end{align*}

\begin{defn}
 Suppose $K$ is a non-Archimedean field which might not be algebraically closed, and let $\phi \in K(y)$. We say that $\phi(y) = \frac{g(y)}{h(y)}$ has \emph{explicit good reduction} iff $\deg(\overline \phi) = \deg(\phi)$, otherwise $\phi$ has \emph{bad reduction}. If there is a fractional linear transformation $\eta \in \PGL(2, K)$ such that $\eta\circ \phi \circ \eta^{-1}$ has explicit good reduction then we say $\phi$ has \emph{good reduction}. If instead there is such an $\eta \in \PGL(2, \overline K)$ then $\phi$ has \emph{potentially good reduction}. 
\end{defn}

If $\phi$ has explicit good reduction then $\overline a \mapsto \overline{\phi(a)}$ is well defined and equal to $\overline \phi(\overline a)$. Conversely, if the degree drops then we can find $a, b \in K$ with $\phi(a) = 0$, $\phi(b) = \infty$ and the same reduction $\overline a = \overline b$, thus $\overline 0 = \overline{\phi(a)} \ne \overline{\phi(b)} = \overline \infty$; moreover both may be distinct from $\overline\phi(\overline a)$.

For a much more thorough discussion of reduction, see \cite[\S4.3]{Bene}.

\unfinished{Definitely include more in this section.}


\section{Skew Products on the Berkovich Projective Line}\label{sec:skew}
\improvement{Find a better section title?}
 
 The aim of this section is to define a \emph{skew product} on the Berkovich projective line and compare it to a rational map. Whilst these maps have new and unusual quirks and are a strict generalisation of Berkovich rational maps, most of the properties we are used to will be recovered.
 
 \begin{defn}
  Let $\Psi$ be an endomorphism of $K(y)$ extending an automorphism of $K$, i.e. the following diagram commutes:
  \[
\begin{tikzcd}
 K(y) & K(y) \arrow[swap]{l}{\Psi} \\
K \arrow[hook]{u} & K \arrow[hook]{u} \arrow{l}{\Psi_1}
\end{tikzcd}
\]
In this case we will call $\Psi : K(y) \to K(y)$ a \emph{skew endomorphism} of $K(y)$. We will typically denote the restriction $\left.\Psi\right|_K$ by $\Psi_1$.
\end{defn}
 
 \subsection{Motivation}
 
 Often, we shall think of $\Psi$ coming from a rational map of ruled surfaces. We will give detail on this construction in a sequel, but describe a special case of the situation now to give examples and motivation. Classically, a skew product (in analysis and geometry) is one of the form \[\phi(x, y) = (\phi_1(x), \phi_2(x, y))\] defined on some product space $A \times B$. Let us focus on the simple case of $\P^1 \times \P^1$ over the field $\k$. One may think of the following diagram commuting with a first projection map $h(x, y) = x$; this will help to generalise the concept later.
\[
\begin{tikzcd}
 \P^1 \times \P^1 \arrow[swap]{d}{h} \arrow[dashed]{r}{\phi} & \P^1 \times \P^1 \arrow{d}{h}\\
 \P^1 \arrow[swap]{r}{\phi_1} & \P^1
\end{tikzcd}
\]
The information given by $\phi$ is equivalent to a $\k$-algebra homomorphism of function fields.
\begin{align*}
 \phi^* : \k(x, y) & \longrightarrow \k(x, y)\\
x & \longmapsto \phi_1(x)\\
y & \longmapsto \phi_2(x, y)
\end{align*}
After changing coordinates we may assume that $\phi_1(0) = 0$ and look in a neighbourhood of $x=0$ then we obtain a $\k$-algebra map $\phi_1^* : \k[[x]] \to \k[[x]]$ which extends to one of the local function field $\phi^* : \k((x))(y) \to \k((x))(y)$. In more algebraic terminology, we took the completion of the local ring $\k[x]_{(x)}$.
  \[
\begin{tikzcd}
 \k((x))(y) & \k((x))(y) \arrow[swap]{l}{\phi^*} \\
\k((x)) \arrow[hook]{u}{h^*} & \k((x)) \arrow[hook, swap]{u}{h^*} \arrow{l}{\phi_1^*}
\end{tikzcd}
\]
After taking the algebraic closure of $\k((x))$ to obtain the Puiseux series $\K(\k)$, this map can be extended to a $\k$-algebra endomorphism $\phi^* : \K(y) \to \K(y)$.  We can write $\phi_1^*(x) = \phi_1(x) \in \k[[x]]$ with $\phi_1(x) = \lambda x^n + \bO(x^{n+1})$ and $\lambda \in \k^\times$ then this extends to an `\equivar{}' skew endomorphism over $\K$. With $\phi_2 \in \k((x))(y)$ we call such a map a \emph{$\k$-rational skew endomorphism}.

If $\phi_1(x)$ were the identity, then $\phi$ would represent the rational map $y \mapsto \phi_2(y)$ over $\k((x))$ and naturally induce a Berkovich rational map on $\P^1_\an(\K)$. Unfortunately, $\phi_1$ is rarely trivial, and $\phi$ will not translate to a Berkovich rational map. A different mapping on the Berkovich projective line is needed - \emph{the skew product.}




\subsection{The Problem}

If we had that $\Psi_1 = \id$ then $\Psi$ would be a $K$\emph{-algebra endomorphism} of $K(y)$, indeed it would be the \emph{rational map} $y \mapsto \Psi(y)$ over $\P^1(K)$. We could then define a Berkovich rational map $\Psi_* : \P^1_\an \to \P^1_\an$, as in \autoref{sec:ratl}, by 
\[\norm[\Psi_*(\zeta)]{f} = \norm[\zeta]{\Psi(f)}.\]
 The crucial calculation for the skew product to be well-defined was that $\Psi_*(\zeta)$ preserves the norm on $K$, meaning that for every $a \in K$, $\norm[\Psi_*(\zeta)]{a} = \abs a$. We might na\"ively try this definition with an arbitrary skew endomorphism; let $a \in K$, then by the expected definition we have
\[\norm[\Psi_*(\zeta)]{a} = \norm[\zeta]{\Psi(a)} = \abs{\Psi(a)} = \abs{\Psi_1(a)}.\]
In general, unlike the rational case, $\Psi_1$ is an arbitrary field automorphism of $K$ and could do anything to the absolute value. In the above we need $\abs{\Psi_1(a)} = \abs{a}$. Requiring this is reasonable, but for the definition of skew product below, we only ask that $\abs{\Psi_1(a)} = \abs{a}^{\frac 1\q}$ uniformly for some $\q>0$. The special cases where $1/\q \in \N$, especially $\q = 1$, will be of great interest in applications of this theory.

The construction of an arbitrary algebraically defined map on a Berkovich space is not new; see for instance \cite{FJ04}. Specifically, one can always normalise $\norm[\zeta]{\Psi(f)}$ as necessary depending on both $\Psi$ and $f$, to ensure a well-defined function. In any case, this will be a geometrically natural definition because, roughly speaking, the corresponding prime ideal of $\zeta$ i.e.\ $p_\zeta = \set{f : \norm[\zeta]f < 1}$ is mapped to the corresponding prime of its image \[\Psi^{-1}(p_\zeta) = \set{g : \Psi(g) \in p_\zeta} = \set{g : \norm[\zeta]{\Psi(g)} < 1} = \set{g : \norm[\Psi_*(\zeta)]g < 1}.\] Our construction is arbitrary enough to be applicable to a broader class of examples in complex dynamics, but the uniform normalisation factor allows for the dynamical behaviour of $\Psi_*$ to be better understood.

\unfinished{Mention other constructions e.g. Favre-Jonsson}

\subsection{Skew Products}

\unfinished{Talk about how this is geometrically natural}

\begin{defn}[Skew Product]\label{defn:skew:skew}
 Suppose that $\Psi : K(y) \to K(y)$ is a skew endomorphism of $K(y)$ and there is a $\q$ such that $\abs{\Psi(a)} = \abs{\Psi_1(a)} = \abs{a}^{\frac 1\q}$ for every $a \in K$. Then we say $\Psi$ is \emph{\equivar} with \emph{scale factor} $\q$. 
 Given such a $\Psi$, we define $\Psi_*$, a \emph{skew product over $K$}, as follows.
\begin{align*}
 \Psi_* : \P^1_\an(K) &\longrightarrow \P^1_\an(K)\\
 \zeta &\longmapsto \Psi_*(\zeta)\\
 \text{where } \norm[\Psi_*(\zeta)]{f} &= \norm[\zeta]{\Psi(f)}^\q
\end{align*}
If $\q=1$ then we call $\Psi_*$ a \emph{simple} skew product. Otherwise, if $\q < 1$ we say it is \emph{superattracting}, and if $\q > 1$ we may say it is \emph{superrepelling}.
\end{defn}

Consider the case deriving from a skew product $\phi(x, y) = (\phi_1(x), \phi_2(x, y))$ on a surface, with $\phi_1(x) = \lambda x^n + \bO(x^{n+1})$, and recall the above discussion that we have a $\k$-rational skew endomorphism $\phi^*$. Then the induced skew product on the Berkovich projective line, $\phi_* : \P^1_\an(\K) \to \P^1_\an(\K)$ will be called a \emph{$\k$-rational skew product} and has scale factor $\q = \frac 1n$. In particular, if $n=1$ then $\phi_*$ is a simple $\k$-rational skew product; this name takes after the fact $x=0$ is a simple zero of $\phi_1(x)$. Furthermore, note that $0$ is a superattracting fixed point of $\phi_1(x)$ when $n > 1$, and hence why we call $\phi_*$ `superattracting'. A deeper discussion of $\k$-rational skew products will be provided in the sequel, but we will continue to use them for examples and explain basic constructions.

\begin{thm}
 Suppose that $\Psi$ is an \equivar{} skew endomorphism. Then the skew product $\Psi_* : \P^1_\an \to \P^1_\an$ is a well defined map on the Berkovich projective line.
\end{thm}

\begin{proof}
It is clear that $\Psi_*(\zeta)$ is a multiplicative seminorm because $\Psi$ is a ring homomorphism. \unfinished{expand?}We need to check that $\Psi_*(\zeta)$ extends the norm on $K$. Indeed given $a \in K$,
\[\norm[\Psi_*(\zeta)]{a} = \norm[\zeta]{\Psi(a)}^\q = \abs{\Psi(a)}^\q = \left(\abs a^{\frac 1\q}\right)^\q = \abs a.\]
\end{proof}

\begin{prop}\label{prop:lowerstarfunctorial}
 If $\Phi, \Psi$ are \equivar{} skew endomorphisms of $K(y)$ then $(\Phi \circ \Psi)_* = \Psi_* \circ \Phi _*$ i.e.\ $( \cdot )_*$ is a contravariant functor.
\end{prop}

\begin{proof}
\[\norm[\Psi_* \circ \Phi_*(\zeta)]{f} =  \norm[\Phi_*(\zeta)]{\Psi(f)} =  \norm[\zeta]{\Phi(\Psi(f))} = \norm[(\Phi \circ \Psi)_*(\zeta)]f\]
\end{proof}

The following definition and theorem is fundamental to working with skew products - it says we can always decompose a skew product into a field automorphism and a rational map. On $\P^1_\an$ this will help because the former induces a bijection with some nice geometric properties and the latter induces a Berkovich rational map which is well understood.

\begin{defn}\label{defn:skew:split}
 Let $\Psi$ be a skew endomorphism. We define the following two homomorphisms. Firstly $\Psi_1 = \left.\Psi\right|_K$ but extended trivially to $K(y)$, and secondly with $\Psi_2$ we distill the action of $\Psi$ on $y$.
 \begin{align*}
 \Psi_1 : K(y) &\longrightarrow K(y)\\
  a &\longmapsto \Psi(a) \quad \forall a \in K\\
  y &\longmapsto y\\
  \ &\ \\
 \Psi_2 : K(y) &\longrightarrow K(y)\\
  a &\longmapsto a \qquad\ \ \forall a \in K\\
  y &\longmapsto \Psi(y)
\end{align*}
\end{defn}

\begin{thm}\label{thm:skew:comp}
 Let $\Psi$ be an \equivar{} skew endomorphism. Then
\begin{itemize}
 \item $\Psi = \Psi_2 \circ \Psi_1$;
 \item $\Psi_* = \Psi_{1*} \circ \Psi_{2*}$;
 \item $\Psi_{2*}$ is a rational map on $\P^1_\an$.
\end{itemize}
\end{thm}

\begin{proof}
 Clearly $\Psi_1, \Psi_2$ are ring homomorphisms, so it is enough to show $\Psi = \Psi_2 \circ \Psi_1$ on generators of $K(y)$, or more simply on an arbitrary $a \in K$ and on $y$. \[ \Psi_2 \circ \Psi_1(y) = \Psi_2(y) = \Psi(y)\]
 The last equality is by the definition of $\Psi_2$.
  \[ \Psi_2 \circ \Psi_1(a) = \Psi_2(\Psi(a)) = \Psi(a)\]
  The last equality is because $\Psi_2(b) = b$ for any $b \in K$, such as $b = \Psi(a)$. 
 We finish the proof using \autoref{prop:lowerstarfunctorial}.
\end{proof}

Our notation is inspired by the case of a $\k$-rational skew product.
\begin{align*}
 \phi : \P^1\times \P^1 & \dashto \P^1\times \P^1\\
  (x, y) &\longmapsto (\phi_1(x), \phi_2(x, y))
\end{align*}
In this case $\Psi = \phi^*$ and we consider the induced skew product. From the original geometric perspective, $\phi_1(x) = \phi^*(x)$ and $\phi_2(x, y) = \phi^*(y)$. Then the decomposition of $\Psi = \phi^*$ is very natural since it separetes into its into its actions on $x$ and $y$:
\begin{align*}
\Psi_1 : k(x, y) &\longrightarrow k(x, y) &  \Psi_2 : k(x, y) &\longrightarrow k(x, y)\\
  x &\longmapsto \phi_1(x) &  x &\longmapsto x\\
  y &\longmapsto y&   y &\longmapsto \phi_2(x, y)\\
\end{align*}
We write $\phi_1^* = \Psi_1$ and $\phi_2^* = \Psi_2$. One may verify that $\phi^* = \phi_2^* \circ \phi_1^*$ and $\phi_* = \phi_{1*} \circ \phi_{2*}$, but antecedent to \autoref{thm:skew:comp} is the (set theoretic) composition
\[\phi(x, y) = (\phi_1(x), \phi_2(x, y)) = (\phi_1(x), y) \circ (x, \phi_2(x, y)).\]

From now on, we may denote an \equivar{} skew endomorphism by $\phi^*$ even if it is not derived from a geometric skew product on a surface. To be clear, to say $\phi_*$ is a skew product still will mean it derives from an \equivar{} skew endomorphism $\phi^*$, but it may not be $\k$-rational. Furthermore, we may write $\phi_*= \phi_{1*} \circ \phi_{2*}$ as a cue to the splitting from \autoref{defn:skew:split}.

\begin{rmk}
 We can see now that not only does a skew product generalise the definition of a rational map, but every skew product is a composition of a rational map and the action of a field automorphism. This will be most useful for our understanding of how a skew product acts in one iterate. It will be much harder to understand multiple iterations (its dynamics), however the decomposition is still valuable. \unfinished{See result about it not being rational or conjugate to one.}
\end{rmk}

\subsection{Properties of Skew Products}

The following theorem say that given $\phi_* = \phi_{1*} \circ \phi_{2*}$, the field automorphism part, $\phi_*$ is a geometrically nice map on $\P^1_\an$. However caution is warranted: a non-trivial field automorphism will induce a highly non-analytic map.

\begin{thm}\label{thm:skew:scaledhomeo}
Let $\Psi$ be an \equivar{} automorphism of $K$ extended trivially to $K(y)$ with scale factor $\q$, i.e. $\Psi_2 = \id$. Then the induced skew product $\Psi_* : \P^1_\an \to \P^1_\an$  
\begin{enumerate}
 \item is a homeomorphism on $\P^1_\an(K)$;
 \item scales hyperbolic distances by a factor of $\q$;
 \item is the unique continuous extension of $\Psi^{-1}$ on $\P^1(K) \subset \P^1_\an(K)$;
 \item maps Berkovich points to those of the same type; and
 \item is order preserving on the poset $(\P^1_\an, \preceq)$.
\end{enumerate}
In particular:\unsure{take off this first line?}
\begin{align*}
\Psi_*(a) &= \Psi^{-1}(a)\\
 \Psi_*(D(a, r)) &= D(\Psi^{-1}(a), r^\q)\\
 \Psi_*\left(\CD(a, r)\right) &= \CD(\Psi^{-1}(a), r^\q)\\
\Psi_*(\zeta(a, r)) &= \zeta(\Psi^{-1}(a), r^\q)\\
 \Psi_*(D_\an(a, r)) &= D_\an(\Psi^{-1}(a), r^\q)\\
 \Psi_*\left(\CD_\an(a, r)\right) &= \CD_\an(\Psi^{-1}(a), r^\q)
\end{align*}
\end{thm}
 
\begin{rmk}
 The fact that this skew product $\phi_{1*}$ extends the \emph{inverse} of the homomorphism $\phi_1^*$ on $\P^1(K)$ is somewhat more natural in the geometric setting over the Puiseux series. Here we see it as a `pre-composition' of functions of $x$, where $\gamma(x) \mapsto \gamma(\phi_1^{-1}(x))$. 
 This is very different from how homomorphisms on $K$ acting only on $y$ (`rational maps') generate a post-composition-like function on $\P^1(K)$.
 
 Consider a germ of a curve through $x=0$, namely $x \mapsto (x, \gamma(x))$. Then $\phi = (\phi_1, \phi_2)$ applied to this gives $x \mapsto (\phi_1(x), \phi_2(x, \gamma(x))$. To rewrite this in the form $(x, \tilde\gamma(x))$ we must \emph{precompose} with $\phi_1^{-1}(x)$ to get \[x \mapsto (x, \phi_2(\phi_1^{-1} (x), \gamma(\phi_1^{-1}(x)))).\]
\end{rmk}

\begin{prop}
Let $\phi_*= \phi_{1*} \circ \phi_{2*}$ be a simple $\k$-rational skew product. Then $\phi_* : \P^1(\K) \to \P^1(\K)$ acts as follows
\begin{align*}
 \phi_* : \P^1(\K) &\longrightarrow \P^1(\K)\\
  a(x) &\longmapsto \phi_2(\phi_1^{-1}(x), a(\phi_1^{-1}(x))) = \phi_2\circ (\id \times a) \circ \phi_1^{-1} (x)
\end{align*}
\end{prop}

\begin{proof}[Proof of \autoref{thm:skew:scaledhomeo}]
 %
Since $\Psi$ is an isomorphism, $\norm[(\Psi^{-1})_*(\zeta)]f = \norm[\zeta]{\Psi^{-1}(f)}$ provides an inverse to $\Psi_*$.

First we prove that the restriction of $\Psi_*$ to classical points is equal to $\Psi^{-1}$ on $\P^1(K)$. Let $\zeta = a \in K$ be Type I, then
\[\norm[\Psi_*(\zeta)]{y-b} = \norm[\zeta]{y-\Psi(b)}^\q = \abs{a - \Psi(b)}^\q = \abs{\Psi(\Psi^{-1}(a) - b)}^\q\] \[= \left(\abs{\Psi^{-1}(a) - b}^{\frac 1\q}\right)^\q = \abs{\Psi^{-1}(a) - b} = \norm[\Psi^{-1}(a)]{y-b}.\]
It is a similar exercise in the definitions to prove that $\psi_*(\infty) = \infty$.

Now observe that since $|\Psi_*(a) - \Psi_*(b)| = |\Psi^{-1}(a - b)| = |a-b|^\q$ we have:
\begin{align*}
 \Psi_*(D(a, r)) &= D(\Psi^{-1}(a), r^\q)\\
 \Psi_*(\CD(a, r)) &= \CD(\Psi^{-1}(a), r^\q)
\end{align*}
It follows that \[\Psi_*(\zeta(a, r)) = \zeta(\Psi^{-1}(a), r^\q)\]
since for $f \in K[y]$ \[\norm[\zeta(a, r)]{\Psi(f)} = \sup_{b\in D(a, r)} \norm[b]{\Psi(f)} = \sup_{b\in D(a, r)} \norm[\Psi_*(b)]{f} = \sup_{b' \in D(\Psi^{-1}(a), r^\q)} \!\!\!\norm[b']{f} = \norm[\zeta(\Psi^{-1}(a), r^\q)]{f}\]
 Similarly, since $\norm[\Psi_*(\zeta)]{y-\Psi^{-1}(a)} = \norm[\zeta]{y-a}^\q$ we have 
  \begin{align*}
  \Psi_*(D_\an(a, r)) &= D_\an(\Psi^{-1}(a), r^\q)\\
 \Psi_*(\CD_\an(a, r)) &= \CD_\an(\Psi^{-1}(a), r^\q)
\end{align*}
This determines the images of Type II/III points. Note that $\abs a = r \iff \abs{\Psi^{-1}(a)} = r^\q$, hence $r \in \abs{K^\times} \iff r^\q \in \abs{K^\times}$. Thus Type II and III points are individually preserved.

We have also implicitly shown that for disks $D$ and $E$ we have \[D \subseteq E \iff \Psi_*(D) \subseteq \Psi_*(E),\] and similar for their Berkovich versions $D_\an$ and $E_\an$ this is equivalent to \[D_\an \subseteq E_\an \iff \Psi_*(D_\an) \subseteq \Psi_*(E_\an).\]

This shows that a nested sequence of disks $D_1 \supsetneq D_2 \supsetneq D_3 \supsetneq \cdots$ remains nested; it is also clear that this has empty intersection if and only if the sequence of images do. Therefore Type IV points are preserved since the images of Type IV points can be described by the image of such a sequence of disks.

 To see that $\Psi_*$ is order preserving, recall that $\zeta \prec \xi$ implies that $\xi = \zeta(a, R)$ or $\xi = \infty$. Since $\infty = \Psi_*(\infty)$ is a maximum, the latter case is trivial for all $\zeta$. In the case that $\xi = \zeta(a, R)$, we know that $\zeta \in \CD_\an(a, R)$. \[\Psi_*(\zeta) \in \Psi_*(\CD_\an(a, R)) = \CD_\an(\Psi^{-1}(a), R^\q),\] therefore 
\[\zeta \preceq \zeta(a, R) \iff \Psi_*(\zeta) \preceq \Psi_*(\zeta(a, R)).\]

 We have just shown that $\Psi_*$ preserves types, the ordering $\preceq$, and also that \[\diam(\Psi_*(\zeta)) = \diam(\zeta)^\q.\]

Since the basis of the topology if given by open affinoids, and these are finite intersections of disks, to show continuity, we only need to look at preimages of disks. Hence we know from the above that $\Psi_*$ is an open map. Since $\Psi_*$ has an inverse given by $(\Psi^{-1})_*$, we also have
 \[ \Psi_*^{-1}(D_\an(\gamma, r)) = D_\an(\Psi(\gamma), r^{\frac 1\q}),\] 
 \[ \Psi_*^{-1}(\bar D_\an(\gamma, r)) = \bar D_\an(\Psi(\gamma), r^{\frac 1\q}),\]
proving the continuity. Thus $\Psi_*$ is a homeomorphism.


Since $\Psi_*$ is order preserving, to show the scaling or isometry on the hyperbolic metric we begin with pairs of related points. Suppose that  $\zeta \prec \xi$ have diameter $r$ and $R$ respectively. Then $\Psi_*(\zeta) \prec \Psi_*(\xi)$ and hence
\[d_\bH(\Psi_*(\zeta), \Psi_*(\xi)) = \log(R^\q)-\log(r^\q) = \q\left (\log(R)-\log(r)\right) = \q\,d_\bH(\zeta, \xi).\]
 Otherwise, we have $\zeta$ and $\xi$ unrelated with diameters $r, s$ and join $\zeta(a, R)$. By $\Psi_*$ and $\Psi_*^{-1}$ preserving order one can check that $\Psi_*(\zeta \vee \xi) = \Psi_*(\zeta) \vee \Psi_*(\xi)$. Therefore the shortest path from $\Psi_*(\zeta)$ to $\Psi_*(\xi)$ is through $\zeta(\Psi^{-1}(a), R^\q)$, which is the homeomorphic image of $[\zeta, \xi]$. 
 Since $\diam(\Psi_*(\zeta)) = \diam(\zeta)^\q$, the length of this path will be \[d_\bH(\Psi_*(\zeta), \Psi_*(\xi)) = 2\log(R^\q)-\log(r^\q) - \log(s^\q) = \q\,d_\bH(\zeta, \xi).\]
\end{proof}

\begin{thm}\label{thm:skew:autaffinoidmapping}
Let $\Psi$ be an \equivar{} automorphism of $K$ extended trivially to $K(y)$ with scale factor $\q$. For any Berkovich affinoid $W \subseteq \P^1_\an$, let $W_{\text{I}} = W \cap \P^1(K)$. Then $\Psi_*(W)$ is the Berkovich affinoid of the same type (if any) corresponding to $\Psi_* (W_{\text{I}}) = \Psi^{-1}(W_{\text{I}})$, and $\Psi_*^{-1}(W)$ is the Berkovich affinoid of the same type (if any) corresponding to $\Psi_*^{-1}(W_{\text{I}}) = \Psi(W_{\text{I}})$. Boundaries are mapped bijectively to boundaries.
\end{thm}

\begin{proof}
 By \autoref{thm:skew:scaledhomeo}, $\Psi_*$ bijectively maps $D(a, r)$ to $D(\Psi^{-1}(a), r)$ if and only if it maps $D_\an(a, r)$ to $D_\an(\Psi^{-1}(a), r)$. The same goes for closed disks on the affine line, and similarly for disks in the projective line. Since $\Psi_*$ is a bijection, an affinoid $W = D_1 \cap \cdots \cap D_n$ is mapped to $\Psi_*(D_1) \cap \cdots \cap \Psi_*(D_n)$ etc. In fact, $\Psi_*(W)$ has the same number of `holes' as $W$.
\end{proof}

\begin{thm}\label{thm:skew:opencts}
 Let $\phi_*$ be a non-constant skew product over $K$. Then $\phi_*$
\begin{enumerate}
\item is a continuous function;
\item is an open mapping;
 \item is the unique continuous extension of $(\phi_1^*)^{-1} \circ \phi_2$ on $\P^1(K) \subset \P^1_\an(K)$; and
 \item preserves the types of each Berkovich point.
\end{enumerate}
\end{thm}
\unsure{Present horizontally?}

\begin{proof}
Since $\phi_* = \phi_{1*} \circ \phi_{2*}$ therefore this follows from the same results for $\phi_{1*}$ and $\phi_{2*}$, namely, \autoref{thm:skew:scaledhomeo} \autoref{thm:skew:autaffinoidmapping}, \cite[Theorem 7.4, Corollary 7.9, Corollary 7.16]{Bene}.
\end{proof}


The following theorem generalises \cite[Theorem 7.8]{Bene} for rational maps to skew products.

\begin{thm}\label{thm:skew:affinoidmapping}
Let $\phi_*$ be a non-constant skew product over $K$, let $W \subseteq \P^1_\an(K)$ be a connected Berkovich affinoid, and let
$W_I = W \cap \P^1(K)$ be the corresponding connected affinoid in $\P^1(K)$. Then $\phi_*(W)$ is the Berkovich connected affinoid of the same type (if any) corresponding to $\phi_*(W_I)$, and $\phi_*^{-1}(W)$ is the Berkovich affinoid of the same type (if any) corresponding to $\phi_*^{-1}(W_I)$. Moreover, the following hold.
\begin{enumerate}[label=(\alph*), ref=\theenumi]
  \item $\partial(\phi_*(W))\subseteq \phi_*(\partial W)$.
  \item Each of the connected components $V_1, \dots , V_m$ of $\phi_*^{-1}(W)$ is a connected Berkovich affinoid mapping onto $W$.
  \item For each $i = 1, \dots, m$,
 \[\phi_*(\partial V_i) = \partial W \text{ and } \phi_*(\interior V_i) = \interior W,\]
  where $\interior X$ denotes the interior of the set $X$.
  \item If $W$ is open, then $\phi_*(\partial V_i) \cap W = \emptyset$.
\end{enumerate}
\end{thm}

A key omission from this theorem states that each map $\phi_* : V_i \to W$ is $d_j$-to-$1$ counting multiplicity, with $d_1 + \cdots + d_m = \rdeg(\phi)$. This will follow later in \autoref{prop:skew:dto1} when we have a good notion of multiplicity, called local degree. Overall, one can compare with the classical case \autoref{thm:semninorms:affinoidmapping}.

\begin{proof}
Since $\phi_* = \phi_{1*} \circ \phi_{2*}$ this follows from \autoref{thm:skew:autaffinoidmapping} and \cite[Theorem 7.8]{Bene}.
\end{proof}

\improvement{Better waffle and citations?}

The following generalises \cite[Corollary 7.9]{Bene}.

\begin{cor}[Properness Criterion]
 Let $\phi_*$ be a non-constant skew product, and let $U \subseteq \P^1_\an$ be an open connected Berkovich affinoid. Then the following are equivalent.
 \begin{enumerate}[label=(\alph*), ref=\theenumi]
\item $\phi_*(U)\cap \phi_*(\partial U)= \emptyset$.
\item $U$ is a connected component of $\phi_*^{-1}(\phi_*(U))$.
\end{enumerate}
\end{cor}

The next theorem generalises \cite[Theorem 7.12]{Bene} and is important for understanding local ramification or \emph{degrees}, as explained in the next subsection.

\begin{thm}\label{thm:skew:imgtypeII}
 Let $\phi_*= \phi_{1*} \circ \phi_{2*}$ be a non-constant skew product with scale factor $\q$, and let $\zeta = \zeta(a, r) \in \P^1_\an$ be a point of Type II or III. Let $\lambda < 1$ be large enough so that
 \[\phi_2(y) = \sum_{n \in \Z} b_n(y - a)^n\] converges on the annulus
 \[U_\lambda = \set{\lambda r < |y-a| < r}\] and such that $\phi_2(y) - b_0$ has both the inner and outer Weierstrass degrees equal to $d$. 
 Setting $s = {b_d}^\q r^{d\q}$, we have
 \[\phi_*(U_\lambda) = 
\begin{cases}
 \set{\lambda^{d\q} s < |y-\phi_{1*}(b_0)| < s} & d > 0\\
 \set{s < |y-\phi_{1*}(b_0)| < \lambda^{d\q} s} & d < 0
\end{cases}
\]
Similarly if instead $\lambda > 1$ and $\phi_2(y)$ converges on the annulus 
\[V_\lambda = \set{r < |y-a| < r\lambda}\] and such that $\phi_2(y) - b_0$ has both the inner and outer Weierstrass degrees equal to $d$, then we have
\[\phi_*(V_\lambda) = 
\begin{cases}
 \set{s < |y-\phi_{1*}(b_0)| < \lambda^{d\q} s} & d > 0\\
 \set{\lambda^{d\q} s < |y-\phi_{1*}(b_0)| < s} & d < 0
\end{cases}
\]
 
Moreover $r \in \abs{K^\times} \iff s \in \abs{K^\times}$, and
\[\phi_*(\zeta) = \zeta(\phi_{1*}(b_0), s).\]
\end{thm}


\begin{proof}
From \cite[Theorem 7.12]{Bene} we know this result for rational maps, meaning that 
\[\phi_{2*}(U_\lambda) = 
\begin{cases}
 \set{\abs{b_d}(\lambda r)^d < \abs{y-b_0} < \abs{b_d}r^d} & d > 0\\
 \set{\abs{b_d}r^d < \abs{y-b_0} <  \abs{b_d}(\lambda r)^d} & d < 0
\end{cases}
\]
The action of $\phi_{1*}$ is to scale by $\q$, in particular \autoref{thm:skew:scaledhomeo} shows that
\[\phi_{1*}\left(\set{R_1 < \abs{z-b} < R_2}\right) = \set{R_1^\q < \abs{z-\phi_{1*}(b)} < R_2^\q}\]
therefore 
\[\phi_*(U_\lambda) = \phi_{1*} \circ \phi_{2*}(U_\lambda) =
\begin{cases}
 \set{\abs{b_d}^\q(\lambda r)^{d\q} < \abs{y-\phi_{1*}(b_0)} < \abs{b_d}^\q r^{d\q}} & d > 0\\
 \set{\abs{b_d}^\q r^{d\q} < \abs{y-\phi_{1*}(b_0)} <  \abs{b_d}^\q(\lambda r)^{d\q}} & d < 0
\end{cases}
\]
Similar proves the $\lambda > 1$ case.
\end{proof}


%

\subsection{Local Degrees in Directions}

The aim of the next few subsections is to generalise the theory of local degrees from rational maps to skew products; compare with Benedetto \cite[\S7.3, 7.4, 7.5]{Bene}.

\begin{defn}
 Let $\phi_*= \phi_{1*} \circ \phi_{2*}$ be a skew product. The \emph{relative degree} of $\phi_*$ (or $\phi$), $\rdeg(\phi_*) = \rdeg(\phi)$ is the degree of the rational function $\phi_2(y) = \phi^*(y) \in K(y)$.
 
Suppose that $\phi_2$ has the Taylor series at $y=a$ given by \[\phi_2(y) = b_0 + b_d(y-a)^d + \bO\left((y-a)^{d+1}\right),\] where $b_d \ne 0$. Then \emph{algebraic multiplicity} of $\phi$ (and $\phi_2$) at $a$, $\deg_a(\phi) = \deg_a(\phi_2)$ is $d$, 
 the degree of the first non-zero term of the Taylor series $\phi_2(y) - b_0$ about $y=a$, equivalently, it is the multiplicity of $a$ as a root of the equation $\phi_2(y) = \phi_2(a)$.
\end{defn}

In the case of a $\k$-rational skew product $\phi$, this $\rdeg(\phi)$ is the degree of $\phi_2(x, y)$ with respect to $y$ only; this justifies the use of `relative'. We expand the definition of map on directions and local degrees to skew products, again using annuli.

\begin{defn}
Let $\phi_*= \phi_{1*} \circ \phi_{2*}$ be a non-constant skew product, $\zeta \in \P^1_\an$ and $\bvec v \in \Dir\zeta$ be a direction at $\zeta$. We define the \emph{local degree of $\phi_*$ at $\zeta$ in the direction $\bvec v$}, $\deg_{\zeta, \bvec v}(\phi)$, and a direction $\phi_\#(\bvec v) \in \Dir{\phi_*(\zeta)}$ at $\phi_*(\zeta)$ as follows.
\begin{itemize}
 \item For a Type I point $\zeta = a$ and the unique direction $\bvec v$ at $\zeta$, $\phi_\#(\bvec v)$ is the unique direction at $\phi_*(\zeta) = b_0$, and the local degree at $a$ in this direction is the algebraic multiplicity of $\phi_2$ at $a$ \[\deg_{\zeta, \bvec v}(\phi) = \deg_a(\phi_2).\]
 \item For a Type II/III point $\zeta = \zeta(a, r)$ and a direction $\bvec v = \vec{v}(a)$, then $\phi_\#(\bvec v)$ is the direction at $\phi_*(\zeta)$ containing $\phi_*(U)$ where $U = \set{\lambda r < |\zeta-a| < r}$ is the annulus in \autoref{thm:skew:imgtypeII} and the local degree in the direction of $\bvec v$ will be the common Weierstrass degree \[\deg_{\zeta, \bvec v}(\phi) = \wdeg_{a, r}(\phi_2).\]
In the case where the direction $\bvec v = \vec v(a)$ contains $\infty$, we use the annulus $U = \set{r < |z-a| < r\lambda}$ with $\lambda > 1$, as in \autoref{thm:skew:imgtypeII}. The common degree will be \[\deg_{\zeta, \bvec v}(\phi) = \overline\wdeg_{a, r}(\phi_2).\]
 \item For a Type IV point $\zeta$ and the unique direction $\bvec v$ at $\zeta$, $\phi_\#(\bvec v)$ is the unique direction at $\phi_*(\zeta)$, and $\deg_{\zeta, \bvec v}(\phi) = \deg_{\zeta, \bvec v}(\phi_2)$. This is the Weierstrass degree of $\phi_2$ in a sufficiently small neighbourhood $D_\an(a, r)$ of $\zeta$, \[\deg_{\zeta, \bvec v}(\phi) = \wdeg_{a, r}(\phi_2).\]
\end{itemize}
\end{defn}

\begin{prop}\label{prop:skew:scaledirections}
 Let $\phi_*= \phi_{1*} \circ \phi_{2*}$ be a non-constant skew product. Then 
\begin{enumerate}
 \item $\deg_{\zeta, \bvec v}(\phi_1) = 1\quad \forall \zeta \in \P^1_\an, \bvec v \in \Dir\zeta$
 \item $\phi_{1\#} : \Dir\zeta \to \Dir{\phi_{1*}(\zeta)}$ is a bijection at every $\zeta \in \P^1_\an$.
 \item $\deg_{\zeta, \bvec v}(\phi) = \deg_{\zeta, \bvec v}(\phi_2)\quad \forall \zeta \in \P^1_\an, \bvec v \in \Dir\zeta.$
\end{enumerate}
\end{prop}

The following shows that local degree acts precisely like a directional derivative, and we have a chain rule. Compare to \cite[Proposition 7.1]{Bene}.

\begin{prop}\label{prop:skew:chaindir}
 Let $\phi_*$, $\psi_*$ be non-constant skew products, let $\zeta \in \P^1_\an$, and let $\bvec v$ be a direction at $\zeta$. Then
\begin{enumerate}
 \item $1 \le \deg_{\zeta, \bvec v}(\phi) \le \rdeg(\phi)$
 \item $\psi_\# \circ \phi_\# = (\psi \circ \phi)_\#$
 \item $\deg_{\zeta, \bvec v}(\psi \circ \phi) = \deg_{\zeta, \bvec v}(\phi) \cdot \deg_{\phi_*(\zeta), \phi_\#(\bvec v)}(\psi)$.
\end{enumerate}
\end{prop}

\unfinished{Prove?}

As for rational maps, $m = \deg_{\zeta, \bvec v}(\phi)$ is the number for which $\phi_*$ is an $m$-to-$1$ mapping on a thin annulus in the direction $\bvec v$. However unlike rational maps, this $m$ is not exactly the factor with which $\phi_*$ scales distances in some principal branch of $\bvec v$; the scale factor $\q$ also plays its part.

\begin{thm}\label{thm:intervalstretch}
 Let $\phi_*$ be a skew product of scale factor $\q$, let $\zeta \in \P^1_\an$, let $\bvec v$ be a direction at $\zeta$, and let $m = \deg_{\zeta, \bvec v}(\phi)$. Then there is a point $\zeta' \in \bvec v$ such that $\phi_*$ maps $[\zeta , \zeta']$ homeomorphically onto $[\phi_*(\zeta) , \phi_*(\zeta')]$, scaling distances by a factor of $m\q$, i.e.
 \[d_\bH(\phi_*(\xi_1) , \phi_*(\xi_2)) = m\cdot\q\cdot d_\bH(\xi_1 , \xi_2)\quad \forall \xi_1, \xi_2 \in [\zeta , \zeta']\cap\bH\]
\end{thm}

\begin{proof}
 By \autoref{thm:skew:comp}, we can write $\phi_* = \phi_{1*} \circ \phi_{2*}$. By \autoref{prop:skew:scaledirections} we have that $m = \deg_{\zeta, \bvec v}(\phi_2)$ and so by \cite[Theorem 7.22]{Bene}, we have the conclusion for the rational map $\phi_{2*}$. Meaning there is a point $\zeta' \in \bvec v$ such that $\phi_{2*}$ maps $[\zeta , \zeta']$ homeomorphically onto $[\phi_{2*}(\zeta), \phi_{2*}(\zeta')]$, scaling distances by a factor of $m$.
 \[d_\bH(\phi_{2*}(\xi_1) , \phi_{2*}(\xi_2)) = m\cdot d_\bH(\xi_1 , \xi_2)\quad \forall \xi_1, \xi_2 \in [\zeta , \zeta']\cap\bH.\]
 By \autoref{thm:skew:scaledhomeo} $\phi_{1*}$ is a homeomorphism scaling all paths by a factor of $\q$ i.e. \[d_\bH(\phi_*(\xi_1) , \phi_*(\xi_2)) = \q \cdot d_\bH(\phi_{2*}(\xi_1) , \phi_{2*}(\xi_2)) \quad \forall \xi_1, \xi_2\] so we are done.
\end{proof}

\begin{cor}\label{cor:berk:lipschitz}
 Let $\phi_*$ be a skew product of relative degree $d$ and scale factor $\q$. Then $\phi_*$ is Lipschitz on $\bH$ with respect to the hyperbolic metric $d_\bH$, with Lipschitz constant at most $d\q$.
\end{cor}

\improvement{Add result about minimum stretch factor}
The following is a direct translation of \cite[Proposition 7.25]{Bene} to skew products via \autoref{thm:skew:comp}.

\begin{prop}
 Let $\phi_*$ be a skew product, let $\zeta = \zeta(a, r) \in \P^1_\an$ be a point of Type II or III, and let $\bvec v$ be a direction at $\zeta$. Suppose that $0$ and $\infty$ lie in different directions at $\phi_*(\zeta)$.
Let $M_0$ and $M_\infty$, respectively, be the number of zeros and poles, counted
with multiplicity, of $\phi$ in $\bvec v$; and let $N_0$ and $N_\infty$, respectively, be the number
of zeros and poles, counted with multiplicity, of $\phi_*$ in $\CD(a, r)$. Then
\[\deg_{\zeta, \bvec v}(\phi) =
\begin{cases}
 M_0-M_\infty & \text{if } \infty \nin \bvec v\ \and\ 0 \in \phi_\#(\bvec v)\\
 M_\infty-M_0 & \text{if } \infty \nin \bvec v\ \and\ \infty \in \phi_\#(\bvec v)\\
 N_0-N_\infty & \text{if } \infty \in \bvec v\ \and\ 0 \in \phi_\#(\bvec v)\\
 N_\infty-N_0 & \text{if } \infty \in \bvec v\ \and\ \infty \in \phi_\#(\bvec v)\\
\end{cases}.
\]
\end{prop}

\subsection{Local Degrees at Points}

Following \cite[\S7.4]{Bene} we define and study the local degree $m = \deg_\zeta(\phi)$ of a skew product $\phi_* : \P^1_\an \to \P^1_\an$ at a point $\zeta \in \P^1_\an$. This will generalise the algebraic multiplicity $\deg_a(\phi)$ for classical points $a \in \P^1$. We shall see this gives the correct number $m$ that serves as both the local topological degree ($\phi_*$ is locally $m$-to-$1$ near $\zeta$) and its local algebraic degree.


\begin{defn}
 Let $\phi$ be a skew product, and let $\zeta \in \P^1_\an$. The \emph{local degree}, of $\phi$ at $\zeta$ is \improvement{This could probably be called a $\limsup$}
 \[\deg_\zeta(\phi) = \inf_U\left(\sup_b\left(\sum_{a \in U\cap \phi_*^{-1}(b)} \deg_a(\phi)\right)\right),\]
 where the infimum is over all open $U \subseteq \P^1_\an$ containing $\zeta$, the supremum is over all $b \in \phi_*(U)\cap\P^1(K)$.
\end{defn}

The right hand side measures the generic number of preimages of a point $b \in \phi_*(U)$ in $U$. One could worry about subtracting `too many' disks from the neighbourhood $U$, but one can only delete finitely many, whereas infinitely many remain. Moreover, by considering \autoref{thm:berk:topbasis} one can see that $U$ must always contain a thin annulus (or punctured disk) in \emph{every} direction at $\zeta$, and by \autoref{thm:skew:imgtypeII} such an annulus $V$ maps to another annulus $W$ with boundary $\phi_*(\zeta)$, which by \autoref{thm:seminorms:mapannulus} is uniformly $\deg_{\zeta, \bvec v}(\phi)$-to-$1$ on classical points counting multiplicity. To be precise, the local degree counts the total number of preimages of a sufficiently small $W$ which lie in $U$, possibly in multiple (but finitely many) directions at $\zeta$; shrinking $U$ may cause us to shrink $W$, but it will never change this number. Clearly this relates to local degrees in directions, as is spelled out in the next results. Indeed it turns out that picking any such $W$ at $\phi_*(\zeta)$, and counting preimages of $W$ close to $\zeta$ will produce $\deg_\zeta(\phi)$.\improvement{Agh writing! Try again! Avoid notational conflict with $W$ below.}

\begin{prop}
 Let $\phi_*= \phi_{1*} \circ \phi_{2*}$ be a non-constant skew product. Then 
\begin{enumerate}
 \item $\deg_\zeta(\phi_1) = 1\quad \forall \zeta \in \P^1_\an$,
 \item $\deg_\zeta(\phi) = \deg_\zeta(\phi_2)\quad \forall \zeta \in \P^1_\an$.
\end{enumerate}
\end{prop}

\begin{lem}
 Let $\phi_*$ be a skew product, and let $\zeta \in \P^1_\an$, and let $U \subseteq \P^1_\an$ be an open set containing $\zeta$. Then there is an open connected Berkovich affinoid $W$ such that
 
\begin{enumerate}[label=(\alph*), ref=\theenumi]
 \item $\zeta \in W \subseteq U$,
 \item $\phi_*(\partial W) \subseteq \partial(\phi_*(W))$, and
 \item $\sum_{a \in W\cap \phi_*^{-1}(b)}\deg_a(\phi) = \deg_\zeta(\phi)$ for every Type I point $a$ in $\phi_*(W)$.
\end{enumerate}
\end{lem}

 \improvement{This lemma is topological, Benedetto overcomplicates it; definitely rewrite proof.}
 
 The following theorem provides a chain rule for local degree and states that $\phi_*$ is exactly $\rdeg(\phi)$-to-$1$, counting `multiplicity', meaning local degree taking the role of multiplicity. Compare with \cite[Theorem 7.29]{Bene}.
 
\begin{thm}\label{thm:skew:chainrule}
 Let $\phi_*, \psi_*$ be skew products, and let $\zeta \in \P^1_\an$. Then
\begin{enumerate}[label=(\alph*), ref=\theenumi]
 \item The function $\P^1_\an \longrightarrow \set{1, \dots, \rdeg(\phi)}$ given by $\zeta \longmapsto \deg_\zeta(\phi)$ is upper semicontinuous. That is, the set
 \[\set[\zeta \in \P^1_\an]{\deg_\zeta(\phi) \ge m}\]
 is closed in the weak topology.
 \item $\deg_\zeta(\psi \circ \phi) = \deg_\zeta(\phi) \cdot \deg_{\phi_*(\zeta)}(\psi)$.
 \item $\sum_{\xi \in \phi^{-1}(\zeta)} \deg_\xi(\phi) = \rdeg(\phi)$.
\end{enumerate}
\end{thm}

\improvement{Prove it! Do we actually use Benedetto here?}

\begin{prop}\unsure{why is this here? move? In it's current form this needs a proof.}
 Let $\phi_*$ be a skew product of scale factor $\q$ over $K$ of characteristic $0$. Then $\phi$ is an isometry on $\bH$ with respect to $d_\bH$ if and only if $\rdeg(\phi) = 1$ and $\q = 1$. 
\end{prop}

\begin{qtn}
What about characteristic $p$?
 \improvement[inline]{Cover this case and write about a skew product which is actually rational. Frobenius $\phi_1^* = \phi_2 = y^p$ with ... gives identity.}
 \info[inline]{Idea: write a rational map of degree $< p$ with $K(y^p)$ `coefficients' and show that the derivative is $0$ if and only if it is `constant' i.e. in $K(y^p)$. Then it is actually a $p$th power of a similar rational map of actual degree $<p$.}
\end{qtn}

\begin{thm}\label{thm:sumdirdegs}
 Let $\phi_*$ be a skew product, and let $\zeta \in \P^1_\an$, and let $\bvec w$ be a direction at $\phi_*(\zeta)$. Then
 \[\deg_\zeta(\phi) = \sum_{\phi_\#(\bvec v) = \bvec w} \deg_{\zeta, \bvec v}(\phi),\]
 where the sum is over all directions $\bvec v$ at $\zeta$ for which $\phi_\#(\bvec v) = \bvec w$.
\end{thm}
 
\begin{cor}
 Let $\phi_*$ be a skew product, and let $\zeta \in \P^1_\an$.
\begin{enumerate}[label=(\alph*), ref=\theenumi]
 \item If $\zeta = a$ is of Type I, then $\deg_\zeta(\phi)$ agrees with the algebraic multiplicity $\deg_a(\phi)$ of $\phi_2(y)$ at $a$.
 \item If $\zeta$ is of Type III, then \[\deg_\zeta(\phi) = \deg_{\zeta, \bvec u}(\phi) = \deg_{\zeta, \bvec v}(\phi),\]
 where $\bvec u$ and $\bvec v$ are the two directions at $\zeta$.
 \item If $\zeta$ is of Type IV, then $\deg_\zeta(\phi) = \wdeg_{a, r}(\phi_2)$ for any sufficiently small disc $D_\an(a, r)$ containing $\zeta$.
\end{enumerate}
\end{cor}


The following generalisation of \cite[Proposition 7.33]{Bene} completes the picture given by \autoref{thm:skew:affinoidmapping}.

\begin{prop}\label{prop:skew:dto1}
 Let $\phi_*$ be a skew product of relative degree $d \ge 2$, and let $W \subseteq \P^1_\an$ be a connected Berkovich affinoid. Write
\[\phi_*^{-1}(W) = V_1 \cup \cdots \cup V_m\]
as a disjoint union of connected Berkovich affinoids according to \autoref{thm:skew:affinoidmapping}. Then for each $i = 1, \dots, m$, there is an integer $1 \le d_i \le d$ such that every point in $W$ has exactly $d_i$ preimages in $V_i$, counting multiplicity. Moreover, $d_1 + \cdots + d_m = d$.
\end{prop}

\subsection{Reduction and Computation of Local Degrees}

Recall the discussion of reduction for rational maps in \autoref{sec:berk:red}. Now we will extend these ideas to skew products.

For a skew product $\phi_*= \phi_{1*} \circ \phi_{2*}$, recall from \autoref{thm:skew:scaledhomeo} that $\phi_{1*} : \P^1_\an \to \P^1_\an$ bijectively maps each direction at $\zeta(0, 1)$ to another direction at $\zeta(0, 1)$. Hence by the geometric definition of explicit good reduction $\phi_{1*}$ always has explicit good reduction; indeed $\phi_1^*$ descends to a field automorphism $\overline\phi_1^* : \k \to \k$. 

\begin{lem}
 Let $\phi_*= \phi_{1*} \circ \phi_{2*}$ be a skew product. Then \[\phi_{1\#}(\bvec v) = (\overline\phi_1^*)^{-1}(\bvec v) = \phi_{1*}(\bvec v)\] where $\phi_{1\#} : \Dir{\zeta(0, 1)} \to \Dir{\zeta(0, 1)}$ is the map on tangents at $\zeta(0, 1)$.
\end{lem}

\begin{defn}
 Let $\phi_*= \phi_{1*} \circ \phi_{2*}$ be a skew product. We define the reduction $\overline \phi_1  = (\overline\phi_1^*)^{-1}$ and $\overline \phi = \overline \phi_1 \circ \overline \phi_2$.
\end{defn}

We will say that $\phi_*$ has \emph{explicit good reduction} if and only if $\phi_{2*}$ has explicit good reduction (as a rational map). The following theorem and proposition formalise these ideas, generalising \cite[Theorem 7.34, Lemma 7.35]{Bene}.


\begin{thm}\label{thm:skew:reduction}
 Let $\phi_*= \phi_{1*} \circ \phi_{2*}$ be a skew product of relative degree $d \ge 1$.
 Then $\overline \phi_2$ is nonconstant if and only if $\phi_2(\zeta(0, 1)) = \zeta(0, 1)$. In that case,
\[\deg_{\zeta(0, 1)}(\phi) = \deg(\overline \phi_2).\]
Also in that case, define a direction $\bvec v$ at $\zeta(0, 1)$ to be \emph{bad} if $\phi_*(\bvec v)$ contains both a zero and a pole of $\phi_2$, and define $T$ to be the set of bad directions $\bvec v$. Then
\begin{enumerate}[label=(\alph*), ref=\theenumi]
 \item $T$ is a finite set,
\item $\phi_*(\bvec v) = \P^1_\an$ for each direction $\bvec v \in T$, and
\item $\phi_\#(\bvec v) = \phi_*(\bvec v) = \vec v(\phi_*(\xi))$ for all directions $\bvec v$ at $\zeta(0, 1)$ not in $T$ and all $\xi \in \bvec v$. For such directions $\bvec v$, we have
\[\deg_{\zeta(0, 1), \bvec v}(\phi) = \wdeg_{\bvec v}(\phi_2) =
\begin{cases}
 \wdeg_{a, 1}(\phi_2) & \bvec v = \vec v(a)\\
 \overline \wdeg_{0, 1}(\phi_2) & \bvec v = \vec v(\infty)
\end{cases}
\]
\end{enumerate}
Finally, $\phi_2$ has explicit good reduction, i.e., $\deg(\overline \phi_2) = d$, if and only if $\overline \phi_2$ is
nonconstant and $T = \emp$.
\end{thm}

\begin{prop}
 Let $\phi_*= \phi_{1*} \circ \phi_{2*}$ be a skew product such that $\bar \phi_2$ is nonconstant. Then 
 
\begin{enumerate}[label=(\alph*), ref=\theenumi]
 \item $\phi_*(\zeta(0, 1)) = \zeta(0, 1)$.
 \item For any $a, b \in \P^1(K)$, we have
 \[\overline \phi(\overline a) = \overline b\ \text{ if and only if }\ \phi_\#(\vec v(a)) = \vec v(b).\]
 \item For any $a \in \P^1(K)$, we have 
 \[\deg_{\zeta(0, 1), \vec v(a)}(\phi) = \deg_{\overline a}(\overline \phi_2)\]
\end{enumerate}
\end{prop}

\subsection{The Injectivity and Ramification Loci}

\begin{defn}
 Let $\phi_*$ be a simple skew product. The ramification locus of $f$ is the set
 \[\Ram(\phi) = \set[\zeta \in \P^1_\an]{\deg_\zeta(\phi) \ge 2},\]
 and the injectivity locus of $f$ is the complement
 \[\Inj(\phi) = \P^1_\an \sm \Ram(\phi) = \set[\zeta \in \P^1_\an]{\deg_\zeta(\phi) = 1}.\]
\end{defn}

\begin{prop}
 \[\Ram(\phi) = \Ram(\phi_2),\quad\Inj(\phi) = \Inj(\phi_2)\]
\end{prop}

\begin{thm}
 Let $\phi_*$ be a simple skew product of relative degree $d \ge 2$. Then every connected component of the injectivity locus $\Inj(\phi)$ is a rational open connected affinoid.
\end{thm}

\begin{defn}
 The set of Type I critical points of a rational map $\phi$ is denoted $\Crit(\phi)$.
\end{defn}

\begin{defn}
 Let $\phi_* = \phi_{1*} \circ \phi_{2*}$ be a skew product, and let $p$ be the residue characteristic of $\hv$. Suppose that for every point $\zeta \in \P^1_\an$, the local degree $\deg_\zeta(\phi)$ is not divisible by $p$. Then we say that $\phi_*$ is \emph{tame}.
\end{defn}

In particular, if $p = 0$ or $p > \deg(\phi)$ then $\phi$ is tame.

This next result is a corollary of the theorem of Faber for rational maps \cite{Faber13.1}. See also \cite{Faber13.2} when $\phi_*$ might not be tame.

\begin{thm}\label{thm:skew:ramconvcrit}
 Let $\phi_*$ be a tame skew product of relative degree $d \ge 2$. Then the ramification locus $\Ram(\phi)$ is a subset of $\Hull(\Crit(\phi_2))$ whose set of endpoints is precisely the critical set of $\phi_2$, $\Crit(\phi_2) = \Ram(\phi) \cap \P^1(K)$.
 %
\end{thm}


\begin{cor}\label{cor:skew:TypeIVdegone}
Let $\phi_*$ be a tame skew product, then
 $\deg_\zeta(\phi) = 1$ for every Type IV point of $\P^1_\an$.
\end{cor}

\subsection{Geometric and Topological Ideas}\label{sec:skew:top}

The next theorem is surely well known. It says that if (a continuous open map of a tree) $\phi_*$ is not injective on some interval (or subtree) $T$ then is there is a point $\zeta$ in (the interior of) $T$ for which $\phi_*(\zeta)$ is locally an extreme value. This also represents Rolle's Theorem for skew products.

\begin{thm}[Extreme Value Theorem]\label{thm:skew:extvalue}
 Let $\phi_*$ be a non-constant skew product.
\begin{enumerate}[label=(\alph*), ref=\theenumi]
 \item Suppose that $\zeta_1, \zeta_2 \in \P^1_\an$ and $\phi_*([\zeta_1, \zeta_2]) \ne [\phi_*(\zeta_1), \phi_*(\zeta_2)]$ then there is a point $\zeta \in (\zeta_1, \zeta_2)$ such that the \emph{extreme value} $\xi = \phi_*(\zeta)$ is an endpoint of $\phi_*([\zeta_1, \zeta_2])$ and $\phi_\#(\vec v(\zeta_1)) = \phi_\#(\vec v(\zeta_2))$. In particular, $\zeta \in \Ram(\phi)$.
 \item Let $T \subset \P^1_\an$ be connected. If $\phi_*$ is not a homeomorphism on $T$ then there is a point $\zeta$ with two directions $\bvec u, \bvec v$ intersecting $T$ (edges) such that $\phi_\#(\bvec u) = \phi_\#(\bvec v)$. In particular, $\zeta \in \Ram(\phi)$.
\end{enumerate}
 \end{thm}
 
\begin{rmk}
 By the unique path connectedness of $\P^1_\an$, it is always true that \[\phi_*([\zeta_1, \zeta_2]) \supseteq [\phi_*(\zeta_1), \phi_*(\zeta_2)].\]
\end{rmk}

\begin{proof}
 First consider the second part. Suppose $\phi_*(\zeta_1) = \phi_*(\zeta_2)$ for some $\zeta_1, \zeta_2 \in T$. Then since a non-constant skew product is nowhere locally constant, we have that \[\phi_*([\zeta_1, \zeta_2]) \ne [\phi_*(\zeta_1), \phi_*(\zeta_2)].\] The rest follows from proving the first part of the theorem.
 
We will find an extreme point $\zeta \in (\zeta_1, \zeta_2)$. First we want to reduce to an interval in $\bH$. By continuity and simple exhaustion of $[\zeta_1, \zeta_2]$ by slightly smaller intervals, we may assume $\zeta_1, \zeta_2 \in \bH$ whilst keeping $\phi_*([\zeta_1, \zeta_2]) \sm [\phi_*(\zeta_1), \phi_*(\zeta_2)]$ non-empty.

Both $\phi_*([\zeta_1, \zeta_2])$ and $[\phi_*(\zeta_1), \phi_*(\zeta_2)] = I$ are compact and contained in $\bH$. 
Then we can define a continuous function mapping $\zeta \in [\zeta_1, \zeta_2]$ to $d_\bH(\phi_*(\zeta), I)$. Note that this function would be $0$ at $\zeta_1, \zeta_2$. The extreme value theorem for real valued functions on a compact topological space provides an `extreme point' $\zeta \in (\zeta_1, \zeta_2)$ for this function; whence $\phi_*(\zeta) = \xi \in \phi_*([\zeta_1, \zeta_2])$ attains the maximal distance from $I$. By our choice of extreme value $\xi$, only one direction $\bvec w$ at $\xi$ contains $\phi_*([\zeta_1, \zeta_2]) \supset I$, else we could find another point in $\phi_*([\zeta_1, \zeta_2])$ with a greater hyperbolic distance from $I$. Therefore both $\bvec u = \vec v(\zeta_1)$ and $\bvec v = \vec v(\zeta_2)$ are mapped by $\phi_\#$ to $\bvec w$. This implies that $\deg_\zeta(\phi) \ge 2$ by \autoref{thm:sumdirdegs}.
\end{proof}
 
\improvement{Move this, rename and integrate with other results}
\begin{cor}\label{cor:skew:injinterval}
 Let $\phi_*$ be a skew product of scale factor $\q$, and consider two points $\zeta_1, \zeta_2 \in \P^1_\an$. If $(\zeta_1, \zeta_2) \subset \Inj(\phi)$ then
 \[\phi_* : [\zeta_1, \zeta_2] \to  [\phi_*(\zeta_1), \phi_*(\zeta_2)]\] is a homeomorphism scaling all distances by a factor of $\q$, i.e.
 \[d_\bH(\phi_*(\xi_1) , \phi_*(\xi_2)) = \q\cdot d_\bH(\xi_1 , \xi_2)\quad \forall \xi_1, \xi_2 \in [\zeta_1, \zeta_2]\cap\bH.\]
\end{cor}


\begin{cor}\label{cor:skew:fixedinterval}
 Let $\phi_*$ be a skew product with two fixed points $\zeta, \xi \in \P^1_\an$. If $(\zeta, \xi) \subset \Inj(\phi)$ then it is an interval of fixed points and necessarily $\phi_*$ is simple ($\q = 1$).
\end{cor}


\begin{proof}[Proof of \autoref{cor:skew:injinterval}]
By the contrapositive of the extreme value theorem \autoref{thm:skew:extvalue}, if $(\zeta_1, \zeta_2) \subset \Inj(\phi)$ then $\phi_*$ is injective there. By definition of the hyperbolic metric, it is locally Euclidean and additive, meaning that if $\xi_2 \in (\xi_1, \xi_3)$ then $d_\bH(\xi_1, \xi_3) = d_\bH(\xi_1, \xi_2) + d_\bH(\xi_2, \xi_3)$. Therefore it is enough to show this on each piece of a (finite) cover of $[\zeta_1, \zeta_2]$. Given any $\zeta \in [\zeta_1, \zeta_2]$, by \autoref{thm:intervalstretch} we can find a $\zeta_j' \in \vec v(\zeta_j)$ such that $\phi_*$ maps $[\zeta , \zeta_j']$ onto $[\phi_*(\zeta) , \phi_*(\zeta'_j)]$, scaling distances by a factor of $\q$. Since $[\zeta , \zeta_j']$ and $[\zeta , \zeta_j]$ are intervals in the same direction, they must agree on an initial segment, so we may shorten this interval to assume $\zeta_j' \in (\zeta , \zeta_j]$. So given $\zeta \in (\zeta_1, \zeta_2)$, the interval $(\zeta_1' , \zeta_2') \subseteq (\zeta_1 , \zeta_2)$ contains $\zeta$ and $\phi_*$ scales distances by $\q$ here; the case $\zeta = \zeta_j$ is similar. We have covered $[\zeta_1, \zeta_2]$ with intervals obeying the dersired conclusion and by compactness we can pick finitely many.
\end{proof}

The following theorem upgrades the conclusions of \autoref{thm:intervalstretch}, giving more information with stronger hypotheses. However the preceding simple proposition allows us use the same (weaker) set of hypotheses as the original version, given only one point and a direction.

\begin{prop}\label{prop:skew:preintervalstretch}
\unfinished{unfinished}
 Let $\phi_*$ be a skew product of scale factor $\q$, let $\zeta_1 \in \P^1_\an$, and $\bvec v \in \Dir\zeta$ be a direction. Then one can always find such a $\zeta_2 \in \bvec v$ satisfying the hypotheses of \autoref{thm:skew:bigintervalstretch} where either $\zeta_2$ is Type I or $d_\bH(\zeta_1, \zeta_2)$ is maximal, $\zeta_2 = \zeta(a, r)$ and there is a pole in $\CD(a, r) \sm D(a, r)$.
\end{prop}

With residue characteristic $p > 0$, the fact that $\zeta_2$ (of maximal distance) is not Type IV is not obvious. In this case there are \emph{many} choices of Berkovich point with the same diameter or distance from $\zeta_1$. Apart from this maximality, the proof is essentially the same as finding the interval in \autoref{thm:intervalstretch}.

\begin{thm}\label{thm:skew:bigintervalstretch}
 Let $\phi_*$ be a skew product of scale factor $\q$, let $\zeta_1, \zeta_2 \in \P^1_\an$, and let $U$ be the connected component\footnote{an open annulus, punctured disk or twice-punctured $\P^1$} of $\P^1_\an \sm \set{\zeta_1, \zeta_2}$ bounded by both points. Suppose that $\deg_{\zeta_1, \vec v(\zeta_2)}(\phi) = \deg_{\zeta_2, \vec v(\zeta_1)}(\phi) = m$, and also that the image of $U$ also has two boundary points (the map is proper on $U$). Equivalently when $U = \set{r < \abs{\zeta - a} < R}$ for $0 \le r < R \le \infty$, and $\phi_2 = \sum_n b_n(y-a)^n$, then this hypothesis means $\phi_{2*}(y) - b_0$ has no zeros or poles in $U$ with inner and outer Weierstrass degree equal to $m$.
 
 Then for every $\zeta \in [\zeta_1, \zeta_2]$ we have 
 \[\deg_\zeta(\phi) = \deg_{\zeta, \vec v(\zeta_1)}(\phi) = \deg_{\zeta, \vec v(\zeta_2)}(\phi) = m\]
  Moreover, $U \cap \phi_*^{-1}[\phi_*(\zeta_1) , \phi_*(\zeta_2)] = [\zeta_1 , \zeta_2]$ and $\phi_* : [\zeta_1 , \zeta_2] \to [\phi_*(\zeta_1) , \phi_*(\zeta_2)]$, is a homeomorphism scaling distances by a factor of $m\q$, i.e.
 \[d_\bH(\phi_*(\xi_1) , \phi_*(\xi_2)) = m\cdot\q\cdot d_\bH(\xi_1 , \xi_2)\quad \forall \xi_1, \xi_2 \in [\zeta , \xi]\cap\bH\]
\end{thm}

\begin{rmk}
 For any direction $\bvec u \in \Dir\zeta$ other than $\vec v(\zeta_1)$ or $\vec v(\zeta_2)$, we have \[\deg_{\zeta, \bvec u}(\phi) = m_1\] where $m_1 = 1$, unless $m_0 = p^j \divides m$ because $\chara \k = p$ and the reduction $\overline \phi = \psi(z^{p^j})$ is inseparable. Either way, the map on these other directions is $m/m_1$-to-one.
\end{rmk}

The following theorem is about hyperbolicity in $\bH$, inspired by the Schwarz-Pick Theorem and the proof of \cite[Proposition 11.3]{Bene}. It say that distances never decrease more than $\q$ in the absence of turning points (ramification). In the other direction any two points have preimages with distance at most $1/\q$ times the original.


\begin{thm}[Hyperbolicity Theorem]\label{thm:skew:hyperbolicity}
 Let $\phi_*$ be a non-constant skew product of scale factor $\q$. Let $\zeta_1, \zeta_2 \in \P^1_\an$ and suppose \[\phi_* : [\zeta_1 , \zeta_2] \to [\phi_*(\zeta_1) , \phi_*(\zeta_2)]\] is a homeomorphism. Then 
 \[d_\bH(\phi_*(\xi_1) , \phi_*(\xi_2)) \ge \q\cdot d_\bH(\xi_1 , \xi_2)\quad \forall \xi_1, \xi_2 \in [\zeta_1 , \zeta_2]\cap\bH.\]
 Conversely, let $\xi_1, \xi_2 \in \P^1_\an$, $\zeta_1 \in \phi_*^{-1}(\xi_1)$. Then there exists, $\zeta_2 \in \phi_*^{-1}(\xi_2)$ such that \[\phi_* : [\zeta_1 , \zeta_2] \to [\xi_1, \xi_2]\] is a homeomorphism. Furthermore, we can choose $\zeta_2$ to be in any direction $\bvec v \in \Dir{\zeta_1}$ such that $\phi_*(\bvec v) \ni \xi_2$. Whence if $\xi_1, \xi_2 \in \bH$ then \[d_\bH(\xi_1 , \xi_2) \ge \q\cdot d_\bH(\zeta_1 , \zeta_2).\]
\end{thm}

\begin{proof}
Let $\zeta \in [\zeta_1 , \zeta_2]$ with $m = \deg_\zeta(\phi)$. Then in both directions (or one at the endpoints) \autoref{prop:skew:preintervalstretch} and \autoref{thm:skew:bigintervalstretch} say that there is a $\zeta' \in \vec v(\zeta_1)$ and $\zeta'' \in \vec v(\zeta_2)$ with distances expanding by at least $m\q$ i.e. \[d_\bH(\phi_*(\xi_1) , \phi_*(\xi_2)) = m\cdot\q\cdot d_\bH(\xi_1 , \xi_2)\quad \forall \xi_1, \xi_2 \in [\zeta' , \zeta'']\cap\bH\]
Since $\phi_*$ is homeomorphic on $\zeta_1, \zeta_2]$, we can finish the argument by additivity and compactness as in the proof of \autoref{cor:skew:injinterval}.

Now we prove the `converse' statement. First we claim that $T$, the component of $\phi_*^{-1}[\xi_1, \xi_2]$ containing $\zeta_1$, is a finite subtree equal to $\Hull(\phi_*^{-1}(\xi_1) \cup \phi_*^{-1}(\xi_2))$. Second we claim that $\phi_*$ is bijective on the edges of $T$. Then we prove the result as follows. Pick any $\zeta_1 \in \phi_*^{-1}(\xi_1)$ and $\bvec v$ in the direction of $\phi_*^{-1}(\xi_2)$. Let $\zeta \in T$ be the next vertex from $\zeta_1$ in in the direction $\bvec v$. If $\phi_*(\zeta) = \xi_2$ then we are done with $\zeta = \zeta_2$ and the path $[\zeta_1, \zeta]$, using the second claim. Otherwise since $\phi_\#$ is always surjective, we can pick a direction $\bvec v \in \Dir\zeta$ such that $\phi_\#(\bvec v) = \vec v(\xi_2)$. Then there is an edge $[\zeta, \zeta']$ of $T$ in the direction $\bvec v$ which is bijective with its image $[\phi_*(\zeta), \phi_*(\zeta')]$. Clearly $[\zeta_1, \zeta)$ and $[\zeta, \zeta']$ are distinct with distinct images; therefore $[\zeta_1, \zeta']$ is bijective with it's image. Continuing this way, one edge at a time we can create a finite path of edges in $T$ bijective with its image until we find $\zeta_2 \in \phi_*^{-1}(\xi_2)$.

To prove the first claim, consider $\xi \in [\xi_1, \xi_2]$ and any $\zeta \in \phi_*^{-1}(\xi)$. Then for any direction $\bvec v = \vec v(\xi_j) \in \Dir\xi$ there must be at least one (but finitely many) preimage(s) $\bvec u \in \Dir\zeta$. By \autoref{thm:skew:reduction} any such $\bvec u$ is either a good direction with $\phi_*(\bvec u) = \bvec v$ or a bad direction with $\phi_*(\bvec u) = \P^1_\an$; either way $\bvec u$ contains a preimage of some $\zeta_j \phi_*^{-1}(\xi_j)$. Thus $\zeta \in \Hull(\phi_*^{-1}(\xi_1) \cup \phi_*^{-1}(\xi_2))$. One can check that a convex hull of finitely many points is a finite subtree.

To prove the second claim, we use the extreme value theorem \autoref{thm:skew:extvalue}. If $\phi_*$ was not bijective on an edge $(\zeta', \zeta'')$, then it contains a point $\zeta$ such that both of the directions $\vec v(\zeta'), \vec v(\zeta'') \in \Dir{\zeta}$ get mapped to the same one $\phi_\#(\vec v(\zeta')) = \phi_\#(\vec v(\zeta'))$, which we will say WLOG is $\vec v(\xi_1)$. 
We know that $\phi_*(\tilde\zeta) \in (\zeta_1, \zeta_2)$ because it is not an endpoint, and again recall that $\phi_\#$ is surjective on directions. Since $\zeta$ is on an edge of $T \subseteq \phi_*^{-1}[\xi_1, \xi_2]$, the preimage of each of the two directions $\vec v(\xi_1), \vec v(\xi_2)$ must be one of the only two directions at $\zeta$ containing $T$, namely $\vec v(\zeta'), \vec v(\zeta'')$. Clearly this is a contradiction.
\end{proof}

%
%





%
%





\section{Periodic Points}\label{sec:per}

\begin{defn}
 Let $\phi_*$ be a skew product and $\zeta \in \P^1_\an$.
 \begin{enumerate}[label=(\alph*), ref=\theenumi]
 \item We say that $\zeta$ is a \emph{periodic point} iff $\phi_*^n(\zeta) = \zeta$ for some $n \ge 1$; if $n$ is the smallest such integer then we call this the \emph{period} of $\zeta$.
 \item If $n$ is the smallest such integer then we call this the \emph{period} of $\zeta$.
 \item In the special case where $n=1$, where $\phi_*(\zeta) = \zeta$, we say that $\zeta$ is a \emph{fixed point}.
 \item If there is an $n_0$ such that $\phi_*^{n_0}(\zeta)$ is periodic then we say that $\zeta$ is \emph{preperiodic}; or in other words, $\phi_*^{n + n_0}(\zeta) = \phi_*^{n_0}(\zeta)$ for some minimal period $n$ and preperiod $n_0$. 
 \item Otherwise if $\phi_*^m(\zeta) \ne \phi_*^n(\zeta)$ for every distinct pair $m, n \in \N$ then we say $\zeta$ is \emph{wandering}.
 \item The \emph{(forward) orbit} of $\zeta$ is $\Orb^+_\phi(\zeta) = \set[\phi_*^n(\zeta)]{n \ge 0}$.
 \item The \emph{backward orbit} of $\zeta$ is $\Orb^-_\phi(\zeta) = \bigcup_{n \ge 0}\phi_*^{-n}(\zeta)$.
 \item The \emph{grand orbit} of $\zeta$ is $\GO_\phi(\zeta) = \bigcup_{n \ge 0}\Orb^-_\phi(\phi_*^n(\zeta))$.
 \end{enumerate}
\end{defn}

In dynamics, how a map behaves in the neighbourhood of a periodic point is more important than the point itself. Typically it matters if points are `attracting', `repelling', or `indifferent'; in higher dimensional dynamics of manifolds, these behaviours can occur simultaneously at a fixed `saddle' point in two different directions, but not in one dimension. Here the Berkovich projective line in one-dimensional, but being a tree with its Type II vertices, there are many directions, leading use to consider saddle points. One stark difference between skew products and rational maps is that skew products can have attracting behaviour but rational maps \emph{never} do.

\begin{defn}
Let $\phi_*$ be a skew product of scale factor $\q$, $\zeta \in \bH$ and consider a direction $\bvec{v}\in \Dir\zeta$. We call $\q \deg_{\zeta, \bvec v}(\phi)$ the \emph{multiplier} of $\bvec v$ and we say $\bvec{v}$ is:
\begin{enumerate}[label=(\alph*), ref=\theenumi]
 \item \emph{indifferent} iff $\q \deg_{\zeta, \bvec v}(\phi) = 1$;
 \item \emph{attracting} iff $\q \deg_{\zeta, \bvec v}(\phi) < 1$; or
 \item \emph{repelling} iff $\q \deg_{\zeta, \bvec v}(\phi) > 1$.
\end{enumerate}
\end{defn}

This definition classifies the general behaviour of nearby points in the given direction by the geometric gradient, as per \autoref{thm:intervalstretch}.

\begin{defn}
Let $\phi_*$ be a skew product, and suppose $\zeta \in \bH$ is a fixed point of $\phi_*$. We say $\zeta$ is:
\begin{enumerate}[label=(\alph*), ref=\theenumi]
 \item \emph{\stronglyindifferent} all directions in $\Dir\zeta$ are indifferent;
 \item \emph{\stronglyattracting} iff at least one direction $\bvec v \in \Dir\zeta$ is attracting, and the rest are indifferent;
 \item \emph{\stronglyrepelling} iff at least one direction $\bvec v \in \Dir\zeta$ is repelling, and the rest are indifferent;
 \item \emph{saddle} iff $\Dir\zeta$ contains both repelling and attracting directions.
\end{enumerate}
Additionally, we call $\deg_\zeta(\phi) \q$ the \emph{multiplier} of $\zeta$. 
Furthermore we say $\zeta$ is:
\begin{enumerate}[label=(\roman*), ref=\theenumi]
 \item \emph{\weaklyindifferent} iff $\q \deg_\zeta(\phi) = 1$;
 \item \emph{\weaklyattracting} iff $\q \deg_\zeta(\phi) < 1$; or
 \item \emph{\weaklyrepelling} iff $\q \deg_\zeta(\phi) > 1$.
\end{enumerate}
We extend these definitions to periodic $\zeta$, by considering it as a fixed point of $\phi_*^n$. Finally we may say $\zeta$ is \emph{\weakly{} non-repelling} iff it is \weaklyindifferent{} or \weaklyattracting{} i.e. $\q \deg_\zeta(\phi) \le 1$.
\end{defn}\unsure{these might not be the best definitions.}

For a rational map $\phi_*$, a fixed point $\zeta \in \bH$ is either indifferent or repelling (never attracting) according to whether the local degree $\deg_\zeta(\phi)$ is $1$ or not. In our case, this remains true when $\q = 1$, but otherwise there are more possibilities. When $\q > 1$ naturally every Type II, III, and IV point is repelling. The most difficult case is $\q <1$, although $\q\rdeg(\phi) < 1$ implies that $\phi_*$ is a contraction mapping on $\bH$, so the dynamics are somewhat trivial.

\begin{ex}
 Let us consider $\hk$ the Levi-Civita series over $\C$, and a skew product defined by $\phi(x, y) = (x^3, y^2)$. Then $\q = 1/3$ and $\rdeg(\phi) = 2$. Then the Gauss point $\zeta(0, 1)$ is a repelling fixed point for $\phi_{2*}$ but for $\phi_*$ it is attracting; in fact all points converge here under iteration. Instead, consider $\phi(x, y) = (x^2, y^3)$, then $\q = 1/3$ and $\rdeg(\phi) = 2$. The Gauss point is a saddle fixed point of $\phi_*$. Indeed, all but two directions are attracting at a rate of their multiplier $\q = 1/2$, but $\vec v(0), \vec v(\infty)$ are repelling with multiplier $3/2$.
\end{ex}

For Type I points we need some slightly different definitions.

\begin{defn}\label{defn:dyn:classicalfixed}
Let $\phi_*$ be a skew product, and suppose $a \in \P^1$ is a fixed classical point of $\phi_*$. Suppose that $\phi_2(y)$ can be expressed locally as \[\phi_2(y) = \phi^*(a) + b_1(y-a) + b_2(y-a)^2 + b_3(y-a)^3 + \cdots\] and let $d = \deg_a(\phi)$ be the smallest integer such that $b_d$ non-zero.
We say $a$ is:
\begin{enumerate}[label=(\alph*), ref=\theenumi]
 \item \emph{superattracting} iff $d\q > 1$;
 \item \emph{superrepelling} iff $d\q < 1$;
 \item \emph{attracting} iff it is superattracting or $d\q = 1$ and $\abs{b_d} < 1$;
 \item \emph{repelling} iff it is superrepelling or $d\q = 1$ and $\abs{b_d} > 1$; or
 \item \emph{indifferent} iff $d\q = 1$ and $\abs{b_d} = 1$.
\end{enumerate}
In the last three cases, with $d\q = 1$, we call $\lambda = \abs{b_d}^\q$ the \emph{multiplier}.
We extend these definitions to periodic $a$, by considering it as a fixed point of $\phi_*^n$.
\end{defn}

\begin{prop}\label{prop:dyn:classicalfixed}
 Let $\phi_*$ be a skew product, and suppose $a \in \P^1$ is a fixed classical point of $\phi_*$.
\begin{enumerate}[label=(\roman*), ref=\theenumi]
 \item $a$ is attracting if and only if there is an $\eps > 0$, such that $\forall r < \eps$ we have \[\phi_*(D(a, r)) \subsetneq D(a, r).\]
 \item $a$ is repelling if and only if there is an $\eps > 0$, such that $\forall r < \eps$ we have \[\phi_*(D(a, r)) \supsetneq D(a, r).\]
 \item $a$ is indifferent if and only if there is an $\eps > 0$, such that $\forall r < \eps$ we have \[\phi_*(D(a, r)) = D(a, r).\]
\end{enumerate}
\end{prop}

\begin{proof}
 Given $d = \deg_a(\phi)$ for all sufficiently small $r$, only the first non-zero term (the $d$th) in $\phi_2(y)-\phi^*(a)$ will dominate. For all such $r$, \[\phi_*(D(a, r)) = D(a, \abs{b_d}^\q r^{d\q}).\]
 For $d\q > 1$, we can see that $r > \abs{b_d}^\q r^{d\q}$ for any $r < \abs{b_d}^{-\q/(d\q-1)}$. 
 Similarly when $d\q < 1$, we can see that $r < \abs{b_d}^\q r^{d\q}$ for any $r < \abs{b_d}^{\q/(1-d\q)}$. 
 Finally if $d \q = 1$ then the image diameter is $\abs{b_d}^\q r$ so it only depends on the size of $b_d$.
\end{proof}

A strange quirk of the scale factor $\q$ is that it tends to make classical points attract when it is large whilst it makes hyperbolic $\zeta \in \bH$ repel; conversely when $\q < 1$ it generally causes classical points to repel but other points to attract!


The second stark difference with skew products is the existence of infinitely many fixed classical points. For a rational map this quantity is controlled by the fixed point equation $\phi(z) = z$ which is a degree $d + 1$ equation (not forgetting infinity). The following quick examples serve to illustrate this; most of them give at least $\#\k$ many fixed points.

\begin{ex}\label{ex:dyn:infinitelymanyfixedone}
 Choose any $\phi_1(x) \in \k(x)$. Then the (skew) product map $(x, y) \mapsto (\phi_1(x), y)$ fixes every $c \in \k \subset \K$. 
\end{ex}

In fact we can find maps where any coefficient of $\gamma(x)$ may be arbitrary.

\begin{ex}\label{ex:dyn:infinitelymanyfixedtwo}
 Let $\lambda \in \k$, $m \in \Z$. Then the product map $\phi : (x, y) \mapsto (\lambda x, \lambda^m y)$ fixes $a = cx^m$ for any $c \in \k$. In general, \[(x, y) \longmapsto \left(\phi_1(x), \frac{\phi_1(x)^k y}{x^m}\right)\] fixes $cx^m$. This is just a conjugation of the previous example, with the rational map $x^my$.
\end{ex}

It seems that for a skew product $\phi_*$ and almost any first derivative of $\overline \phi$ the reduction map will give an uncountable family.

\unsure{Too long; didn't read?}
\begin{prop}\label{prop:constructfixedpoint}
 Let $\lambda \in \C$ with $\lambda^n \ne \lambda\ \forall n \ne 1$, and let $m \in \N_+$. Let $\phi_*$ be a $\C$-rational skew product of the form \[(x, y) \longmapsto \left(\lambda x,\, \sum_{n=1}^\infty g_n y^n\right)\] where $g_n \in \C[[x]]$ and $g_1(0) = \lambda^m$. Then $\phi_*$ has an uncountable family of fixed points which can be written $\gamma(x) = \sum_{n=1}^\infty c_n x^n$ where $c_m$ is freely chosen.
\end{prop}

\begin{proof}
We wish to find a solution to the equation $\phi_*(\gamma) = \gamma$. This is equivalent to \[\gamma(\lambda x) = \phi_2(x, \gamma(x)).\]
We will prove this holds using all the $x$-derivatives at $x=0$, finding an appropriate solution $\gamma(x)$ along the way. For the zeroth derivative we have
\[\gamma(0) = \sum_{n=1}^\infty g_n(0) (\gamma(0))^n.\]
Picking $\gamma(0) = 0$ will make the rest of the proposition easy. For the first derivative we have
\[\lambda\gamma'(\lambda x) = \sum_{n=1}^\infty a_n'(x) (\gamma(x))^n + ng_n(x) (\gamma(x))^{n-1}\gamma'(x) \equiv \sum_{n=1}^\infty F_{1, n}(x) + G_n(x)\gamma'(x),\]
which at $x = 0$ (and $\gamma(0) = 0$) is
\[\lambda\gamma'(0) = \sum_{n=1}^\infty a_n'(0) (0)^n + ng_n(0) (0)^{n-1}\gamma'(0) = g_1(0)\gamma'(0)\]
If $k = 1$ then $g_1(0) = \lambda$ and we find that $\gamma'(0)$ is a free choice according to this equation, as desired. Otherwise we must choose $\gamma'(0) = 0$. Above we chose $F_{1, n}(x)$ and $G_n(x)$ so that $F_{1, n}$ is a homogeneous polynomial degree $n+1$ in $g_n(x)$, $\gamma(x)$ and their derivatives, and similarly $G_n$ is homogeneous degree $n$. We inductively define $F_{r, n}(x)$ with the same properties by differentiating.
\begin{align*}
 \td{}{x} (F_{r-1, n}(x) + G_n\gamma^{(r-1)}(x)) =&\ F_{r-1, n}'(x) + G_n'(x)\gamma^{(r-1)}(x)\hspace{-0.7cm} &&+\ G_n(x)\gamma^{(r)}(x)\\
 \equiv&\ \hspace{1.2cm} F_{r, n}(x)\ &&+\ G_n(x)\gamma^{(r)}(x).
\end{align*}
Then we also have that (by chain and product rule) $F_{r, n}(x)$ is a homogeneous polynomial degree $n+1$ in $g_n(x)$, $\gamma(x)$ and their derivatives. By induction, the polynomials $F_{r, n}(x)$ and $G_n(x)$ contain no $\gamma$ derivatives of order $r$ or higher.

Note that $G_n(0) = 0$ for every $n \ge 2$ because $\gamma(0) = 0$. Also, $G_1(0) = \phi_1(0) = \lambda^m$. Note also that if $\gamma(0) = \gamma'(0) = \cdots = \gamma^{(r-1)}(0) = 0$ then $F_{r, n}(0) = 0\ \forall n \ge 1$ since every one monomial in $F_{r, n}$ has a product of $n$ of these, by product and chain rule.

%
%
%
%
 
 Now we will prove inductively that we can satisfy the equation $\gamma(\lambda x) = \phi_2(x, \gamma(x))$ for every $r$-th derivative at $x=0$ by prescribing an appropriate $\gamma^{(r)}(0)$, or in the case of $m=r$, leaving free choice.
  \[\lambda^r \gamma^{(r)}(\lambda x) = \sum_{n=1}^\infty F_{r, n}(x) + G_n(x)\gamma^{(r)}(x).\]
 \[\therefore (\lambda^r - \lambda^m)\gamma^{(r)}(0) = \sum_{n=1}^\infty F_{r, n}(0)\]
 If $k \ge r$ then we had $\gamma(0) = \gamma'(0) = \cdots = \gamma^{(r-1)}(0) = 0$ thus far and so $F_{r, n}(0) = 0\ \forall n$ by the remark above.
 
  If $m = r$ then the equation is satisfied and $\gamma^{(r-1)}(0)$ is a free choice as required.
  If $m \ne r$ then by hypothesis $\lambda^r \ne \lambda^m$ and we simply choose \[\gamma^{(r-1)}(0) = \frac{1}{\lambda^r - \lambda^m}\sum_{n=1}^\infty F_{r, n}(0),\] to ensure that the equation for the second derivative is satisfied.
 If $m > r$ then we just defined $\gamma^{(r-1)}(0) = 0$.
\end{proof}


\begin{qtn}\label{conj:fixedpoint}
 Let $\phi = (\phi_1, \phi_2)$ be a simple $\k$-rational skew product of relative degree $d$. Suppose $\phi_1(x) = \lambda x + \bO(x^2)$. Is the following true?
 
$\phi_*$ has infinitely many fixed points if and only if after a suitable conjugation $\overline{\phi} = y$ (i.e.\ Gauss point is fixed and every direction there is fixed), and $\lambda^n \ne \lambda\ \forall n \ne 1$.

 Otherwise, $\phi$ has $d + 1$ fixed points counted with multiplicity, as expected from the algebraic case, and $\lambda$ is a (complex) root of unity.
 
 If at least one direction is true, then by considering that all but finitely many directions have degree $1$, we can conclude that there are \emph{rays of fixed points} of the form $[a, \zeta]$ in all but finitely many directions at $\zeta$.
\end{qtn}


On the other hand, a skew product can have \emph{zero} classical fixed points!

\begin{ex}\label{ex:dyn:zerofixed}
 Define a skew automorphism $\phi^* : \K(\C)(y) \to \K(\C)(y)$ as follows. Let $\phi^*$ be the complex conjugation field automorphism exchanging $i$ and $-i$, denoted $\overline a = \phi^*(a)$. Let $\phi_2(y) = -1/y$. Any fixed point $a$ satisfies $a = \phi_*(a) = (\phi_1^*)^{-1} \circ \phi_2(a) \implies \phi_2(a) = \phi_1^*(a)$, meaning $\overline a = -1/a$. This map has good reduction and in the residue field we see that $c\overline c = -1$ where $c \in \C$ is the reduction of $a$; this is impossible.
\end{ex}

Recall the following result giving existence of classical fixed points for power series.

\begin{prop}[\cite{Bene} Proposition 4.17]
 Let $r > 0$, let $a \in \hv$, and let $U$ be either $D(a, r)$ or $\CD(a, r)$. Let $\phi \in \hv[[y-a]]$ be a power series converging on $U$ such that $\phi-a$ has finite Weierstrass degree $d$ on $U$. Suppose that $\phi(U)\cap U \ne \emp$, and that either
\begin{enumerate}[label=(\alph*), ref=\theenumi]
\item $\phi(U) \ne U$, 
\item $d \ne 1$, or
\item $d=1$ and $\abs{\phi'(a)-1} = 1$.
\end{enumerate}
Then the disk $U$ contains a fixed point of $\phi$.
\end{prop}

Unfortunately, such a statement would be false for skew products. The proof relies on analysis of the power series $\phi(y)-y$ whose solutions are fixed points. In the case of a skew product $\phi_*$, the function $\phi_* - \id$ is generally not a even a skew product, so a priori $\phi_*(y) = y$ need not have any solutions in $\P^1$. It appears that we should rely on geometry and dynamics; when $\phi(U) \subsetneq U$ we indeed have an attracting fixed point, see \autoref{prop:dyn:attrfixed}. When $\phi(U) \supsetneq U$ it is very difficult to find Type I fixed points, and it is possible \improvement{write example} that the backward orbit of this disk does not shrink to a (Type I) point, however this at least guarantees a fixed Berkovich point in $\CD_\an(a, r)$.

\begin{cor}\label{cor:dyn:approxfixed}
 Let $\phi_*$ be a skew product of scale factor $\q$, let $r > 0$, $a \in \hv$, and let $U$ be either $D(a, r)$ or $\CD(a, r)$. Suppose $\phi_2$ has a power series converging on $U$ such that $\phi_2-a$ has finite Weierstrass degree $d$ on $U$. Suppose that $\phi_*(U)\cap U \ne \emp$ and that either
 \begin{enumerate}[label=(\alph*), ref=\theenumi]
\item $\phi_*(U) \ne U$, or
\item $d \ne 1$.
\end{enumerate}
Then there classical point $b \in U$ with $\abs{\phi_*(b)-b} \le r^\q$, and also $\phi_*$ has a fixed point $\zeta \in \CD_\an(a, r)$.\improvement{We can probably rule out $\zeta$ Type IV using the argument in the following theorems}
\end{cor}

\begin{proof}
 By changing coordinates, we can ensure that  $U = \CD(0, r)$ and $\phi_*(U) = \CD(0, s)$ with $r, s < 1$ (the open case is similar). Then $U$ must intersect $\CD(0, s^{1/\q}) = \phi_{1*}^{-1}(\phi_*(U)) = \phi_{2*}(U)$ intersects $U$. By \cite[Proposition 4.17]{Bene}, there is a classical fixed point $b$ for $\phi_2$ in $U$. Then $\phi_*(b)$ has absolute value $\abs{\phi_*(b)} = \abs b^\q \le \max \set{r^\q, r}$.
 
 If $\phi_*(\CD(a, r)) = \CD(a, r)$ then $\zeta(a, r)$ is fixed. If $\phi_*(\CD(a, r)) \supsetneq \CD(a, r)$ then we consider the strictly decreasing sequence \[U = U_1 \supsetneq U_2 \supsetneq U_3 \supsetneq \cdots\] where $\phi_*(U_{n+1}) = U_n$. 
 If these disks have empty intersection then the corresponding Type IV point is a fixed point. Otherwise their intersection is a point or a disk $\CD(b, s)$ with $\diam(\phi_*^{-n}(U))$ converging to the same diameter; then the corresponding Berkovich point is fixed by $\phi_*$.
 Suppose $\phi_*(\CD(a, r)) \subsetneq \CD(a, r)$. If $\q \ge 1$ then the next proposition finds a unique Type I fixed point. If $\q < 1$ then similar to the last case we can at least find a fixed point $\zeta \in D_\an(a, r)$ by considering the sequence \[U \supsetneq \phi_*(U) \supsetneq \phi_*^2(U) \supsetneq \phi_*^3(U) \supsetneq \cdots\] 
\end{proof}

\begin{prop}\label{prop:dyn:attrfixed}
 Let $\phi_* : \P^1_\an(\hv) \to \P^1_\an(\hv)$ be a skew product of scale factor $\q \ge 1$, let $r > 0$, $a \in \hv$, and let $U$ be either $D(a, r)$ or $\CD(a, r)$. Suppose that $\phi_*(U) \subsetneq U$, then $U$ contains a unique attracting fixed point $a \in U$.
\end{prop}

\begin{proof}
We take the closed case, the open one is similar. Change coordinates so that $a \in \phi_*(U)$ and write \[\phi_2(y) = b_0 + b_1(y-a) + b_2(y-a)^2 + b_3(y-a)^3 + \cdots\]
Then $\max_{n \ge 1} \abs{b_n}^\q r^{n\q} < r$ so $C = \max_{n\ge 1} \abs{b_n}^\q r^{n\q -1} < 1$. Note that for every $s < r$ we have $\max_{n \ge 1} \abs{b_n}^\q s^{n\q -1} \le C$. Let $r = r_0$ and $\phi_*(\CD(a, r_0)) = \CD(a, r_1)$. The important observation is that $C = r_1/r_0$ did not depend on our choice of $a$ in $\CD(a, r_0)$, which we will change soon. Next we get that $\phi_*(\CD(a, r_1)) = \CD(\phi_{1*}(b_0), r_2) \subset \CD(a, r_1)$. Now by replacing $a$ with $\phi_{1*}(b_0)$ (and hence the values of $b_j$ for $j \ge 1$) we preserve $C$ and see that
\[r_2 = \max_{n \ge 1} \abs{b_n}^\q r_1^{n\q} < r_1\max_{n \ge 1} \abs{b_n}^\q r_0^{n\q -1} \le Cr_1.\] Therefore $r_2/r_1 \le C$. Continuing this way, we define $\phi_*(\CD(a, r_n)) = \CD(\phi_*(a), r_{n+1})$ and can conclude that $r_{n+1} \le Cr_n$ for every $n$. Thus $\diam(\phi_*n(U)) = r_n \le C^{n-1}r_1$. Since this converges to $0$ and $\hv$ is complete, the intersection must contain some classical point $\tilde a$ which is fixed by $\phi_*$.
\end{proof}

\begin{defn}
 Let $\phi_*$ be a skew product, suppose that $\zeta \in \P^1_\an$ is fixed, and $\bvec v \in \Dir\zeta$ a direction. We say that $\bvec v$ is \emph{exceptional} iff $\bvec v$ is periodic under $\phi_\#$ and $\deg_{\zeta, \bvec v}(\phi) = \deg_\zeta(\phi)$.
\end{defn}

This definition is only really interesting when $\zeta$ is Type II. The following proposition lays out the basic facts and is left as an exercise to the reader.

\begin{prop}\label{prop:dyn:excdir}
 Let $\phi_*$ be a skew product, suppose that $\zeta \in \P^1_\an$ is fixed. The following facts hold.
\begin{enumerate}
 \item A direction $\bvec v$ is exceptional if and only if $\phi_\#^{-n}(\bvec v) = \set{\bvec v}$ for some $n \ge 1$.
 \item Hence if $\zeta$ is Type I, III, or IV, then every direction at $\zeta$ is exceptional.
 \item When $\zeta$ is Type II, then $\bvec v = \vec v(a)$ is exceptional if and only if $\overline a$ is exceptional for $\overline \phi$ i.e. $\overline a$ is periodic for $\overline \phi$ and $\deg_{\overline a}(\overline\phi_2) = \deg(\overline\phi_2) = \deg_{\zeta, \vec v(a)}(\phi)$.
 \item Finally, if $\deg_\zeta(\phi) = 1$ then any periodic direction is exceptional.
\end{enumerate}
\end{prop}

\begin{thm}\label{thm:dyn:typeIVindiff}
 Let $\phi_*$ be a skew product of scale factor $\q \ge 1$, and let $\zeta \in \P^1_\an$ be a periodic point of Type IV. Then $\zeta$ is indifferent and $\q = 1$.
\end{thm}

\improvement{What about $\q < 1$?}

\begin{rmk}
 This is immediate if $\q = 1$ and $\phi_*$ is tame, because then also its local degree is $1$ by \autoref{cor:skew:TypeIVdegone}. Moreover, putting these two results together shows that if $\phi_*$ is tame with $\q \ne 1$ then none of its Type IV points are fixed.
\end{rmk}

\begin{proof}
 Let $\zeta$ be a Type IV point with degree $d = \deg_\zeta(\phi)$ and diameter $r_0 = \diam(\zeta)$. By replacing $\phi_*$ with $\phi_*^n$ and $\q$ by $\q^n$ we may assume $\zeta$ is fixed. We may also change coordinates to ensure that $r_0 < 1$.
 
 By \autoref{thm:skew:bigintervalstretch} there is an interval $[\zeta, \xi]$ which is uniformly expanded by a factor of $d\q$. Let $r$ be such that $\xi = \zeta(a, r)$, then by the aforementioned theorem we can assume that $\phi_*(\xi) = \phi_*(\zeta(a, r)) = \zeta(a, (r/r_0)^{d\q}r_0)$ and moreover $\phi_*(\CD_\an(a, r)) = \CD_\an(a, (r/r_0)^{d\q}r_0) \supsetneq \CD_\an(a, r)$; assume $r$ is small enough that there are no poles in this neighbourhood. Furthermore,  if $\q < 1$, then shrink $r$ further until $r^\q > (r/r_0)^{d\q}r_0$. This is possible since as $r \to r_0$ the right side converges to $r_0$ but the left converges to $r_0^\q > r_0$. By \autoref{cor:dyn:approxfixed} there is an approximate fixed point $b \in \CD(a, r)$ such that either $b$ is fixed or $\abs{\phi_*(b) - b} = \max{r, r^\q}$. Now define $s \le r$ and such that $\xi' = \zeta(b, s) = b \vee \zeta$; clearly $\vec v(b) \ne \vec v(\zeta)$. There are now a couple of cases.
 
 Suppose that either $b$ is fixed or that $\q \ge 1$, so either way $\abs{\phi_*(b) - b} \le r$ and $\phi_*(b) \in \CD(a, r)$. Therefore at $\phi_*(\xi') \succ \xi'$, we have $\phi_*(b) \in \vec v(\xi') = \vec v(\zeta)$. Since there are no poles in this neighbourhood, all directions are good so $\phi_\#(\vec v(b)) = \vec v(\phi_*(b)) = \vec v(\zeta)$. However by \autoref{thm:skew:bigintervalstretch}, $\vec v(\zeta)$ is exceptional i.e. $\phi_\#$ is a degree $d$ mapping on directions at $\xi'$ and also $\deg_{\xi', \vec v(\zeta)}(\phi) = d$. By \autoref{prop:dyn:excdir}, $\phi_\#^{-n}(\bvec v) = \set{\bvec v}$ so we cannot have $\phi_\#(\vec v(b)) = \phi_\#(\vec v(\zeta))$; contradiction.
 
  \unfinished{Unfinished! I conjecture that this works for $\q < 1$.}
\end{proof}

\begin{thm}\label{thm:dyn:typeIIIindiff}
 Let $\phi_*$ be a skew product of scale factor $\q$, and let $\zeta \in \P^1_\an$ be a periodic point of Type III. Assume that $\q$ is algebraic over $\Q$ or more generally that the value group is a $\Q(\q)$-module. Then $\zeta$ is indifferent and each direction at $\zeta$ is fixed by $\phi_\#^n$.
\end{thm}

\begin{rmk}\label{rmk:dyn:algebraic}
 The unusual requirement that $\q$ is algebraic is necessary for showing that Type III periodic points are indifferent. When we say that the value group is a $\Q(\q)$ - module, we precisely mean that in the sense that for any $1 \ne \abs a \in \abs{K^\times}$, the $\Q$-vector space $\log_{\abs a}\abs{K^\times}$ is a $\Q(\q)$-module.
 Observe that since the skew product is well-defined, for every $a \in K$, $\abs {\phi_*(a)} = \abs a^\q$ and thus \[\log_{\abs a}\abs{K^\times} \supset \Q\langle \q^n : n \in \Z\rangle,\] so the value group is naturally already a $\Q[\q , \q^{-1}]$-module. 
  If the value group is $\Q$ then $\q$ must be rational and if algebraic then $\Q[\q , \q^{-1}] = \Q(\q)$.\\
  The necessity of this hypothesis is shown in the following counterexample.
\end{rmk}

\begin{ex}\label{ex:dyn:typeIIIrep}
 Let $e$ be Euler's number and consider the field \[K = \C((x))\left(x^{\frac{1}{m}e^n} : n \in \Z, m \in \N_+\right)\] with the `order of vanishing' absolute value. Then $K$ is an extension of the Puiseux series $\K$ and has value group $\abs{K^\times} \cong \Q\langle e^n : n \in \Z \rangle$. Since the value group is already a $\Q$-vector space, the algebraic closure and its completion have the same value group.
 
 Now consider the skew product defined by \[\phi(x, y) = (x^{1/e}, y/x).\]
 Then $\q = e$ is not algebraic and the claim is that $\zeta = \zeta(0, \abs x^{e/(e-1)})$ is a Type III fixed point.
 
 If $e/(e-1)$ were in the value group then we must have \[\frac{e}{e-1} = \sum_{j=1}^N r_je^{n_j}\] for some $r_j \in \Q$ and $n_j\in \Z$. But then \[0 =  -e + \sum_{j=1}^N r_je^{n_j+1} - r_je^{n_j}\] shows that $e$ is algebraic - contradiction. So indeed $\zeta$ is Type III.
 
 By \autoref{thm:skew:imgtypeII}, \[\phi_*(\zeta) = \zeta\left(\phi_{1*}(0), \left(\abs{x}^\frac{e}{e-1}/\abs{x}\right)^e\right) = \zeta\left(0, \abs{x}^{e\left(\frac{e}{e-1} - \frac{e-1}{e-1}\right)}\right) = \zeta\left(0, \abs{x}^{\frac{e}{e-1}}\right) = \zeta\]
 Finally, since $\deg(\phi)=1$, and $\q = e >1$, the multiplier of $\zeta$ is $\q > 1$.
 We have shown that $\zeta$ is a fixed, repelling Type III point.
\end{ex}

\begin{proof}[Proof of \autoref{thm:dyn:typeIIIindiff}]
Change coordinates so that $\zeta = \zeta(0, r)$, where $r \nin \abs{K^\times}$. Pick an annulus $U = \set{r\lambda < \abs y < r}$ as in \autoref{thm:skew:imgtypeII} where $\phi_2$ has a Laurent series \[\phi_2(y) = \sum_{n\in \Z} b_n y^n\] with Weierstrass degree $d$. Then we have\[\zeta(0, r) = \zeta\left(0, \abs{b_d}^\q r^{d\q}\right).\]
Therefore $r = \abs{b_d}^\q r^{d\q}$ which implies $\abs{b_d}^\q = r^{1-d\q}$. If $d\q = 1$, then by definition $\zeta$ is indifferent, so we are done. Otherwise (possibly using \autoref{rmk:dyn:algebraic}) we can write \[r = \abs{b_d}^{\q/(1-d\q)}.\] Since the value group is a $\Q(\q)$-module we can multiply and divide exponents by $1-d\q$ preserving the value group, meaning $\abs{b}^{\frac 1{1-d\q}} \in \abs{K^\times} \iff \abs b \in \abs{K^\times}$. Therefore $r \in \abs{K^\times}$; we have our contradiction.\\
\end{proof}

\section{Fatou and Julia}\label{sec:fj}

\begin{defn}
Let $\phi_*$ be a skew product. We say an open set $U \subseteq \P^1_\an$ is \emph{dynamically stable} under $\phi_*$ iff $\displaystyle \bigcup_{n \ge 0} \phi_*^n(U)$ omits infinitely many points of $\P^1_\an$.

The \emph{(Berkovich) Fatou set} of $\phi_*$, denoted $\F_{\phi, \an}$, is the subset of $\P^1_\an$ consisting of all points $\zeta \in \P^1_\an$ having a dynamically stable neighbourhood.

The \emph{(Berkovich) Julia set} of $\phi_*$ is the complement $\J_{\phi, \an} = \P^1_\an \sm \F_{\phi, \an}$ of the Berkovich Fatou set.
\end{defn}

The next proposition generalises \cite[Proposition 8.2]{Bene}. 

\begin{prop}\label{prop:fatoujuliabasics}
 Let $\phi$ be a skew product with Berkovich Fatou set $\F_{\phi, \an}$ and Berkovich Julia set $\J_{\phi, \an}$.
 
\begin{enumerate}[label=(\alph*), ref=\theenumi]
 \item $\F_{\phi, \an}$ is an open subset of $\P^1_\an$ and $\J_{\phi, \an}$ is a closed subset of $\P^1_\an$.
 \item $\phi_*^{-1}(\F_{\phi, \an}) = \phi_*(\F_{\phi, \an}) = \F_{\phi, \an}$ and $\phi_*^{-1}(\J_{\phi, \an}) = \phi_*(\J_{\phi, \an}) = \J_{\phi, \an}$.
 \item For every integer $m \ge 1$, $\F_{\phi^m,\an} = \F_{\phi, \an}$, and $\J_{\phi^m,\an} = \J_{\phi, \an}$.
 \item For any $\eta \in \PGL(2, K)$, if we set $\psi = \eta \circ f \circ \eta^{-1}$, then
 \[\F_{\psi,\an} = \eta(\F_{\phi, \an}) \quad \and\quad \J_{\psi,\an} = \eta(\J_{\phi, \an}).\]
\end{enumerate}
\end{prop}

%
 

 Let $\phi_* = \phi_{1*} \circ \phi_{2*}$ be a skew product. Recall that we say $\phi_*$ has explicit good reduction iff $\phi_2$ has explicit good reduction as a rational function.

%

\unsure[inline]{Contrapositive phrasing. Dear committee, please tell me: Which is better?}

\begin{thm}\label{thm:perptdynone}
 Let $\phi_*$ be a skew product of scale factor $\q$, and let $\zeta \in \P^1_\an$ be a periodic point of Type II, III, or IV of period $n \ge 1$.
 Then $\zeta \in \J_{\phi, \an}$ is Julia if and only if either
\begin{enumerate}
 \item $\zeta$ is \weaklyrepelling{} i.e. $\deg_\zeta(\phi)\q > 1$; or
 \item $\zeta$ is Type II and either none of the directions $\bvec v \in \Dir\zeta$ are exceptional or any of the bad directions $\bvec v \in \Dir\zeta$ is not exceptional.
\end{enumerate}
 Moreover if $\zeta \in \F_{\phi, \an}$ is Fatou then every direction intersecting $\J_{\phi, \an}$ is exceptional and is $\phi_\#^{n'}(\bvec v)$ for some bad direction $\bvec v$.
\end{thm}

We state the contrapositive for the benefit of the reader.

\begin{thm}[\autoref{thm:perptdynone} contrapositive]
 Let $\phi_*$ be a skew product of scale factor $\q$, and let $\zeta \in \P^1_\an$ be a periodic point of Type II, III, or IV of period $n \ge 1$.
 Then $\zeta \in \F_{\phi, \an}$ is Fatou if and only if both
\begin{enumerate}
 \item $\zeta$ is \weakly{} non-repelling i.e. $\deg_\zeta(\phi)\q \le 1$; and
 \item at least one direction $\bvec v \in \Dir\zeta$ is exceptional and also every bad direction is exceptional when $\zeta$ is Type II. \improvement{Could remove Type II}
\end{enumerate}
 Moreover if $\zeta \in \F_{\phi, \an}$ is Fatou then every direction intersecting $\J_{\phi, \an}$ is exceptional and is $\phi_\#^{n'}(\bvec v)$ for some bad direction $\bvec v$. 
\end{thm}

For a Type II point, we can rephrase this theorem a la Benedetto:
If $\zeta$ is Type III or IV then $\zeta$ non-repelling if and only if it is Fatou. 
 In particular by \autoref{thm:dyn:typeIIIindiff}, if $\abs{K^\times}$ is a $\Q(\q)$-module then every Type III point is indifferent and Fatou. By \autoref{thm:dyn:typeIVindiff}, \change{conjecture when $\q < 1$} every Type IV point is indifferent and Fatou when $\q \ge 1$. 
 If $\zeta$ is of Type II and \weakly{} non-repelling, let $T$ be the finite set of bad directions of $\phi_*^n$ at $\zeta$, as in \autoref{thm:skew:reduction}. If every $\bvec v \in T$ is exceptional under the action $\phi_\#^n$ at $\zeta$, then there is an open Berkovich affinoid $U$ containing $\zeta$ such that $\phi_*^n(U) \subseteq U$, and hence $\zeta \in \F_{\phi, \an}$. Otherwise $\zeta \in \J_{\phi, \an}$.

\begin{rmk}
 Observe that any \weaklyrepelling{} Type II point must be Julia, even when it is saddle or attracting! These are easy to construct. \improvement[inline]{Write and link to examples. e.g. $(x^2, y(y-1)/(y-x))$.}
\end{rmk}
 
 

The following lemma says that a repelling direction of a Type II, III or IV point will fill out the whole direction under iteration; this is needed to prove the next theorem. It is inspired by \cite[Lemma 8.11]{Bene} which only applies to Type II points.

\begin{lem}\label{lem:dyn:repellingdirn}
 Let $\phi_*$ be a skew product and let $\zeta \in \P^1_\an$ be Type II, III or IV. Suppose that $\phi_*$ fixes both $\zeta$ and also a direction $\bvec v$ at $\zeta$ with multiplicity $m = \deg_{\zeta, \bvec v}(\phi)$ and multiplier $m\q > 1$. If $\zeta$ is Type IV let $U$ be a disk neighbourhood of $\zeta$, otherwise let $U$ be a Berkovich annulus in $\bvec v$ with $\zeta$ as one boundary point. Then $\bigcup_n \phi_*^n(U)$ contains all but at most one point of $\bvec v$. Moreover, if there is one such point $b \in \bvec v$, then it is of Type I, and for any disk $D \ni b$, there is some $N \ge 1$ such that \[\bvec v \sm D \subseteq \phi_*^N(U).\]
\end{lem}

\begin{proof}
First suppose $\zeta = \zeta(a, R)$ but change coordinates in the domain and range such that $a = 0$, $\bvec v = \vec v(0)$, and $R < 1$ (possibly irrational), we may assume that $\vec v(0)$ is fixed with multiplier $m\q > 1$ where $m = \deg_{\zeta(0, R), \vec v(0)}(\phi)$. Then let 
\[U = \set{r < \abs{\zeta} < R} = D_\an(0, R)\sm D_\an(0, r),\] for some radii $0 < r < R$.
We may also assume $r$ is large enough so that $U = U_1$ has inner and outer Weierstrass degree $m$ and generally we define $U_{n+1} = \phi_*(U_n)$. It follows that $U \subseteq U_2 = \set{(r/R)^{m\q}R< \abs\zeta < R}$ moreover we can deduce that
\[U = U_1 \subseteq U_2 \subseteq U_2 \subseteq U_3 \subseteq \cdots.\]
We claim that every $U_n$ either contains $D_\an(0, R)$ or is of the form $\set{r_n < \abs{\zeta - a_n} < R}$ where $r_n = (r/R)^{(m\q)^n}R$. The proof is by induction and is clear in the initial case. Now suppose the claim is true for $U_n$. If it contains $D_\an(0, R)$ then done, otherwise $U_n$ is an open annulus with two boundary $\zeta(0, R)$ and $\zeta(a_n, r_n)$. Further suppose that $U_{n+1} \nsupset D_\an(0, R)$, does not contain some point namely $a_{n+1}$. Consider the map $\psi_* \circ \phi_*$ where $\psi_*$ is the rational map $\psi(y) = 1/(y - a_{n+1})$. Then $\psi_* \circ \phi_*$ does not have any poles in $U_n$. This map has outer Weierstrass degree $\wdeg_{0, R}(\psi_* \circ \phi_*) = -m$ and inner Weierstrass degree $\overline\wdeg_{a_n, r_n}(\psi_* \circ \phi_*) \le -m$. Observe that $\psi_*(\zeta(0, R)) = \zeta(0, R^{-1})$ and $\psi_*(\zeta(0, r_n)) = \zeta(0, r_n^{-1})$.

If the inner Weierstrass degree is equal to $-m$ then by \autoref{thm:seminorms:mapannulus} the image is the annulus $\set{s < \abs\zeta < t}$ where $s = CR^{-m\q}$ and $t = Cr_n^{-m\q}$ for some $C > 0$ given by a coefficient in a Laurent expansion. We can deduce that $s = R^{-1}$ because of the fixed boundary point, and thus $t = (R/r_n)^{m\q}R^{-1} = 1/r_{n+1}$. Applying $\psi_*^{-1}$ gives us that $U_{n+1} = \set{r_{n+1} < \abs{\zeta - a_{n+1}} < R}$.

Otherwise if the inner Weierstrass degree is not $-m$, then \autoref{thm:seminorms:mapannulus} says that $\psi_* \circ \phi_*(U_n) = D(0, \max\set{s, t})$ where similarly to the above we can deduce that $s = R^{-1}$ and $t > 1/r_{n+1}$. Therefore $U_{n+1} \supset \P^1_\an \sm \CD(a_{n+1}, r_{n+1})$.

To finish the proof of this case, suppose that none of the $U_n$ are equal $D_\an(0, R)$. Given that $r_n \to 0$ as $n \to \infty$, we have that $\bigcap_n \CD(a_n, r_n)$ can only be a Type I point, say $b \in \bvec v$. Then $U_n = D_\an(b, R) \sm \CD_\an(b, r_n)$ and the rest is clear.

Now we address the case that $\zeta$ is Type IV - this is easier and similar to the above. If $\bigcup_n \phi_*^n(U) \ne \P^1_\an$ then it must omit a Type I point so we can move this to $\infty$. Let $\zeta$ have diameter $R$ and let $\eps > 0$ be very small and pick $a \in D(a, R+\eps)$ such that $\zeta(a, R + \eps) \in [\zeta, \infty]$ has local degree $m$. Then $U = U_0 = D_\an(a, R+\eps)$ contains $\zeta$; we let $r = R+\eps$ and we define $U_{n+1} = \phi_*(U_n)$. We claim that every $U_n$ is of the form $D_\an(a, r_n)$ where $r_n \ge (r/R)^{(m\q)^n}R$\unsure{check/rewrite}. Given the claim we are done since $r_n \to \infty$ and so $\bigcup_n \phi_*^n(U)$ omits at most the point $\infty$.

Firstly, $U_1 = D_\an(a, (r/R)^{m\q}R)$ is clear from \autoref{thm:skew:bigintervalstretch}.
Assume the case of $n \ge 1$, and recall that $U_n$ is a disk not containing any poles of $\phi_*$. Let $U_n' = \set{r < \abs{\zeta - a} < r_n}$. The $\phi_*$ on $U_n'$ has outer Weierstrass degree $\wdeg_{a, r_n}(\phi_*) = \ge m$ and inner Weierstrass degree $\overline\wdeg_{a, r}(\phi_*) = m$. Recall that in the absence of poles, Weierstrass degree can only increase with diameter, see \autoref{prop:seminorms:rationaldegrees}. 
If the outer Weierstrass degree is equal to $m$ then by \autoref{thm:seminorms:mapannulus} the image is the annulus $\set{s < \abs\zeta < t}$ where $s = Cr^{m\q}$ and $t = Cr_n^{m\q}$ for some $C > 0$ given by the $m$th coefficient in a Taylor expansion. Otherwise if the outer Weierstrass degree is not $m$, then in fact $t > Cr_n^{m\q}$. We can deduce that $s \approx R$ because of the fixed boundary point, and thus $t = r_{n+1} \ge (r_n/R)^{m\q}R$ and $U_{n+1} = D_\an(a, r_{n+1})$ as desired.
\end{proof}

Now we are ready to prove the theorem. This generalises the similar result for rational maps \cite[Theorem 8.7]{Bene} to skew products. However the situation is far more complicated in the case of Type III and IV points. Without the hypotheses and results of \autoref{thm:dyn:typeIIIindiff}, \autoref{thm:dyn:typeIVindiff}, as we had for rational maps, we cannot just dismiss these points as indifferent (or attracting) a priori to prove they are Fatou; indeed we treat the cases that they are repelling, showing that they are Julia.


\begin{proof}[Proof of \autoref{thm:perptdynone}]
After replacing $\phi_*$ with $\phi_*^n$ we may assume that $\zeta$ is fixed.

First suppose $\zeta$ is Type III. Each direction is exceptional and we may assume fixed after replacing $\phi_*$ by $\phi_*^2$. Let $m = \deg_\zeta(\phi)$. We need to show that $\zeta$ is Fatou if and only if $m\q \le 1$. 
Suppose that $\zeta = \zeta(a, R)$ is Fatou. Then $\zeta$ has an annular neighbourhood of the form $U = \set{R-\eps < \abs{\zeta - a} < R+\eps}$ such that $\bigcup_n \phi_*^n(U)$ omits infinitely many points. By \autoref{lem:dyn:repellingdirn}, if $m\q > 1$ then $\bigcup_n \phi_*^n(U)$ omits at most two points, hence we can conclude that $m\q \le 1$. Conversely, assuming that $d\q \le 1$, \autoref{thm:skew:bigintervalstretch} with \autoref{thm:skew:imgtypeII} shows that both annuli $V_1 = \set{R-\eps < \abs{\zeta - a} < R}$ and $V_2 = \set{R < \abs{\zeta - a} < R +\eps}$ are forward invariant. Therefore the Berkovich open annulus $U = \set{R-\eps < \abs{\zeta - a} < R+\eps}$ contains $\zeta$ and $\phi_*(U) \subseteq U$. Hence $\bigcup_n \phi_*^n(U) \subseteq U$ is a dynamically stable neighbourhood; whence $\zeta \in U \subseteq \F_{\phi, \an}$.

Next suppose $\zeta$ is Type IV and $m = \deg_\zeta(\phi)$; this is very similar to the last case. If $m \q \le 1$ then we can find a neighbourhood $U = D_\an(a, r) \ni \zeta$ such that $\zeta \in \phi_*(U) \subseteq U$ is Fatou. Conversely if $m \q > 1$ then \autoref{lem:dyn:repellingdirn} says that for any neighbourhood $U$ of $\zeta$ we have $\bigcup_n \phi_*^n(U)$ omitting at most one point; whence $\zeta\in \J_{\phi, \an}$ is Julia.

Finally we address the case $\zeta$ is Type II. After a change of coordinates we may assume $\zeta(0, 1)$ is the Gauss point of local degree $m = \deg_\zeta(\phi)$.

First suppose that $m\q \le 1$ and that the directions $\bvec v_1, \dots, \bvec v_N$ are exceptional, $N \ge 1$ and these include all of the bad directions. Note that for any exceptional direction, its (finite) grand (forward and backward) orbit of directions is also exceptional too, so let us assume all these orbits are contained in our list. By further replacing $\phi_*$ by an iterate we may assume that every $\bvec v_j$ is fixed; it is important to note that $\phi_\#^{-1}(\bvec v_j) = \bvec v_j$. Change coordinates so that $\bvec v_j = \vec v(a_j)$. Let $r < 1$ be large enough according to \autoref{thm:skew:bigintervalstretch} so that for every $1 \le j \le N$, we have $\overline \wdeg_{a_j, r}(\phi) = \deg_{\zeta(0, 1)\bvec v_j}(\phi) = m$ and $U_j = U \cap \bvec v_j = D_\an(a_j, 1) \sm \CD(a_j, r)$ has no poles or zeros, mapping to an open annulus in $\bvec v_j$. We construct a dynamically stable neighbourhood as follows.
\[U = \P^1_\an \sm \left(\CD(a_1, r) \cup \cdots \cup \CD(a_N, r)\right)\]
Suppose $\bvec w \ne \bvec v_j$ for any $j$. Then $\bvec w$ is a good direction and $\phi_\#(\bvec w) = \ne \bvec v_j$ for any $j$ because $\bvec v_j$ is exceptional and fixed. Hence $\phi_*(\bvec w) = \phi_\#(\bvec w) \subset U$. Otherwise consider the image of $U_j = U \cap \bvec v_j = D_\an(a_j, 1) \sm \CD(a_j, r)$. By our assumption on the size of $r$, \autoref{thm:skew:bigintervalstretch} says that if $m\q \le 1$ then $\phi_*(U_j) = D_\an(a_j, 1) \sm \CD(a_j, r^{m\q}) \subseteq U_j$. Overall we have just shown, direction by direction, that $\phi_*(U) \subseteq U$, whence $\zeta \in U \subseteq \F_{\phi, \an}$.

Conversely suppose that $\zeta(0, 1) \in \F_{\phi, \an}$ is Fatou. Then there must be a dynamically stable $U \ni \zeta(0, 1)$ and we may assume by \autoref{thm:berk:topbasis} that $U$ is a connected Berkovich open affinoid \[U = \P^1_\an \sm (V_1 \cup \cdots \cup V_M),\] and that each $V_j$ is a closed disk of the form $\CD_\an(a_j, r_j)$.
Let $\bvec v_1, \dots, \bvec v_l$ be the finite (at most $M$ but we will show $l=0$) set of directions not contained in $\bigcup_n \phi_*^n(U)$, and let $S = \set{\overline{a_1}, \dots, \overline{a_l}}$ be the corresponding set of residues in $\k$. Firstly, we observe that any other direction $\bvec w \subset U$, must be a good direction; otherwise we would have $\phi_*(\bvec w) = \P^1_\an \subset \bigcup_n \phi_*^n(U)$ by \autoref{thm:skew:reduction}. 
Further consider \autoref{thm:skew:reduction} and how these are mapped under $\overline \phi$; we can deduce that $\overline \phi^{-1}(S) \subset S$ and hence furthermore, $\overline \phi$ is a permutation on $S$. By replacing $\phi_*$ with an iterate, we may assume each of these directions is fixed and hence $\bvec v_j$ is totally invariant under $\phi_\#$, therefore $\deg_{\zeta, \bvec v_j}(\phi) = \deg_{\zeta}(\phi) = m$ by \autoref{thm:skew:reduction} and \autoref{thm:sumdirdegs}. In conclusion we have shown that each of the directions $\bvec v_j$ is exceptional, including every bad direction. It only remains to show that $m\q \le 1$, below.

We can see that $U$ contains an annulus $U_j = U \cap \bvec v_j$ of the form $\set{r < \abs{\zeta - a_j} < 1}$ which maps into the same direction $\bvec v_j$. Now suppose that $m\q > 1$; then by \autoref{lem:dyn:repellingdirn} $\bigcup_n \phi_*^n(U_j)$ contains all but (at most) one point (WLOG $a_j$) of $\bvec v_j$. Overall, this contradicts our assumption that $U$ is dynamically stable, since \[\bigcup_n \phi_*^n(U) \supseteq \P^1_\an \sm \set{a_1, \dots, a_l}.\]

Finally, since $\bigcup_n \phi_*^n(U)$ is dynamically stable itself, we can infer that it is a Fatou neighbourhood, so the entire Julia set must be contained in the directions $\bvec v_1, \dots, \bvec v_l$. On the other hand, we could construct these from scratch using only the finite orbits of bad directions, proving that every Julia direction is in the orbit of a bad direction.
\end{proof}

\begin{prop}\label{prop:dyn:classicalnonrepelling}
 Let $\phi_*$ be a skew product and $a \in \P^1$ be a periodic classical point which is non-repelling. Then $a \in \F_{\phi, \an}$.
\end{prop}

\begin{proof}
 By \autoref{prop:dyn:classicalfixed}, $a$ has a disk neighbourhood $D_\an(a, r)$ such that $\phi_*(D_\an(a, r)) \subseteq D_\an(a, r)$. Thus $\bigcup_n \phi_*^n(D_\an(a, r)) \subseteq D_\an(a, r)$ omits infinitely many points.
\end{proof}

\begin{rmk}
 The question of repelling points is complicated when $\q < 1$. A superrepelling point can be Fatou, for example consider the following.
\end{rmk}

\begin{ex}\label{ex:dyn:superrepellingfatou}
Define a $\C$-rational skew product by $\phi(x, y) = (x^2, y)$. Every direction at the Gauss point is fixed and attracting, whilst every complex number $c \in \C \subset \K$ is superrepelling. In fact $\J_{\phi, \an} = \emp$. Note however that $\rdeg(\phi)\q = \half < 1$.\unfinished{With more time I want to prove that $\rdeg(\phi)\q \ge 1$ implies the Julia set is non-empty.}
\end{ex}

\begin{prop}\label{prop:dyn:classicalrepelling}
 Let $\phi_*$ be a skew product and $a \in \P^1$ be a periodic classical point which is repelling but not superrepelling (meaning $\q\deg_a(\phi) = 1$ with multiplier $\lambda > 1$). Then $a \in \J_{\phi, \an}$.
\end{prop}

\begin{proof}
 Let $D$ be a disk small enough so that $\phi_*(D) \supsetneq D$, using \autoref{prop:dyn:classicalfixed}. We shall prove that $\bigcup_n \phi_*^n(D)$ omits at most one point, so assume on the contrary that it omits at least one and change coordinates so that $a = 0$ and an omitted point is $\infty$. Now write \[\phi_2(y) = b_dy^d + b_{d+1}y^{d+1} + \cdots\] and where $d = \deg_0(\phi)$ i.e. $b_d \ne 0$. Note that this series converges on any $\phi_*^n(D)$ since $\infty$ is omitted from every such image. Moreover, we can write $\phi_*^n(D) = D(0, r_n)$ with $r_0 = r$ by \autoref{thm:seminorms:mapdisk}. By hypothesis, $a$ is repelling with $\abs{b_d}^\q r^{d\q} = \abs{b_d}^\q r > r$ so we may assume that $\lambda = \abs{b_d}^\q > 1$. 
 For a general $n$, again by \autoref{thm:seminorms:mapdisk} we have that 
 \[r_{n+1} = \abs{b_e}^\q r_n^{e\q} \ge \abs{b_m}^\q r_n^{m\q} \quad \forall m \in \N\] where $e = \wdeg_{0, r_n}(\phi_2)$. In particular we have \[\abs{b_e}^\q r_n^{e\q} \ge \abs{b_d}^\q r_n^{d\q} = \abs{b_d}^\q r_n = \lambda r_n.\]
 This shows inductively that $r_n \ge \lambda^n r$ for every $n \in \N$ and therefore $\bigcup_n D(0, r_n) = \P^1\sm\set\infty$.
\end{proof}


Next we continue to generalise results from \cite[\S8]{Bene}. Some require different hypotheses, but most go through without modification to their proofs. Note that for rational maps the degree $d$ is an integer and $\q=1$, so the condition $d \ge 2$ is the same as $d\q > 1$. The latter is the correct generalisation in the results below because the key idea is that a point with maximal degree (good reduction) must be repelling and Julia.

\begin{prop}
 Let $\phi_*$ be a skew product of relative degree $d$ and scale factor $\q$. If $d\q > 1$ then the following statements are equivalent.
\begin{enumerate}[label=(\alph*), ref=\theenumi]
 \item $\phi$ has explicit good reduction.
 \item The Gauss point $\zeta(0, 1)$ is a fixed point of multiplier $d\q$.
 \item $\phi_*^{-1}(\zeta(0, 1)) = \set{\zeta(0, 1)}$.
 \item The Berkovich Julia set $\J_{\phi, \an}$ is the singleton set $\set{\zeta(0, 1)}$.
\end{enumerate}
\end{prop}

\begin{prop}
 Let $\phi_*$ be a skew product of relative degree $d \ge 2$ and scale factor $\q > 1/d$, and assume that $\abs{K^\times}$ is a $\Q(\q)$-module. Suppose there is a point $\zeta \in \P^1_\an(K) \sm \P^1(K)$ of Type II, III, or IV with finite grand orbit under $\phi_*$. Then there is only one such point, and it is a repelling fixed point of Type II, with multiplier $d\q$.
In that case, $\J_{\phi, \an} = \set{\zeta}$, and there exists $\eta \in \PGL(2, K)$ such that $\eta(\zeta) = \zeta(0,1)$. 
\end{prop}

\begin{cor}
 Let $\phi_*$ be a skew product of relative degree $d \ge 2$ and scale factor $\q > 1/d$. Suppose that $\phi_*^n$ has explicit good reduction for some $n \ge 1$. Then $\phi_*$ itself has explicit good reduction.
\end{cor}

Next we recall some elementary (non-Archimedean) dynamics and extend it to our setting; see \cite[\S1.4]{Bene} for more detail.

\begin{defn}
 Let $K$ be a field, and $\phi_*$ be a skew product of relative degree $\rdeg(\phi)$.
\begin{itemize}
 \item We say $\phi_*$ is \emph{totally ramified} at a point $a \in \P^1(K)$ iff $\deg_a(\phi) = \rdeg(\phi) \ge 2$.
 \item We say that $a \in \P^1(K)$ is an \emph{exceptional point} of $\phi_*$ if the grand orbit $\GO_\phi(a)$ of $a$ is finite.
 \item We define the set $E_{\phi_*}$ of all exceptional points of $\phi_*$, called the \emph{exceptional set} of $\phi_*$.
\end{itemize}
\end{defn}

It is not hard to show that every exceptional point is totally ramified. If $a \in E_{\phi_*}$ then the mapping $\phi_* : \GO_\phi(a) \to \GO_\phi(a)$ is bijective, and one can check it is a cyclic permutation with $\phi_*^{-n}(a) = \set a$ where $n = |\GO_\phi(a)|$. We can therefore limit the number of exceptional points by limiting the number of totally ramified points. Before we state the Riemann-Hurwitz formula we should discuss separability.

If $\chara K = p >0$ and a rational map can be written of the form $\phi(z) = \psi(z^p) \in K(z^p)$ then $\phi(z)$ has  zero derivative. We call such a map \emph{inseparable}, otherwise \emph{separable} (including the case where $\chara K = 0$). Conversely, one can prove that if $\phi(z)$ has zero derivative, it must either be constant or $\chara K = p > 0$ and $\phi(z) = \psi(z^{p^j})$ for some maximal $j \ge 1$ and separable $\psi \in K(z)$. Every point is critical for $\phi$ with multiplicity divisible by $p^j$. Further if $\deg(\psi) = 1$ then we say $\psi$ is \emph{totally inseparable}.

\begin{thm}\label{thm:separabledecomposition}
 Let $\phi_* = \phi_{1*} \circ \phi_{2*}$ be a skew product of relative degree $d$ and scale factor $\q$. If $\chara K = p > 0$ then we can find an \equivar{} skew endomorphism $\psi^*$ such that
 \[\phi_* = \psi_* = \psi_{1*} \circ \psi_{2*}\] is a skew product of degree $d'$ and scale factor $\q'$ where $\psi_2$ is a separable rational map and $d\q = d'\q'$.
\end{thm}
\improvement{Will eventually incorporate this into the ``irrationality'' section. This is the only way skew products can be written non-uniquely.}

\begin{rmk}
 We call this a \emph{separable decomposition}. This shows that in some sense degree was never well defined, but it is for a separable decomposition. Nothing is lost because we have not yet equated two skew products and made claims about the degree. This concept will prove useful in eliminating and understanding unnecessary Frobenius-related pathologies in our analysis.
\end{rmk}

\begin{proof}
As above, we could write $\phi_2(y) = \psi_2(y^{p^j})$ with $\psi_2$ separable. In fact, since $K$ is algebraically closed (perfect), we can \emph{instead} pull the power outside to find a rational $\psi_2$ with $\phi_2(y) = \psi_2(y)^{p^j}$ and then define $\psi_1^*(a) = F_j^{-1} \circ \phi_1^*$ by composing $\phi^*$ with the inverse of the Frobenius automorphism $F_j : a \mapsto a^{p^j}$. On the classical points $\P^1(K)$ we have
  \[\phi_* = \phi_{1*} \circ \phi_{2*} = (\phi_1^*)^{-1} \circ F_j \circ \psi_{2*} = (F_j^{-1} \circ \phi_1^*)^{-1} \circ \psi_{2*} = (\psi_1^*)^{-1} \circ \psi_{2*} = \psi_{1*} \circ \psi_{2*}.\]
  We define the endomorphism in general by $\psi^* = \psi_2^* \circ \psi_1^*$ then and applying \autoref{thm:skew:scaledhomeo} to show $\phi_* = \psi_*$. The last part that $d' = d / p^j$ and $\q' = \q p^j$ is left as an exercise in the definitions.
\end{proof}

Given a critical point $a \in \P^1$, define $\rho_a \ge 1$ to be the order of vanishing of $\phi'$ at $a$; let $\rho_a = 0$ otherwise. Then $\rho_a \ge \deg_a(\phi) - 1$ with equality if and only if $p \ndivides \deg_a(\phi)$.

\begin{thm}[Riemann-Hurwitz Formula, {\cite[Corollary IV.2.4]{Hart}}]
 Let $\phi(z) \in K(z)$ be a separable rational function of degree $d \ge 1$. Then \[\sum_{a \in \P^1(K)} \rho_a = 2d - 2.\]
\end{thm}

It follows that a rational map $\phi$ can have at most $2d-2$ critical points. Crucially, Riemann-Hurwitz implies that if $\phi$ is not totally inseparable then it has at most $2$ totally ramified points. 
 The next theorem is a generalisation of \cite[Theorem 1.19]{Bene}. The proof runs exactly the same, noting that when we change coordinates to make three fixed points equal to $0, 1, \infty$, this is not effected by $\phi_{1*}$ because $\phi^*$ as a field automorphism has to fix these numbers.

\begin{thm}\label{thm:dyn:classicalexceptional}
Let $\phi_* = \phi_{1*} \circ \phi_{2*}$ be a skew product of relative degree $d \ge 2$.
\begin{enumerate}[label=(\alph*), ref=(\alph*)]
 \item If $\#E_{\phi_*} > 2$, then $\chara K = p > 0$, and $\phi_2$ is totally inseparable, i.e. $\phi_*$ is conjugate $\phi_{1*} \circ \psi_{2*}$ where $\psi_2(y) = y^{p^m}$ for some $m \ge 1$. In particular, $\#E_{\phi_*} = \infty$.
 \item If $\#E_{\phi_*} = 2$, then $\phi_*$ is conjugate to $\phi_{1*} \circ \psi_{2*}$ where either $\psi_2(y) =y^d$ or $y^{-d}$.
\item If $\#E_{\phi_*} = 1$, then $\phi_*$ is conjugate to $\phi_{1*} \circ \psi_{2*}$ where $\psi_2(y)$ is a polynomial.
\end{enumerate}
\end{thm}

\begin{rmk}
 Here lies the utility of separable decomposition. A priori, $\phi_*$ could have infinitely many exceptional points, but if we assume that $\phi_* = \psi_{1*} \circ \psi_{2*}$ has separable decomposition and $\psi_2$ \emph{still} has degree at least $2$, then case (a) above is impossible.
\end{rmk}

\begin{thm}\label{thm:dyn:contraction}
 Let $\phi_* = \phi_{1*} \circ \phi_{2*}$ be a skew product of relative degree $d \ge 2$ and scale factor $\q < 1/d$. Then $\phi_* : \bH \to \bH$ is a contraction mapping with respect to the hyperbolic metric, with unique attracting fixed point $\xi$. Moreover, $\phi_*^n(a) \to \xi$ for every non-exceptional Type I point $a \nin E_{\phi_*}$. Every periodic Type I point is superrepelling.
\end{thm}

\begin{proof}
 By \autoref{cor:berk:lipschitz}, for every $\zeta_1, \zeta_2 \in \bH$, \[d_\bH(\phi_*(\zeta_1) , \phi_*(\zeta_2)) \le (d\q) d_\bH(\zeta_1 , \zeta_2).\]
Hence $\phi_*|_\bH$ is a contraction mapping on a complete metric space, so by the contraction mapping theorem there is a unique attractor $\xi$ with $\phi_*^n(\zeta) \to \zeta_0$ as $n \to \infty$. For the second part, it is enough to show that for any open Berkovich disk $D \in \P^1_\an$ containing $\xi$, and any $a \nin E_{\phi_*}$, there is an $n \in \N$ such that $a \in \phi_*^{-n}(D)$. Let $a \in \P^1$, and WLOG let $D = D_0 = D_\an(a, r) \ni \xi$, and let $\zeta_0 = \zeta(a, r)$ be the boundary of $D$. Now let $D_1$ be the component of $\phi_*^{-1}(D)$ containing $\xi$; this is a disk with boundary say $\zeta_1 \in \phi_*^{-1}(\zeta_0)$. Continuing this way we obtain disks $D_n$ with boundary $\zeta_n$ such that $\phi_*(D_{n+1}) = D_n$ and $\phi_*(\zeta_{n+1}) = \zeta_n$. Let $t_n = d_\bH(\zeta_n, \xi)$; then by \autoref{cor:berk:lipschitz} \[t_n = d_\bH(\zeta_n , \xi) \le (d\q) d_\bH(\zeta_{n+1} , \xi) = (d\q)t_{n+1}.\] Hence $t_n \ge t_0(d\q)^{-n} \to \infty$ as $n \to \infty$. If $a \nin D_n$ for any $n$, then the disks are nested, and by change of coordinates, we may assume that $a = \infty$ so that $D_n = D_\an(a, r_n)$ with $r_n = e^{t_n} \to \infty$ as $n \to \infty$. Clearly, $a = \infty$ is the only point omitted by $\bigcup_n D_n \subseteq \bigcup_n \phi_*^{-n}(D)$, so we conclude that $a$ is an exceptional point. Every periodic Type I point $a$ is superrepelling since $\deg_a(\phi)\q \le d\q < 1$.
\end{proof}

\begin{defn}
 Let $\phi_* = \phi_{1*} \circ \phi_{2*}$ be a skew product of relative degree $d \ge 2$ and scale factor $\q$. If $\q < 1/d$ then we say $\phi_*$ is a \emph{contracting} skew product and we call its unique fixed point the \emph{attractor}. 
\end{defn}
\unfinished{Extra Theorem in comments}
%

\begin{thm}\label{thm:juliasetfacts}
 Let $\phi_*$ be a skew product of relative degree $d \ge 2$ and scale factor $\q > 1/d$, and assume that $\abs{K^\times}$ is a $\Q(\q)$-module. Let $E_{\phi_*}$ be the exceptional set of $\phi_*$ in $\P^1(K)$. Then
\begin{enumerate}[label=(\alph*), ref=(\alph*)]
 \item Let $U \subseteq \P^1_\an$ be an open set intersecting $\J_{\phi, \an}$. Then
 \[\P^1_\an \sm \bigcup_{n \ge 0} \phi^n(U)\]
 is a finite set contained in $E_{\phi_*}$.
 \item Assume that $\k$ is uncountable. Then $\J_{\phi, \an}$ has empty interior and the Fatou set $\F_{\phi, \an}$ is non-empty. Moreover, given $\zeta \in \J_{\phi, \an}$, all but countably many directions (disks) at $\zeta$ are Fatou components.\label{thm:juliasetfacts:b}
 \item For any point $\zeta \in \P^1_\an \sm E_{\phi_*}$, the closure of the backward orbit of $\zeta$ contains $\J_{\phi, \an}$, with equality iff $\zeta \in \J_{\phi, \an}$.
 \item Let $S \subseteq \P^1_\an$ be a closed subset of $\P^1_\an$ that is not contained in $E_{\phi_*}$ and for which $\phi_*^{-1}(S) \subseteq S$. then $\J_{\phi, \an} \subseteq S$.
 \item Assuming $\J_{\phi, \an}$ is nonempty, either
\begin{enumerate}
\item $\J_{\phi, \an}$ consists of a single Type II point, and $\phi$ has potential good reduction; or
\item $\J_{\phi, \an}$ is a perfect set, and it is uncountable.
\end{enumerate}
\item If $\phi_2$ is a polynomial and we define
\[\mathcal{K}_{\phi, \an} = \set[\zeta \in \P^1_\an]{\lim_{n \to \infty} \phi^n(\zeta) \ne \infty}\]
to be the Berkovich filled Julia set of $\phi$, and then $\J_{\phi, \an} = \partial\mathcal{K}_{\phi, \an}$.
\end{enumerate}
\end{thm}

\begin{proof}
 The proofs from Benedetto \cite[Theorem 8.15]{Bene} all hold except for (b); see \autoref{rmk:dyn:holofixedpt} below. With our hypotheses we can prove it as follows.
 
 If $U \subset \J_{\phi, \an}$ is open then by (a), $\bigcup_{n \ge 0} \phi^n(U) \subseteq \J_{\phi, \an}$ contains all but a finite set of points. Since the Julia set is closed, $\J_{\phi, \an} = \P^1_\an$. So to find a contradiction to this and prove the other statement we will show that the Fatou set is non-empty.
 
 Pick a separable decomposition $\phi_* = \phi_{1*} \circ \phi_{2*}$ so that $\# E_{\phi_*} \le 2$ and let $\gamma \in \P^1\sm E_{\phi_*}$. Then $\Orb_\phi^-(\gamma)$ is infinite. Given a Type II $\zeta \in \P^1_\an$ define the following subset of directions \[S_\zeta = \set[\bvec v \in \Dir\zeta]{\bvec v \cap \Orb_\phi^-(\gamma) = \emp}\]
 Any $\bvec v \in S_\zeta$ is a good direction because $\gamma \nin \phi_*(\bvec v)$ and that $\phi_\#(\bvec v) = \phi_*(\bvec v) \in S_{\phi_*(\zeta)}$; therefore for every $n \ge 0$, $\phi_*^n(\bvec v)\cap \Orb_\phi^-(\gamma) = \emp$. Thus \[U = \bigcup_{n=0}^\infty \phi_*^n(\bvec v)\] is also disjoint from $\Orb_\phi^-(\gamma)$ and since $\Orb_\phi^-(\gamma)$ is infinite then we have proved that $\bvec v \subset \F_{\phi, \an}$.
\end{proof}
 
\unfinished[inline]{Will write theorem finding a (\weakly) repelling point, hence the Julia set is non-empty. Similar in flavour to the result by Rivera-Letelier extended by Baker and Rumely, \cite[\S12]{Bene}\\ Further ideas: if we found an attracting basin then its boundary must have a repelling point (use symbolic dynamics in the Cantor case).}

\begin{rmk}\label{rmk:dyn:holofixedpt}
 In complex dynamics and the dynamics of non-Archimedean rational maps, one relies on the Holomorphic fixed point formula to find non-repelling fixed points. In the absence of fixed classical points with multiplier $1$, the formula says there must be attracting (in fact also repelling) fixed classical points. Indifferent and attracting classical points are naturally Fatou. We shall use another strategy to find such points.
\end{rmk}


\section{Fatou Components of Skew Products}\label{sec:class}

The aim of this section is to classify the connected components of the Fatou set $\F_{\phi, \an}$. Originally this was completed by J. Rivera-Letelier \cite{RL1, RL2, RL4} for rational maps, and as far as possible we follow the ideas in \cite[\S9]{Bene} although many new proofs are required. For this section we will restrict to simple skew products over a residue characteristic $0$ field and, where necessary, those defined over a discretely valued subfield. We shall also be able to assume the hypothesis used occasionally in the previous section that $\abs{K^\times}$ is a $\Q(\q)$-module; this is trivial for simple skew products, because $\q = 1$.

\begin{rmk}
 Note that any simple $\k$-rational skew product satisfies these conditions. Recall that a skew product $\phi_* : \P^1_\an(\K) \to \P^1_\an(\K)$ is $\k$-rational iff the base field is the field $\K$ of Puiseux series (or $\hk$) with $\k$ coefficients. We can write $\phi^*(x) = \phi_1(x) = \lambda x^n + \bO(x^{n+1})\in \k[[x]]$, and $\phi^*(y) = \phi_2(x, y) \in \k((x))(y)$. Then $\q = 1$ if and only if $n=1$, whence $\phi_1$ is also defined over $\k((x))$ whose value group is $\abs{\k((x))^\times} = (\set{\abs x^r : r \in \Z}, \times) \cong (\Z, +)$.
\end{rmk}

\subsection{Berkovich Fatou Components}

The following generalises \cite[Theorem 9.1]{Bene} to skew products. 

\begin{thm}\label{thm:dyn:comps}
Let $\phi$ be a skew product of relative degree $d \ge 2$, and let $U$ be a connected component of $\F_{\phi, \an}$. Then the following hold.
\begin{enumerate}[label=(\alph*), ref=\theenumi]
 \item $\phi_*(U)$ is also a connected component of $\F_{\phi, \an}$.
 \item There is an integer $1 \le m \le d$ such that every point in $\phi_*(U)$ has
exactly $m$ preimages in $U$, counting multiplicity.
\item The inverse image $\phi_*^{-1}(U)$ is a finite union of components of $\F_{\phi, \an}$,
each of which is mapped onto $U$ by $\phi$.
\end{enumerate}
\end{thm}

\begin{lem}[{\cite[Lemma 9.2]{Bene}}]
 Let $\phi$ be a skew product of relative degree $d \ge 2$, let $U$ be a connected component of $\F_{\phi, \an}$, and let $x \in U$. Then $U$ is the union of all open connected Berkovich affinoids containing $x$ and contained in $\F_{\phi, \an}$.
\end{lem}

The proofs are topological, following \cite[\S9]{Bene}. To prove \autoref{thm:dyn:comps}, one essentially combines the lemma with \autoref{thm:skew:affinoidmapping}.
By \autoref{thm:dyn:comps} (a) \[\phi_*^m(U) \cap \phi_*^n(U) \ne \emp \iff \phi_*^m(U) = \phi_*^n(U),\] hence the following definitions make sense.

\begin{defn}
 Let $U$ be a connected component of the Fatou set $\F_{\phi, \an}$. Then we call $U$ a \emph{Fatou component}. 
 \begin{enumerate}[label=(\alph*), ref=\theenumi]
 \item If additionally $\phi_*^m(U) = U$ for some minimal $m > 0$ then we say $U$ is a \emph{periodic} Fatou component with \emph{period} $m$.
 \item If there is an $m_0$ such that $\phi_*^{m_0}(U)$ is periodic then we say that $U$ is \emph{preperiodic}; or in other words, $\phi_*^{m + m_0}(U) = \phi_*^{m_0}(U)$ for some minimal $m, m_0$. 
 \item Otherwise (if $U$ is not preperiodic) we say $U$ is \emph{wandering}; equivalently $\phi_*^m(U) \cap \phi_*^n(U) = \emp$ for every distinct $m, n \in \N$.
 \end{enumerate}
\end{defn}

As already mentioned, the goal of this section will be to prove that all preperiodic Fatou components are `attracting' or `indifferent', \autoref{thm:dyn:class}. Any other Fatou component is wandering, and these will be studied in \autoref{sec:wandering}.

\subsection{Attracting Components}

Recall from \autoref{defn:dyn:classicalfixed} and \autoref{prop:dyn:classicalfixed} that $a \in \P^1$ is an attracting periodic point of (minimal) period $m \ge 1$ iff there is an open (Berkovich) disk $D \ni a$ such that
 \[\lim_{n \to \infty} \phi_*^{mn}(b) = a\quad \forall\ b \in D \cap \P^1(K).\]

\begin{defn}
 Let $\phi_*$ be a skew product of relative degree $d \ge 2$. Let $a \in \P^1(K)$ be an attracting periodic point of period $m$. Let $D$ be a disk as above with $\phi_*^{mn}(D)$ converging to $a$.
\begin{enumerate}[label=(\alph*), ref=\theenumi]
 \item The \emph{attracting basin} of $\phi$ is \[\mathfrak{U}_\phi(a) = \bigcup_{n \ge 0} \phi_*^{-n}(D).\]
 \item The \emph{immediate attracting basin} of $\phi_*$ is the connected component $U$ of $\mathfrak{U}_\phi(a)$ containing $a$. We call $U$ an \emph{attracting component} for $a$.
\end{enumerate}
\end{defn}

Recall also the definition of a domain of Cantor type from \cite[Definition 9.4]{Bene}.

\begin{defn}
 Let $U \subseteq \P^1_\an$ be an open set. We say that $U$ is a \emph{domain of Cantor type} if there is an increasing sequence of open connected Berkovich affinoids $U_1 \subsetneq U_2 \subsetneq U_3 \subsetneq \cdots$ such that
\begin{enumerate}
 \item for every $n \ge 1$, each connected component of $\P^1_\an \sm U_n$ contains at
least two points of $\partial U_{n+1}$; and 
\item $U = \bigcup_{n\ge1}U_n$.
\end{enumerate}
\end{defn}

The following is a generalisation to skew products of the theorem due to Rivera-Letelier \cite{RL2}.

\begin{thm}\label{thm:attractingperpt}\improvement{last check for dropping the rationality requirement in the future. I think it's not needed but CHECK.}
 Let $\phi_*$ be a skew product of relative degree $d \ge 2$, let $a \in \P^1(K)$ be an attracting periodic point of minimal period $m \ge 1$, and let $U \subseteq \P^1_\an$ be the immediate attracting basin of $\phi_*$. Then the following hold.
\begin{enumerate}
 \item $U$ is a periodic Fatou component of $\F_{\phi, \an}$, of period $m$.
 \item $U$ is either an open Berkovich disk or a domain of Cantor type.
 \item If $U$ is a disk, then its unique boundary point is a Type II repelling periodic point, of minimal period dividing $m$.
 \item If $U$ is of Cantor type, then its boundary is homeomorphic to the
Cantor set and is contained in the Berkovich Julia set.
\item The mapping $\phi_*^m : U \to U$ is $l$-to-$1$, for some integer $l \ge 2$.
\end{enumerate}
\end{thm}

\begin{proof}
Note that by \autoref{defn:dyn:classicalfixed} it is immediate that $\q \le 1$. The proof is similar to Benedetto's proof of \cite[Theorem 9.5]{Bene}.

By \autoref{prop:dyn:classicalfixed} there is an open disk $V_0$ containing $a$ such that $\phi_*(V_0) \subset V_0$. For $n \in \N$ we define $V_{n+1}$ to be the connected component of $\phi_*^{-1}(V_n)$ containing $V_n$. It follows that $V_0 \subset V_1 \subset V_2 \subset \cdots$, and define $\tilde U = \bigcup_n V_n$. Evidently $\phi_*(\tilde U) = \tilde U$ and $\tilde U \subseteq U$. 

To establish that $\tilde U$ omits two points, we must instead use \autoref{thm:dyn:classicalexceptional}. If the complement of this were empty, then by compactness, for some $n$, $\phi_*^{-n}(D) = \P^1_\an$, but this is impossible. If the complement were a single point then this is an exceptional point (WLOG $\infty$) and we must have that $\phi_2(y)$ is a polynomial. In this case, by \autoref{thm:juliasetfacts}(f) the Julia set separates our attracting point $a$ from $\infty$, but then $U$ would have to contain the Julia set; contradiction.

Suppose every $V_n$ is a disk, then so is $\tilde U$. By changing coordinates, we may assume that $a = 0$ and $\infty \nin \tilde U = D_\an(a, R)$. We may also write that $V_n = D_\an(a, r_n)$; so $r_n \to R$.

On this disk, we may write the power series for $\phi_2(y)$ as \[\phi_2(y) = \sum_{n=d_0}^\infty c_n y^n,\]
where $d_0 = \deg_a(\phi) \ge 1$. It follows that $\max_n \abs{c_n}^\q R^{n\q} = R$ since $D_\an(a, R)$ is fixed; suppose the Weierstrass degree $\overline\wdeg_{0, R}(\phi_2) = l$. Because $a=0$ is attracting, assuming $r_0$ is sufficiently small, we have \[\abs{c_l}^\q {r_0}^{l\q} \le \max_n \abs{c_n}^\q {r_0}^{n\q} = \abs{c_{d_0}}^\q {r_0}^{d_0\q} < {r_0}.\]
Thus $(r_0/R)^{l\q} < r_0/R \implies (r_0/R)^{l\q-1} < 1$ and hence $l\q > 1$. This proves that the multiplier of $\vec v(0)$ is $\deg_{\zeta(0, R), \vec v(0)}(\phi)\q = l\q > 1$, and hence $\vec v(0)$ is a repelling direction. Therefore $\zeta(0, R)$ is numerically repelling, and by \autoref{thm:perptdynone}, this is a Julia point.

In the other case we assume that ($\tilde U$ and) some $V_n$ is not a disk. The rest of the proof is purely topological and identical to the corresponding remainder of Benedetto's proof of \cite[Theorem 9.5]{Bene}.
\improvement{Type out proof. Figure out the three fixed points problem. Answer: prove non-emptiness of Julia set. If empty then wtf.}
\end{proof}

\subsection{Indifferent Components}

\begin{defn}
Let $\phi_*$ be a simple skew product of relative degree $d \ge 2$. We define the \emph{indifference domain}, $\mathfrak I$, of $\phi_*$ to be the set of points $\zeta \in \P^1_\an$ for which there is an open connected affinoid $W$ containing $\zeta$ and an integer $m \ge 1$ such that $\phi_*^m$ maps $W$ bijectively onto itself.
A connected component of the indifference domain is called an \emph{indifferent component} for $\phi$.
\end{defn}

Studying indifference domains for skew products requires greater care than for rational maps, as we will see in the next few results.

\begin{lem}\label{lem:intervalmap}
 Let $\phi_*$ be a simple skew product, $r > 0$ and $\zeta = \zeta(a, r) \in \P^1_\an$. 
 Suppose that $\deg_{\zeta, \vec v(a)}(\phi) = 1$, $\phi_*(\zeta) = \zeta(b, s)$ and $\phi_\#(\vec v(a)) = \vec v(b)$; then there exists an $r_0 < r$ such that $\phi_*(\zeta(a, t)) = \zeta(b, st/r)$ for all $t \in [r_0, r]$.
 
 In the case that $\phi_*$ maps $D(a, r)$ bijectively to $D(b, s)$, $r_0 = \frac{|\phi_*(a) - b|r}{s}$ is the minimum possible.
\end{lem}

In other words given any two intervals $[a, \zeta(a, r)]$ and $[b, \zeta(b, s)]$ where $\phi_\#(\vec v(a)) = \vec v(b)$ with degree $1$, there is at least a short final segment of the former mapping isometrically onto a final segment of the latter. The lemma provides an effective estimate of this subinterval using similar ideas to \autoref{thm:skew:bigintervalstretch}.

\begin{proof}
 By changing coordinates we may assume WLOG $a = b = 0$. Also, by considering that $\phi_{1*}$ fixes $\zeta(0, t)$ for every $t$, we have that $\phi_* = \phi_{2*}$ on $[0, \infty]$.
 \[\phi_2(y) = \sum_na_ny^n\]
 Because $\deg_{\zeta, \vec v(a)}(\phi) = 1$, we know that there is also an inner radius $r_0$ large enough so that the inner Weierstrass degree there is also $1$; if $\phi_2$ is convergent in $D(0, r)$, then the only restriction will be that $|\phi_*(a) - b| = |\phi_0| \le |\phi_1|r_0$ so that the image of $\zeta(0, r_0)$ remains centred on $b = 0$. So for any $t \in [r_0, r)$ \[|a_1|t >  |a_n|t^n \quad \forall n \ge 2\] and \[|a_n|t^n \le  |a_1|t \quad \forall n \le 0.\] This shows that $\phi_*(\zeta(0, t)) = \zeta(0, |\phi_1|t)$ whereas we know that $|\phi_1|r = s$ meaning $|\phi_1| = \frac sr$. Thus $\phi_*(\zeta(0, t)) = \zeta(0, st/r)$. In the case that $\phi_*$ has no poles in $D(a, r)$, the minimal working inner radius is \[r_0 = \frac{|\phi_*(a)-b|}{\abs {a_1}} = \frac{|\phi_*(a) - b|r}{s}.\]
 \end{proof}
 
 The next lemma generalises \cite[Lemma 9.8]{Bene}.

\begin{lem}\label{lem:periodicstring}
 Let $L$ be a discretely valued subfield of $K$, $a \in L$, $r > 0$, and let $\phi_*$ be a simple skew product defined over $L$. Suppose that $\phi_*$ maps $D(a, r)$ to itself bijectively. 
 Given $0 < s < r$, any $b \in D(a, r)$ such that $\CD(b, s) \cap L \ne \emp$, either
\begin{enumerate}[label=(\alph*), ref=\theenumi]
 \item $\zeta(b, s)$ is periodic under $\phi_*$, or
 \item there is a $t \in \abs{L^\times} \cap (s, r)$, such that $\zeta(b, t)$ is periodic under $\phi_*$ and there are infinitely many directions at $\zeta(b, t)$ that contain an iterate $\phi_*^m(b)$.
\end{enumerate}
\end{lem}

\begin{proof}
 Write $a = b = b_0 + b_1$ where $|b_1| \le s$ and $b_0 \in L$ (emphasis on $b_0$ having finitely many terms).
 
If $\zeta(b, s)$ is not periodic then $|\phi_*^n(b) - b| > s$ for every $n \in \N_+$.
 
 Since $|b - b_0| \le s$ and $\phi_*$ is an isometry in $D(b, r)$, we have $|\phi_*^n(b) - \phi_*^n(b_0)| \le s$ for every $n \in \N$. This shows that $|\phi_*^n(b) - b| = |\phi_*^n(b_0) - b_0| > s$ for every $n$. Effectively we can now replace $b$ by $b_0$, assuming $b \in L$. Consider the collection of numbers
 \[A = \set[|\phi_*^n(b) - b|]{n \in \N_+}.\] Since $L$ is discretely valued and $\phi_*$ is defined over $L$, it is immediate that $A\subset \abs{L^\times}$. It is also immediate that $A \subset [s, r)$. Therefore $A$ is finite.

 Now we take the first $j$ such that $|\phi_*^j(b) - b| = \min A = t$, hence $\phi_*^j(\zeta(b, t)) = \zeta(b, t)$. By the isometry property we have $|\phi_*^{j(n+1)}(b) - \phi_*^{jn}(b)| = t$. By the triangle inequality $|\phi_*^{jn}(b) - b| \le t$ for every $n > 0$, but then by minimality of $t$ we must have $|\phi_*^{jn}(b) - b| = t$. Finally applying the isometry again we have $|\phi_*^{jm}(b) - \phi_*^{jn}(b)| = t$ for every $m \ne n$. So we have shown that $(\phi_*^{jn}(b))_{n=1}^\infty$ is an infinite sequence with distinct directions at $\zeta(b, t)$.
\end{proof}

%

\begin{lem}[Rivera-Letelier's approximation lemma]
 Let $b, c \in K$ and $r > s > 0$. Let $\phi$ be a simple skew product that maps the open Berkovich annulus
$\set{s<|\zeta-b|<r}$ bijectively onto $\set{s<|\zeta-c|<r}$ with $\phi(\zeta(b,r)) = \zeta(c,r)$. Then there is a simple skew product $\psi_*$ such that
\begin{enumerate}[label=(\alph*), ref=\theenumi]
\item $\psi_*$ maps $D_\an(b, r)$ onto $D_\an(c, r)$ with Weierstrass degree 1, and
\item $\psi_*(\zeta) = \phi_*(\zeta)$ for any $\zeta \in D_\an(b, r)$ with $\diam(\zeta) \ge s$.
\end{enumerate}
\end{lem}

Rivera-Letelier proved this for rational maps, meaning there is a $\psi_{2*}$ which approximates $\phi_{2*}$ in this way. Composing both with $\phi_{1*}$ gives the approximation to $\phi_*$. The next theorem is also largely due to Rivera-Letelier, adapted to skew products using the modifications above.

\begin{thm}\label{thm:rivindiff}
Let $L$ be a discretely valued subfield of $K$, and let $\phi_*$ be a simple skew product defined over $L$ of relative degree $d \ge 2$. Let $U \subseteq \P^1_\an$ be an indifferent component for $\phi_*$. Then the following hold.
\begin{enumerate}[label=(\alph*), ref=\theenumi]
\item $U$ is a periodic connected component of $\F_{\phi, \an}$, of some minimal period $m \ge 1$.
\item $U$ is a rational open connected Berkovich affinoid.
\item Each of the finitely many points of the boundary $\partial U$ is a Type II point lying in the Berkovich Julia set.
\item $\phi_*^m$ permutes the boundary points of $U$; in particular, each is periodic.
\item The mapping $\phi_*^m : U \to U$ is bijective.
\end{enumerate}
\end{thm}

\info[inline]{Benedetto was confusing in the way he worded the proof of Theorem 9.7. The map $(y-x)(1-y)/(y+x)$ gives an example where the Gauss point is fixed, degree $1$, with $\vec v(0)$ and $\vec v(1)$ in a $2$-cycle; however the former is a bad direction and the latter is good. He seems to suggest that the finite orbit of a bad direction must be bad and therefore ``each good residue class maps bijectively onto a good residue class''. The focus should be on the application of \autoref{thm:perptdynone} to Type II points. Will fix this later by re-typing the whole proof.}

 The proof goes through as in Benedetto, however 
 we need to interpret \autoref{lem:periodicstring} for an arbitrary projective Berkovich disk e.g. $D = \P^1_\an \sm \overline D_\an(1, |x|)$.

\improvement[inline, disable]{\autoref{lem:periodicstring} should be rewritten in terms of the indifference domain. One could even try a compactness argument on the closed interval. This can then provide a way to target boundary points in the theorem.}

\subsection{Classification}

We are ready to state the classification of Fatou components. Rivera-Letelier originally proved this for the field of p-adic numbers in his PhD thesis \cite{RL2}, and gave extensions to the general non-Archimedean setting in \cite{RL1,RL4}; many of his methods live on in this chapter. 

\begin{prop}\label{prop:perptsindiff}
Let $L$ be a discretely valued subfield of $K$, and $\phi_*$ be a simple skew product defined over $L$ of relative degree $d \ge 2$,\unsure{CHECK again. Only needs to be rational if \autoref{thm:attractingperpt} does.} and with indifference domain $\mathfrak I$. Then $\mathfrak I$ contains
\begin{itemize}
\item every Type I indifferent periodic point,
\item every Type II Fatou periodic point,
\item every Type III periodic point, and
\item every Type IV periodic point.
\end{itemize}
Moreover, $\mathfrak I$ does not contain any other periodic points.
\end{prop}

\begin{proof}
 Suppose $\zeta$ is a Type III or IV periodic point; we apply \autoref{thm:dyn:typeIIIindiff} and \autoref{thm:dyn:typeIVindiff} respectively to show they are indifferent. 
 Hence $\zeta$ has a disk or annular neighbourhood which is fixed, as seen in the proof of \autoref{thm:perptdynone}, and a consequence of \autoref{thm:skew:bigintervalstretch}.
 \improvement{If $\q < 1$ we could potentially have an attracting Type IV point? Not for a discretely valued subfield, right? At least any such attracting basin would have to omit any elements of $L$. For the second sentence, things are already funky because we lack injectivity\dots}
 
 If $\zeta$ is Type II then \autoref{thm:perptdynone} says that if $\zeta$ is Fatou then there is an affinoid $U \ni \zeta$, containing all but finitely many (at least one) directions at $\zeta$ which are all exceptional. Again, given that $\phi_*$ is simple, $\zeta$ must be \stronglyindifferent{} and these finitely many directions must be indifferent, hence in fact $U$ is periodic, thus $\zeta \in \mathfrak U_\phi$. By \autoref{thm:rivindiff} no Type II Julia point would be in the indifference domain.\improvement{For $\q < 1$, could definitely have an attracting Type II Fatou point here, then you get no such $U$.}
 
 If $a$ is Type I and indifferent, then by \autoref{prop:dyn:classicalfixed}, it has a neighbourhood $D(a, r)$ which is fixed by an iterate of $\phi_*$, hence $a \in \mathfrak U_\phi$.
 
 If $\zeta$ is periodic but not in this list, then it must be Type I attracting or repelling. By \autoref{prop:dyn:classicalrepelling} a repelling point is Julia so cannot be in the indifference domain. If $\zeta = a$ is attracting, let $U$ be its Fatou component. By \autoref{thm:attractingperpt}, $U$ is the immediate attracting basin of $a$, and $\phi_*^n(U) = U$ is $l$-to-$1$ for some $l \ge 2$. If $U$ was also in the indifference domain then it would be a bijective map by \autoref{thm:rivindiff}.
\end{proof}

\begin{thm}[Classification of Fatou Components]\label{thm:dyn:class}
 Let $L$ be a discretely valued subfield of $K$. Let $\phi_* : \P^1_\an(K) \to \P^1_\an(K)$ be a simple \unsure{This likely only needs to be rational if \autoref{thm:rivindiff} needs to be (see proof).} skew product defined over $L$ of relative degree $d \ge 2$ with Berkovich Fatou set $\F_{\phi, \an}$, and let $U \subset \F_{\phi, \an}$ be a periodic Fatou component. Then $U$ is either an indifferent component or an attracting component, but not both.
\end{thm}

The following two lemmas replicate \cite[Lemma 9.15, Lemma 9.16]{Bene}, and the proofs are essentially identical. In the proof of \autoref{lem:squeezedfixedpt}, instead of showing that $\deg_{\zeta(a, s), \vec v(a)}(\phi) = 1$, instead we would show that the multiplier is $\deg_{\zeta(a, s), \vec v(a)}(\phi)\q \le 1$; with the same desired consequence. For \autoref{lem:typeIIIfixed} we must replace the application of \cite[Theorem 8.7]{Bene} with the combination of \autoref{thm:dyn:typeIIIindiff} and \autoref{thm:perptdynone}.

\begin{lem}\label{lem:squeezedfixedpt}
 Let $\phi_* : \P^1_\an(K) \to \P^1_\an(K)$ be a skew product of relative degree $d \ge 2$ and scale factor $\q$, let $X \subseteq \P^1_\an$ be a connected set, let $a \in K$, and let $(s_n)_{n\ge1}$ be a strictly increasing sequence of positive real numbers with $s_n \nearrow s \in \R$. Suppose that for all $n\ge1$, we have
 
\begin{itemize}
 \item $\zeta(a, s_n) \in X$
 \item $\phi_*(\zeta(a, s_n)) \in \{s_n \le |\zeta - a| < s\} = D_\an(a, s) \sm D_\an(a, s_n)$
\end{itemize}
Then either $X$ contains a Type II fixed point, or $\zeta(a, s)$ is fixed with an attracting direction $\vec v(a)$. In either case $\q \le 1$.
\end{lem}

\begin{lem}\label{lem:typeIIIfixed}
 Let $\phi_*$ be a skew product with scale factor $\q$ such that $\abs{K^\times}$ is a $\Q(\q)$-module, and let $X \subseteq \P^1_\an$ be a connected set with at least two points. If $X$ contains a Type III fixed point, then it also contains a Type II fixed point.
\end{lem}

\begin{lem}[Rivera-Letelier's stable fixed-point lemma]\label{lem:rivfixedpt}
Let $\phi$ be a simple skew product, and let $X \subseteq \P^1_\an$ be a connected set with at least two points. Suppose that $\phi_*(X) \subseteq X$. Then either $X$ contains a Type II fixed point, or the closure of $X$ in $\P^1_\an$ contains an attracting fixed point.
\end{lem}

\begin{proof}[Proof of \autoref{lem:rivfixedpt}]
 Goes through with modifications to the proof of \cite[Lemma 9.17]{Bene}. We note that in end of Benedetto's proof we must have $\phi_*(\zeta(y_1, s_n)) \in D_\an(y_1, r_1)$ because otherwise $\zeta(y_1, r_1) \nin X$ would be between $\zeta(y_1, s_n)$ and $\phi_*(\zeta(y_1, s_n))$. The application of \cite[4.18]{Bene} must be replaced by \autoref{prop:dyn:attrfixed}.
\end{proof}

\begin{proof}[Proof of \autoref{thm:dyn:class}]
 Replacing $\phi_*$ by $\phi_*^n$ if necessary, we may assume that $\phi_*(U) = U$. By \autoref{lem:rivfixedpt}, either $U$ contains a Type II fixed point, or else $\overline U$ contains an attracting fixed point.
 
If there is a Type II fixed point $\zeta \in U \subseteq \F_{\phi, \an}$, then by \autoref{prop:perptsindiff}, $\zeta$ belongs to the indifference domain $\mathfrak I$ of $\phi$. By \autoref{thm:rivindiff}, $U$ must be a component of $\mathfrak I$ and hence an indifferent component.

If $\overline U$ contains an attracting fixed point $a$, then by \autoref{thm:attractingperpt} we have $a \in \F_{\phi, \an}$ and hence its immediate attracting basin must intersect $U$. Therefore $a \in U$ and $U$ is the immediate attracting basin of $a$.

Finally, by \autoref{prop:perptsindiff} again, no indifferent component can contain an attracting periodic point. Thus, $U$ cannot be both an attracting and an indifferent component.
\end{proof}


\subsection{Wandering Domains}\label{sec:wandering}


Recall that a Fatou component $U \subset \F_{\phi, \an}$ of $\phi_*$ is wandering iff the iterates $U, \phi_*(U), \phi_*^2(U), \dots$ are all distinct. In short, we will call $U$ a \emph{wandering domain}. In this section we will explain how wandering domains of simple skew products are eventually disks, with periodic boundary points. So in some sense, they do not wander very much.

\begin{lem}[{Benedetto \cite[Lemma 10.14]{Bene}}]\label{lem:boundcxdisks}
 Let $\phi_*$ be a skew product of relative degree $d \ge 1$, and let $D_1, \dots, D_m \subseteq \P^1_\an$ be disks whose closures are pairwise disjoint. Suppose that for each $i = 1, \dots, m$, at least one connected component of $\phi_*^{-1}(D_i)$ is not a disk. Then $m \le d-1$.
\end{lem}

This lemma, via \autoref{thm:skew:comp}, is a corollary of Benedetto's theorem, originally from \cite{BeneThesis}. The following theorem generalises \cite[Theorem 11.2]{Bene} to skew products. The proof by Benedetto carries through assuming that $\q \ge 1$, using \autoref{lem:boundcxdisks} and \autoref{prop:wand:hyperbolicity}. However the techniques clearly fail when $\q$ is small. Indeed, the proof relies on the idea that the hyperbolic width of the iterates of an affinoid $\phi_*^n(V)$ in a wandering domain are never decreasing, given the conclusion of \autoref{prop:wand:hyperbolicity}, and these iterates $\phi_*^n(V)$ cover a finite set of intervals of finite length. If $\q < 1$ their lengths may shrink like $\q^n$ and it is conceivable they remain affinoids but not disks forever.

\begin{thm}\label{thm:wand:wanddomsaredisks}
Let $\phi_*$ be a skew product of relative degree $d \ge 2$ and scale factor $\q \ge 1$, with Fatou set $\F_{\phi, \an}$. Let $U \subseteq \F_{\phi, \an}$ be a wandering domain of $\phi_*$. Then there is some $N \ge 0$ such that $\phi_*^n(U)$ is a disk for all $n \ge N$.
\end{thm}

\begin{qtn}
 Does \autoref{thm:wand:wanddomsaredisks} hold when $\q < 1$?
\end{qtn}

To prove \autoref{thm:wand:wanddomsaredisks} using the techniques of Benedetto we need to reprove a version of \cite[Proposition 11.3]{Bene}. We provide a somewhat different proof of this using the hyperbolicity theorem \autoref{thm:skew:hyperbolicity}.

\begin{prop}\label{prop:wand:hyperbolicity}
 Let $\phi_*$ be a non-constant skew product of scale factor $\q$, and let $U \subseteq \P^1_\an$ be a connected Berkovich affinoid with at least two boundary points. Let $\delta > 0$ be the minimum hyperbolic distance between two distinct boundary points of $U$. Then for any two distinct boundary points $\xi_1, \xi_2 \in \partial(\phi_*(U))$, we have $d_\bH(\xi_1, \xi_2) \ge \q\delta$.
\end{prop}

\begin{proof}
 Consider the two points $\xi_1, \xi_2 \in \partial(\phi_*(U))$. Firstly, $[\xi_1, \xi_2] \subset \phi_*(U)$ because $U$ and hence $\phi_*(U)$ is connected. Pick any $\zeta_1 \in \partial U$ such that $\phi_*(\zeta_1) = \xi_1$, and by \autoref{thm:skew:affinoidmapping} there is at least one such point. Further by the same theorem, every $\zeta_2 \in \phi_*^{-1}(\xi_2)$ cannot be in $U$ so in the direction of $\vec v(U)$ at $\zeta_1$, it must be in $\partial U$, or outside of $\overline U$ beyond another boundary point. Either way, for $d_\bH(\zeta_1, \zeta_2) \ge \delta$ for any $\zeta_2 \in \phi_*^{-1}(\xi_2) \cap \vec v(U)$.
 By \autoref{thm:skew:hyperbolicity} we have that 
 \[d_\bH(\xi_1, \xi_2) \ge \q \min_{\zeta_2 \in \phi_*^{-1}(\xi_2) \cap \vec v(U)} d_\bH(\zeta_1, \zeta_2) \ge \q\delta.\]
\end{proof}

\begin{defn}
 Let $\phi_*$ be a simple skew product of relative degree $d \ge 2$, let $\F_{\phi, \an}$ be the Berkovich Fatou set of $\phi_*$, and let $\zeta \in \P^1_\an$ be a Type II periodic point of $\phi_*$ of minimal period $m$. We say that a wandering component $U$ of $\F_{\phi, \an}$ is in the \emph{attracting basin} of $\zeta$ if there is some integer $N \ge 0$ such that for all $n \ge 0$,
$\phi_*^{N+nm}(U)$ is a residue class at $\zeta$.
\end{defn}

The following is a modification of \cite[Theorem 11.23]{Bene}.

\begin{thm}\label{thm:wanderingdombasin}
 Let $L$ be a discretely valued subfield of $K$. Let $\phi_*$ be a non-constant simple skew product defined over $L$. Then any wandering domain of $\phi_*$ lies in the attracting basin of a Type II Julia periodic point.
\end{thm}


We end this chapter with the remarkable notion that every Type II Julia point should be preperiodic, as it is for a rational map.

\begin{cor}\label{cor:dyn:typeiijuliaperperiodic}
 Let $L$ be a discretely valued subfield of $K$, whose residue field $\k$ is uncountable, and let $\phi_*$ be a non-constant simple skew product defined over $L$. Then any Type II Julia point is preperiodic.
\end{cor}

\begin{proof}
 By \autoref{thm:juliasetfacts}, all but countably many directions at a Type II point $\zeta$ are Fatou disks. If the any such disk is in the indifference domain then by \autoref{thm:rivindiff}, its boundary points, including $\zeta$ are preperiodic. Similarly, if one disk is an attracting component, then by \autoref{thm:attractingperpt} $\zeta$ is the backward orbit of the unique periodic boundary point of an attracting component. If any such disk is wandering then we are done by the previous theorem.\unsure{could this be replaced with just countably many?}
\end{proof}

%
%
%
%






 \bibliographystyle{alpha}
\bibliography{Thesis}

\begin{thebibliography}{BDM13}

\bibitem[BBP07]{BBP}
Robert Benedetto, Jean-Yves Briend, and Herv\'{e} Perdry.
\newblock Dynamique des polyn\^{o}mes quadratiques sur les corps locaux.
\newblock {\em J. Th\'{e}or. Nombres Bordeaux}, 19(2):325--336, 2007.

\bibitem[BD11]{BDe11}
Matthew Baker and Laura DeMarco.
\newblock Preperiodic points and unlikely intersections.
\newblock {\em Duke Math. J.}, 159(1):1--29, 2011.

\bibitem[BDM13]{BDe13}
Matthew Baker and Laura De~Marco.
\newblock Special curves and postcritically finite polynomials.
\newblock {\em Forum Math. Pi}, 1:e3, 35, 2013.

\bibitem[Ben98]{BeneThesis}
Robert~L. Benedetto.
\newblock {\em Fatou components inp-adic dynamics}.
\newblock ProQuest LLC, Ann Arbor, MI, 1998.
\newblock Thesis (Ph.D.)--Brown University.

\bibitem[Ben00]{Bene0}
Robert~L. Benedetto.
\newblock {$p$}-adic dynamics and {S}ullivan's no wandering domains theorem.
\newblock {\em Compositio Math.}, 122(3):281--298, 2000.

\bibitem[Ben01a]{Bene2}
Robert~L. Benedetto.
\newblock Hyperbolic maps in p-adic dynamics.
\newblock {\em Ergodic Theory and Dynamical Systems}, 21:1 -- 11, 02 2001.

\bibitem[Ben01b]{Bene3}
Robert~L. Benedetto.
\newblock Reduction, dynamics, and {J}ulia sets of rational functions.
\newblock {\em J. Number Theory}, 86(2):175--195, 2001.

\bibitem[Ben02a]{Bene4}
Robert~L. Benedetto.
\newblock Components and periodic points in non-{A}rchimedean dynamics.
\newblock {\em Proc. London Math. Soc. (3)}, 84(1):231--256, 2002.

\bibitem[Ben02b]{Bene5}
Robert~L. Benedetto.
\newblock Examples of wandering domains in {$p$}-adic polynomial dynamics.
\newblock {\em C. R. Math. Acad. Sci. Paris}, 335(7):615--620, 2002.

\bibitem[Ben05]{Bene1}
Robert~L. Benedetto.
\newblock Wandering domains and nontrivial reduction in non-archimedean
  dynamics.
\newblock {\em Illinois Journal of Mathematics}, 49(1):167--193, 2005.

\bibitem[Ben06]{Bene6}
Robert~L. Benedetto.
\newblock Wandering domains in non-{A}rchimedean polynomial dynamics.
\newblock {\em Bull. London Math. Soc.}, 38(6):937--950, 2006.

\bibitem[Ben19]{Bene}
Robert~L. Benedetto.
\newblock {\em Dynamics in one non-archimedean variable}, volume 198 of {\em
  Graduate Studies in Mathematics}.
\newblock American Mathematical Society, Providence, RI, 2019.

\bibitem[Ben22]{BeneSurvey}
Robert~L. Benedetto.
\newblock A survey of non-{A}rchimedean dynamics.
\newblock {\em Notices Amer. Math. Soc.}, 69(5):715--723, 2022.

\bibitem[Ber90]{Berk}
Vladimir~G. Berkovich.
\newblock {\em Spectral theory and analytic geometry over non-{A}rchimedean
  fields}, volume~33 of {\em Mathematical Surveys and Monographs}.
\newblock American Mathematical Society, Providence, RI, 1990.

\bibitem[B{\'{e}}z01]{Bez1}
Jean-Paul B{\'{e}}zivin.
\newblock Sur les points p\'{e}riodiques des applications rationnelles en
  dynamique ultram\'{e}trique.
\newblock {\em Acta Arith.}, 100(1):63--74, 2001.

\bibitem[B{\'{e}}z04]{Bez2}
Jean-Paul B{\'{e}}zivin.
\newblock Sur la compacit\'{e} des ensembles de {J}ulia des polyn\^{o}mes
  {$p$}-adiques.
\newblock {\em Math. Z.}, 246(1-2):273--289, 2004.

\bibitem[BR06]{BR1}
Matthew Baker and Robert Rumely.
\newblock Equidistribution of small points, rational dynamics, and potential
  theory.
\newblock {\em Ann. Inst. Fourier (Grenoble)}, 56(3):625--688, 2006.

\bibitem[BR10]{BR2}
Matthew Baker and Robert Rumely.
\newblock {\em Potential theory and dynamics on the {B}erkovich projective
  line}, volume 159 of {\em Mathematical Surveys and Monographs}.
\newblock American Mathematical Society, Providence, RI, 2010.

\bibitem[CL06]{CL}
Antoine Chambert-Loir.
\newblock Mesures et \'{e}quidistribution sur les espaces de {B}erkovich.
\newblock {\em J. Reine Angew. Math.}, 595:215--235, 2006.

\bibitem[DF14]{DeF1}
Laura DeMarco and Xander Faber.
\newblock Degenerations of complex dynamical systems.
\newblock {\em Forum Math. Sigma}, 2:Paper No. e6, 36, 2014.

\bibitem[DF16]{DeF2}
Laura DeMarco and Xander Faber.
\newblock Degenerations of complex dynamical systems {II}: analytic and
  algebraic stability.
\newblock {\em Math. Ann.}, 365(3-4):1669--1699, 2016.
\newblock With an appendix by Jan Kiwi.

\bibitem[DF17]{DuF}
R.~Dujardin and C.~Favre.
\newblock The dynamical {M}anin-{M}umford problem for plane polynomial
  automorphisms.
\newblock {\em J. Eur. Math. Soc. (JEMS)}, 19(11):3421--3465, 2017.

\bibitem[Dre03]{Drem}
V.~A. Dremov.
\newblock On a {$p$}-adic {J}ulia set.
\newblock {\em Uspekhi Mat. Nauk}, 58(6(354)):151--152, 2003.

\bibitem[Fab13a]{Faber13.1}
Xander Faber.
\newblock Topology and geometry of the {B}erkovich ramification locus for
  rational functions, {I}.
\newblock {\em Manuscripta Math.}, 142(3-4):439--474, 2013.

\bibitem[Fab13b]{Faber13.2}
Xander Faber.
\newblock Topology and geometry of the {B}erkovich ramification locus for
  rational functions, {II}.
\newblock {\em Math. Ann.}, 356(3):819--844, 2013.

\bibitem[Fav20]{Fav20}
Charles Favre.
\newblock Degeneration of endomorphisms of the complex projective space in the
  hybrid space.
\newblock {\em J. Inst. Math. Jussieu}, 19(4):1141--1183, 2020.

\bibitem[FG18]{FG}
Charles Favre and Thomas Gauthier.
\newblock Classification of special curves in the space of cubic polynomials.
\newblock {\em Int. Math. Res. Not. IMRN}, (2):362--411, 2018.

\bibitem[FJ04]{FJ04}
Charles Favre and Mattias Jonsson.
\newblock {\em The valuative tree}, volume 1853 of {\em Lecture Notes in
  Mathematics}.
\newblock Springer-Verlag, Berlin, 2004.

\bibitem[FJ11]{FJ11}
Charles Favre and Mattias Jonsson.
\newblock Dynamical compactifications of {{\(\mathbb{C}^2\)}}.
\newblock {\em Ann. of Math. (2)}, 173(1):211--249, 2011.

\bibitem[FRL04]{FRL1}
Charles Favre and Juan Rivera-Letelier.
\newblock Th\'{e}or\`eme d'\'{e}quidistribution de {B}rolin en dynamique
  {$p$}-adique.
\newblock {\em C. R. Math. Acad. Sci. Paris}, 339(4):271--276, 2004.

\bibitem[FRL06]{FRL2}
Charles Favre and Juan Rivera-Letelier.
\newblock \'{E}quidistribution quantitative des points de petite hauteur sur la
  droite projective.
\newblock {\em Math. Ann.}, 335(2):311--361, 2006.

\bibitem[FRL07]{FRL3}
Charles Favre and Juan Rivera-Letelier.
\newblock Corrigendum to: ``{Q}uantitative uniform distribution of points of
  small height on the projective line'' ({F}rench) [{M}ath. {A}nn. {\bf 335}
  (2006), no. 2, 311--361; mr2221116].
\newblock {\em Math. Ann.}, 339(4):799--801, 2007.

\bibitem[Har77]{Hart}
Robin Hartshorne.
\newblock {\em Algebraic Geometry}.
\newblock Springer, 1977.

\bibitem[Hsi96]{Hsia1}
Liang-Chung Hsia.
\newblock A weak {N}\'{e}ron model with applications to {$p$}-adic dynamical
  systems.
\newblock {\em Compositio Math.}, 100(3):277--304, 1996.

\bibitem[Hsi00]{Hsia2}
Liang-Chung Hsia.
\newblock Closure of periodic points over a non-{A}rchimedean field.
\newblock {\em J. London Math. Soc. (2)}, 62(3):685--700, 2000.

\bibitem[Kiw06]{Kiwi1}
Jan Kiwi.
\newblock Puiseux series polynomial dynamics and iteration of complex cubic
  polynomials.
\newblock {\em Ann. Inst. Fourier (Grenoble)}, 56(5):1337--1404, 2006.

\bibitem[Kiw14]{Kiwi2}
Jan Kiwi.
\newblock Puiseux series dynamics of quadratic rational maps.
\newblock {\em Israel J. Math.}, 201(2):631--700, 2014.

\bibitem[Kiw15]{Kiwi3}
Jan Kiwi.
\newblock Rescaling limits of complex rational maps.
\newblock {\em Duke Math. J.}, 164(7):1437--1470, 2015.

\bibitem[RL03a]{RL2}
Juan Rivera-Letelier.
\newblock Dynamique des fonctions rationnelles sur des corps locaux.
\newblock Number 287, pages xv, 147--230. 2003.
\newblock Geometric methods in dynamics. II.

\bibitem[RL03b]{RL1}
Juan Rivera-Letelier.
\newblock Espace hyperbolique {$p$}-adique et dynamique des fonctions
  rationnelles.
\newblock {\em Compositio Math.}, 138(2):199--231, 2003.

\bibitem[RL04]{RL3}
Juan Rivera-Letelier.
\newblock Sur la structure des ensembles de fatou p-adiques, 2004.

\bibitem[RL05]{RL4}
Juan Rivera-Letelier.
\newblock Points p\'{e}riodiques des fonctions rationnelles dans l'espace
  hyperbolique {$p$}-adique.
\newblock {\em Comment. Math. Helv.}, 80(3):593--629, 2005.

\bibitem[Thu05]{Thuillier}
Amaury Thuillier.
\newblock {\em {Th{\'e}orie du potentiel sur les courbes en g{\'e}om{\'e}trie
  analytique non archim{\'e}dienne. Applications {\`a} la th{\'e}orie
  d'Arakelov}}.
\newblock Theses, {Universit{\'e} Rennes 1}, October 2005.
\newblock Laurent Moret-Bailly (Pr{\'e}sident) Jean-Beno{\^\i}t Bost
  (Rapporteur) Robert Rumely (Rapporteur) Antoine Chambert-Loir (Directeur de
  th{\`e}se) Antoine Ducros (Examinateur) Charles Favre (Examinateur).

\end{thebibliography}

\end{document}